\documentclass[11pt]{amsart}

\usepackage[margin=3cm]{geometry}

\usepackage[dvipsnames]{xcolor}

\usepackage{comment}

\usepackage[numbers]{natbib}
\usepackage{amssymb,amsmath,amsthm,amsfonts,graphicx,bm,mathtools,tikz,algorithm,algpseudocode,enumitem}
\usetikzlibrary{arrows.meta}
\usetikzlibrary{calc}
\usepackage{sansmath}
\usepackage{pgfplots}
\usepackage[colorlinks=true,allcolors=Blue]{hyperref}
\pgfplotsset{compat=1.18}
\usepgfplotslibrary{groupplots}
\definecolor{sbiccol}{HTML}{D55E00}
\definecolor{biccol}{HTML}{0072B2}
\definecolor{ngonecol}{HTML}{000000}
\definecolor{ngtwocol}{HTML}{E69F00}
\definecolor{gcvcol}{HTML}{009E73}

\newtheorem{defn0}{Definition}[section]
\newtheorem{prop0}[defn0]{Proposition}
\newtheorem{thm0}[defn0]{Theorem}
\newtheorem{lemma0}[defn0]{Lemma}
\newtheorem{corollary0}[defn0]{Corollary}
\newtheorem{example0}[defn0]{Example}
\newtheorem{remark0}[defn0]{Remark}
\newtheorem{conjecture0}[defn0]{Conjecture}
\newtheorem{code0}[defn0]{Code}
\newtheorem{fact0}[defn0]{Fact}

\newenvironment{definition}{\medskip \begin{defn0}}{\end{defn0}}

\newenvironment{theorem}{\medskip \begin{thm0}}{\end{thm0}}
\newenvironment{lemma}{\medskip \begin{lemma0}}{\end{lemma0}}
\newenvironment{corollary}{\medskip \begin{corollary0}}{\end{corollary0}}
\newenvironment{example}{\medskip \begin{example0}\rm}{\end{example0}}
\newenvironment{remark}{ \medskip\begin{remark0}\rm}{\end{remark0}}

\newenvironment{fact}{\medskip\begin{fact0}}{\end{fact0}}

\newcommand{\jac}{\operatorname{Jac}}
\newcommand{\tr}{\operatorname{tr}}

\newcommand{\rlct}{\operatorname{RLCT}}
\newcommand{\rank}{\operatorname{rank}}
\DeclareMathOperator*{\argmax}{argmax}

\newcommand{\bl}[1]{{\mathbf #1}}
\newcommand{\bs}[1]{\boldsymbol #1}
\newcommand{\mb}[1]{\mathbb #1}
\newcommand{\mc}[1]{\mathcal #1}
\newcommand{\mf}[1]{\mathfrak #1}
\DeclareMathOperator{\diag}{diag}
\newcommand{\R}{\mathbb{R}}

\newcommand{\Z}{\mathbb{Z}}
\newcommand{\vol}{\operatorname{vol}}
\newcommand{\convp}{\xrightarrow{\;\;p\;\;}}

\newlist{conditions}{enumerate}{1}
\setlist[conditions]{
    label=(C\arabic*),
    ref=C\arabic*,
    leftmargin=*
}

\numberwithin{equation}{section}

\title[Model Selection in Probabilistic Principal Component Analysis]{Asymptotics for Model Selection in Probabilistic Principal Component Analysis}
\author[M. Drton]{Mathias Drton}
\address{Department of Mathematics, TUM School of Computation, Information and Technology, Technical University of Munich and Munich Center for Machine Learning, Germany}
\email{mathias.drton@tum.de}
\author[A. McCormack]{Andrew McCormack}
\address{Department of Mathematical and Statistical Sciences, University of Alberta, Canada}
\email{mccorma2@ualberta.ca}
\author[D. Windisch]{Daniel Windisch}
\address{Department of Computer Science, KU Leuven, Belgium}
\email{daniel.windisch.math@gmail.com}

\begin{document}

\begin{abstract}
The probabilistic formulation of principal component analysis facilitates statistically grounded solutions to the problem of selecting the number of principal components. However, developing tractable model selection methods is complicated by the fact that the probabilistic principal component analysis (PPCA) model exhibits non-standard large sample asymptotics at model singularities, where the Fisher information matrix does not have full rank. In this work, tools from singular learning theory are used to provide a complete description of the marginal likelihood asymptotics for PPCA. The singular Bayesian information criterion (sBIC) along with our asymptotic results provides an effective procedure for selecting the number of principal components. In particular, the sBIC corrects for the overpenalization of the standard dimension-based BIC, which leads to too few principal components being selected, while also selecting the smallest true model with probability converging to one. Our framework extends beyond PPCA, where we also provide expressions for the sBIC that can be used to select between PPCA and factor analysis models. In both simulations and on real data the effectiveness of the proposed sBIC methodology is demonstrated. 
\end{abstract}


\maketitle

\section{Introduction}
\label{sec:Introduction}

 Probabilistic principal component analysis (PPCA) \citep{AndersonRubinFactorAnalysis}, popularized in \cite{PPCATippingBishop}, also at times referred to as a spiked covariance model \citep{JohnstoneSpiked}, is a model-based framework for performing principal component analysis. Under simplifying mean-zero and Gaussianity assumptions, the PPCA model with $r$ principal components (PCs) posits that each $p$-dimensional observation $\bl{X}_i$ is a noisy linear function of a latent $r$-dimensional random vector $\bl{Z}_i$,
 \begin{align}
 \label{eqn:PPCALatent}
     \bl{X}_i = \bl{W}\bl{Z}_i + \bs{\varepsilon}_i, \qquad \bl{Z}_i \sim \mc{N}_r(\bl{0},\bl{I}_r), \quad \bs{\varepsilon}_i \sim \mc{N}_p(\bl{0},\sigma^2\bl{I}_p),
 \end{align}
 for some matrix $\bl{W} \in \mb{R}^{p \times r}$ and $\sigma^2 > 0$, where $\bl{Z}_i$ and the noise vector $\bs{\varepsilon}_i$ are independent. Marginalizing over the latent vector, the data form an i.i.d.~sample of $p$-dimensional random vectors $\bl{X}_1,\dots,\bl{X}_n$ with
 \begin{align}
 \label{eqn:PPCAModel}
     \bl{X}_i \sim \mc{N}_p(\bl{0}, \bl{W}\bl{W}^\intercal + \sigma^2 \bl{I}_p).
 \end{align}
 The connection between PCA and model \eqref{eqn:PPCAModel} is that the eigenvalues of $\bl{W}\bl{W}^\intercal + \sigma^2 \bl{I}_p$ are $(\tau_1^2+\sigma^2,\ldots,\tau_r^2 + \sigma^2,\sigma^2,\ldots,\sigma^2)$, where the $\tau_i$ are the singular values of $\bl{W}$. Hence, this covariance structure implies that there are $r$ orthogonal directions, given by the left-singular vectors of $\bl{W}$, along which the $\bl{X}_i$ are highly variable. The eigenstructure of sample covariance matrices generated under the PPCA model is an important object of study in random matrix theory \citep{RMTBaiSilverstein}. In practice, it is of interest to estimate the directions of high variability as well as to determine the number of such directions, $r$.

The selection of the number of PCs is a problem that has garnered significant attention \citep[Chap.~6]{JolliffePCABook}. Often practitioners resort to \textit{ad hoc} methods involving the inspection of sample eigenvalue plots, possibly due to the expediency sought when using PCA for data pre-processing. Parallel analysis is a related technique that has been developed to provide a more rigorous assessment of eigenvalue plots 
\citep{Dobriban2020PA,DobribanOwen2019PA,Horn1965PA}. Minka \citep{minkaPPCA2000automatic} proposed using the PPCA model as a scaffold for selecting the number of PCs via Bayesian model selection; further Bayesian treatments can be found in \citep{bishop1998bayesianPCA,BouveyronBayesianPCA,sobczyk2017bayesianPPCA,BayesianPCAQuinn}. While a full Bayesian analysis using Monte Carlo simulation is certainly viable \citep{HoffFullBayesSVD}, it has the downside of 
being computationally involved for the practitioner who seeks a swift answer to their dimensionality reduction problem. To circumvent the need for Monte Carlo methods, Minka \citep{minkaPPCA2000automatic} advanced a Laplace approximation to the marginal likelihood of the PPCA model. Such Laplace approximations are also the basis for Schwarz's Bayesian Information Criterion (BIC) \citep{SchwarzBIC}. A subtle point is that the validity of a Laplace approximation requires a statistical model to be identifiably parameterized with a positive definite Fisher information matrix. However, as discussed in Section \ref{sec:PPCAGeometry}, the PPCA model possesses singularities that cause the requisite regularity conditions for a Laplace approximation to be violated, a fact that has not been accounted for in previous works. Despite this, using techniques from random matrix theory, it was shown in \citep{BaiAICBICConsistency} that as long as the gap between the $r$th largest and $p-r$ smallest eigenvalues is sufficiently large, both the Akaike information criterion (AIC) \citep{akaikeAIC} and BIC consistently estimate the true number of PCs in high-dimensional settings with a growing dimension $p$ and a bounded number of PCs.

The central goal of this work is to obtain asymptotic approximations to marginal likelihood integrals in the PPCA model for fixed-$p$, large-$n$ asymptotic regimes. Procedures for dealing with models possessing singularities have been introduced in Watanabe's seminal work on singular learning theory \citep{WatanabeSingularLearningBook,WatanabeBayesianStatisticsBook}. At a high level, singular learning theory 
analyzes a Bayesian marginal likelihood integral by localizing the domain of integration to a neighborhood of the set of parameters that can represent the true data-generating distribution, decomposing this neighborhood into smaller local neighborhoods of possibly different geometric types, and performing a non-trivial change of variables to each local neighborhood. The overall process is termed a resolution of singularities by algebraic geometers \citep[Chap.~2]{WatanabeSingularLearningBook}. 
The resolution reveals that even for models that lack identifiability and have a Fisher information matrix that drops rank, marginal likelihood integrals may be brought into a standard form
that reveals that the asymptotic scaling of the overall integral is entirely determined by two parameters, the \textit{learning coefficient} $\lambda$, and its \textit{multiplicity}, $\mf{m}$. In the case of sufficiently regular models, the BIC captures precisely this asymptotic information, with a learning coefficient that equals one half of the dimension of the parameter space and has multiplicity one. The pair $(\lambda,\mf{m})$ has been fully or partially determined for a growing number of pertinent statistical models, including reduced-rank regression \citep{AoyagiReducedRankRegression}, neural networks \citep{AoyagiNeuralNet,AoyagiMultiLayerLinear2024}, mixture models \citep{RousseauMixture,MixtureModelyamazaki2003singularities,AoyagiVandermondeModelSelection2019}, latent tree models \citep{DrtonZwiernikatentTree}, and factor analysis models \citep{drton2025Factor}. We note that at present there are only a few models, namely, univariate Gaussian mixture models with a fixed variance, the reduced-rank regression model, and the factor analysis model, for which the learning coefficients are completely characterized. The present work provides a complete characterization of the learning coefficients and their multiplicities for PPCA. Furthermore, we introduce, and determine the learning coefficients for, a more general class of partitioned noise PPCA models that includes both the PPCA and factor analysis models as special cases. Using the derived learning coefficients, we develop principled model selection procedures, not just for determining the number of PCs in PPCA, but also for selecting \textit{between} PPCA and factor analysis models.

Mere mathematical derivation of the values of $(\lambda,\mf{m})$ is generally not enough for statistical model selection, as the values may depend on characteristics of the unknown data-generating distribution. Thus, they need to be processed in a practically feasible model selection procedure. The singular Bayesian information criterion (sBIC) \citep{sBICDrtonPlummer} was proposed as a way to utilize the learning coefficients for model selection while mitigating the issue that their values may depend on the unknown true model. Adopting this framework, we use the PPCA learning coefficients for model selection. In PCA and related settings, the BIC is well known to penalize complex models heavily, often resulting in overly simple models \citep{BouveyronBayesianPCA,GICHung}. The sBIC imposes a smaller penalty and is more sensitive in detecting signals in the eigenvalues. We demonstrate empirically that the sBIC convincingly outperforms the BIC in the PCA context. Theory is also provided to support this; the sBIC is shown to be a consistent model selection procedure under a wider collection of eigenstructures than the BIC in high-dimensional settings.

An outline of this paper is as follows:
Section \ref{sec:PPCAGeometry} describes the geometry of the PPCA model and highlights the structure of the singularities that are present. A concise review of the sBIC is provided in Section \ref{sec:BackgroundSLTandsBIC}. Section \ref{sec:RLCTTheory} outlines techniques for obtaining the learning coefficients and multiplicities. The core theoretical results in this paper appear in this section, where the learning coefficients and multiplicities for PPCA are derived. Our results for PPCA are expanded upon in Section \ref{sec:UnifyPPCAandFA}, where the notion of a partitioned noise PPCA model is introduced. This class of models provides a framework for performing model selection between PPCA and factor analysis models. Following \citep{BaiAICBICConsistency}, in Section \ref{sec:HighDimensional} we provide theoretical results describing when the BIC and sBIC are consistent as $p \to \infty$ and $p/n\to c \in (0,1)$ in PPCA. Section \ref{sec:Simulations} provides evidence from simulations for the validity of our asymptotic theory, and contrasts the performance of the sBIC with other model selection techniques for PPCA. This is complemented by two applications to real data.

\section{Geometry of the PPCA Model}
\label{sec:PPCAGeometry}

As a centered normal distribution is determined by its covariance matrix, we identify each model with its set of covariance matrices.
\begin{definition}\label{def:PPCAModel}
The PPCA model from \eqref{eqn:PPCAModel} with a candidate number of $k$ principal components (PCs) is the covariance model
\begin{align*}
    \mc{M}_k \coloneqq \{ \bl{W}\bl{W}^\intercal + \sigma^2 \bl{I}_p: \bl{W} \in \mb{R}^{p \times k}, \;\sigma^2 > 0\}.
\end{align*}
\end{definition}
Varying $k$, the PPCA models are nested by inclusion as
\begin{align}
   \label{eqn:NestednessofPPCA} \{\sigma^2\bl{I}_p:\sigma^2 > 0\} = \mc{M}_0 \subsetneq \mc{M}_1 \subsetneq \cdots \subsetneq \mc{M}_{p-2} \subsetneq \mc{M}_{p-1} = \mc{S}_{++}^p,
\end{align}
where $\mc{S}_{++}^p$ is the (open) cone of positive definite matrices. It is convenient to also set $\mc{M}_{-1} \coloneqq \emptyset$, so that the set difference $\mc{M}_k \setminus \mc{M}_{k-1}$ is defined for every $0 \leq k \leq p-1$.
Given a sample from $\mc{N}_p(\bl{0},\bs{\Sigma})$,
the model selection problem for PPCA is finding the smallest value of $k$ such that $\bs{\Sigma} \in \mc{M}_k$. 

In our context, singularities of a statistical model are defined with reference to a given parameterization. A model is singular if the Fisher information matrix in the coordinates of the parameterization does not always have full rank. We consider the PPCA model as being parameterized by the map $\varphi_k:\mb{R}^{p \times k}\times\mb{R}_{> 0} \rightarrow \mc{S}_{++}^p$ with $\varphi_k(\bl{W},\sigma^2) \coloneqq \bl{W}\bl{W}^\intercal + \sigma^2 \bl{I}_{p}$. The Fisher information matrix does not have full rank in this parameterization because the parameterization is not identifiable:  For $\bl{U} \in \text{O}(k)$, the group of orthogonal $k\times k$ matrices, $\bl{W}\bl{U}(\bl{W}\bl{U})^\intercal = \bl{W}\bl{W}^\intercal$. To deal with this lack of identifiability, the $p \times k$ matrices $\bl{W}$ could alternatively be parameterized by the equivalence classes of matrices $\mb{R}^{p \times k}/\text{O}(k)$, where two matrices $\bl{W}^{(1)}$ and $\bl{W}^{(2)}$ are equivalent in $\mb{R}^{p \times k}/\text{O}(k)$ if and only if there exists a $\bl{U} \in \text{O}(k)$ with $\bl{W}^{(1)} \bl{U} = \bl{W}^{(2)}$. However, even after this modification, important singularities persist because the quotient space $\mb{R}^{p \times k}/\text{O}(k)$ does not have the structure of a smooth manifold. Indeed, singularities can be attributed to matrices $\bl{W}$ that do not have full rank. Only when restricting to $\mb{R}^{p \times k}_*$, the set of matrices in $\mb{R}^{p \times k}$ that have full column rank $k$, we obtain a quotient space $\mb{R}^{p \times k}_* / \text{O}(k)$ that is a smooth manifold; see, e.g., \citep{FixedRankPDManifoldGeomvandereycken}. 

A more direct way to assess the presence of singularities under the $\varphi_k$ parameterization is to consider the rank of the Jacobian matrix $J_{\varphi_k}(\bl{W},\sigma^2)$ of $\varphi_k$ at $(\bl{W},\sigma^2) \in \mb{R}^{p \times k} \times \mb{R}_{> 0}$. The image of $J_{\varphi_k}(\bl{W},\sigma^2)$ is given by the following tangent space at $\bs{\Sigma} = \bl{W}\bl{W}^\intercal + \sigma^2 \bl{I}_p$:
\begin{align*}
  \text{span}\bigg(\{\bl{W}\bl{V}^\intercal + \bl{V}\bl{W}^\intercal : \bl{V} \in \mb{R}^{p \times k} \} \cup \{\bl{I}_p\} \bigg),
\end{align*}
which has dimension $r(p-r) + r(r+1)/2 + 1$ when $\text{rank}(\bl{W}) = r$ and $0 \leq k \leq p-1$, as shown in Lemma \ref{lem:ModelDimension} in the Appendix. 
Hence, $J_{\varphi_k}(\bl{W},\sigma^2)$ fails to achieve its maximum attainable column rank at a pair $(\bl{W},\sigma^2)$ precisely when $\text{rank}(\bl{W}) < k$.  Notice that the set of matrices $\bs{\Sigma} \in \mc{M}_k$ associated to pairs $(\bl{W},\sigma^2)$ with $\text{rank}(\bl{W}) < k$ is precisely equal to $\mc{M}_{k-1}$. Therefore, \textit{the singular points in $\mc{M}_k$ are precisely the covariance matrices that also lie in the smaller PPCA model $\mc{M}_{k-1}$}.
The behavior of $\mc{M}_k$ at its singularities is thus important for model selection.

\begin{example}
Consider the case $p = 2$ and $k = 1$, where $\bl{W}=(w_{11},w_{21})^\intercal$ is a vector. The parameterization map and its Jacobian matrix (accounting for symmetry) are 
\begin{align*}
    \varphi_1(\bl{W},\sigma^2) = \begin{bmatrix}
        w_{11}^2 + \sigma^2 & w_{11}w_{21}
        \\
        w_{11}w_{21} & w_{21}^2 + \sigma^2
    \end{bmatrix}, \quad J_{\varphi_1}(\bl{W},\sigma^2) = \begin{bmatrix}
        2w_{11} & 0 & w_{21}
        \\
        0 & 2w_{21} & w_{11}
        \\
        1 & 1 & 0
    \end{bmatrix}.
\end{align*}
It is seen that $\text{rank}(J_{\varphi_1}(\bl{W},\sigma^2)) = 3$ if and only if $\det(J_{\varphi_1}(\bl{W},\sigma^2))=-2(w_{11}^2+w_{21}^2)\neq 0$ if and only if $\bl{W} \neq \bl{0}$. The Fisher information matrix $\mc{I}(w_{11},w_{21},\sigma^2)$ will thus have rank $3$ if $\bl{W} \neq \bl{0}$ and rank $1$ otherwise. When $p = 2$, the set of covariance matrices $\mc{M}_1$ equals $\mc{S}^p_{++}$, which is a smooth manifold. This illustrates that the singularities at $\bl{W} = \bl{0}$ are not intrinsic to the model $\mc{M}_1$ itself, but rather arise from the chosen standard parameterization.        
\end{example}

By removing the set of singular points $\mc{M}_{k-1}$ from $\mc{M}_k$, it is proven in Lemma \ref{lem:NonSingularPointsareManifold} that the set $\mc{M}_k \setminus \mc{M}_{k-1}$ \textit{can} be given the structure of a smooth manifold that is compatible with the $\varphi_k$ parameterization. The differential geometry of spaces of positive definite matrices with equality constrained eigenvalues is discussed at length in \citep{GroisserStratifiedPD}. In particular, these spaces are stratified spaces, roughly meaning that they consist of a collection of smooth manifolds that are glued together in a systematic fashion. Recalling from the introduction that a covariance matrix $\bs{\Sigma}$ is in $\mc{M}_k$ if and only if its eigenvalues have the form $\gamma_1 \geq \cdots \geq \gamma_k \geq \gamma_{k+1} = \cdots = \gamma_p$, we see that the matrices that are not singularities in $\mc{M}_k$ have eigenstructures of the form $\gamma_1 \geq \cdots \geq \gamma_k > \gamma_{k+1} = \cdots = \gamma_p$, where the $k$th largest eigenvalue $\gamma_{k}$ must be larger than $\gamma_{k+1}$. Conversely, a matrix $\bs{\Sigma} \in \mc{M}_k$ is a singular point if and only if $\gamma_k = \gamma_{k+1}$.

To conclude our discussion of the geometry of $\mc{M}_k$, we mention that the PPCA models have a rich group transformation structure. As membership in $\mc{M}_k\setminus \mc{M}_{k-1}$ is characterized by the eigenstructure of a positive definite matrix, it is invariant under orthogonal transformations, because eigenvalues are invariant under orthogonal transformations. Specifically, $\bs{\Sigma} \in \mc{M}_k \setminus \mc{M}_{k-1}$ implies $\bl{U}\bs{\Sigma}\bl{U}^\intercal \in  \mc{M}_k \setminus \mc{M}_{k-1}$ for all orthogonal matrices $\bl{U}\in\text{O}(p)$. This group structure should be contrasted with symmetries of the related factor analysis model \citep{AlgFactorAnalysisDrton}, which consists of positive definite matrices of the form 
\begin{align*}
  \mc{F}_k \coloneqq  \{ \bl{W}\bl{W}^\intercal + \bl{D}: \bl{W} \in \mb{R}^{p \times k}, \;\; \bl{D} = \diag(\psi_1,\ldots,\psi_p),\; \psi_i > 0\}.
\end{align*}
It is no longer the case that $\bl{U}\bs{\Sigma}\bl{U}^\intercal$ is in $\mc{F}_k\setminus \mc{F}_{k-1}$ whenever $\bs{\Sigma} \in \mc{F}_k \setminus \mc{F}_{k-1}$ and $\bl{U}\in\text{O}(p)$. The factor analysis model instead has symmetries consisting of diagonal matrices, as $\bl{A}\bs{\Sigma}\bl{A}^\intercal \in \mc{F}_k \setminus \mc{F}_{k-1}$ whenever $\bs{\Sigma} \in \mc{F}_k \setminus \mc{F}_{k-1}$ and $\bl{A}$ is nonsingular and diagonal. The PPCA models do not possess these diagonal symmetries, but are scale invariant as $\bs{\Sigma} \in \mc{M}_k \setminus \mc{M}_{k-1}$ implies $a^2 \bs{\Sigma} \in \mc{M}_k \setminus \mc{M}_{k-1}$ for $a \neq 0$. The orthogonal and scale invariance of the PPCA models will play an important role when we determine the PPCA learning coefficients in Section \ref{sec:RLCTTheory}.

\section{Singular Learning Theory and the sBIC}
\label{sec:BackgroundSLTandsBIC}

Bayesian model selection for PPCA begins with prior distributions being placed on the parameters of each of the models $\mc{M}_k$, $k = 0,\ldots,p-1$. We assume that for each $k$, the prior distribution has a Lebesgue density $\phi_k(\bl{W},\sigma^2)$ that is \emph{smooth} and \emph{everywhere positive}, but otherwise arbitrary. Given data $\bl{X} = (\bl{X}_1,\ldots,\bl{X}_n)$, where the $\bl{X}_i$ are i.i.d.~$\mc{N}_p(\bl{0},\bs{\Sigma}_0)$, the fit of a model $\mc{M}_k$ is assessed by its marginal likelihood
\begin{align}
\label{eqn:MarginalLikelihoodforPPCA}
    p(\bl{X}|\mc{M}_k) = \int_{\mb{R}^{p \times k}}\int_{0}^\infty \prod_{i = 1}^n p(\bl{X}_i|\bl{W},\sigma^2,\mc{M}_k) \phi_k(\bl{W},\sigma^2)\;\mathrm{d}\sigma^2\,\mathrm{d}\bl{W},
\end{align}
where $p(\bl{X}_i|\bl{W},\sigma^2,\mc{M}_k)$ is the density of $\mc{N}_p(\bl{0},\bl{W}\bl{W}^\intercal + \sigma^2 \bl{I}_{p})$ at $\bl{X}_i$.
In the sequel, we write $r$ for the minimal number of principal components needed to represent the true covariance matrix $\bs{\Sigma}_0$ in a PPCA model.  In other words, $r=r(\bs{\Sigma}_0)$ is the unique number in $\{0,1,\ldots,p-1\}$ with
\begin{equation}
\label{eq:rSigma0}
\bs{\Sigma}_0 \in \mc{M}_r \setminus \mc{M}_{r-1}.
\end{equation}

Hallmark results of singular learning theory describe the asymptotic behavior of the marginal likelihood (Corollary 6.1 of Watanabe \citep{WatanabeSingularLearningBook}). Applied to our context, we may conclude that when $r \leq k$, such that the posited model $\mc{M}_k$ contains the true underlying $\bs{\Sigma}_0$, the log-marginal likelihood satisfies the following asymptotics as $n \rightarrow \infty$:
\begin{equation}
\begin{aligned}
\label{eqn:LogMargLikeAsymptotics}
    \log\big( p(\bl{X}|\mc{M}_k)\big) = &\log\big(p(\bl{X}|\hat{\bl{W}},\hat{\sigma}^2,\mc{M}_k)\big)\\
    &- \lambda_k(\bs{\Sigma}_0)\log(n) + (\mf{m}_k(\bs{\Sigma}_0) - 1)\log(\log(n)) + O_p(1).
\end{aligned}
\end{equation}
Here $(\hat{\bl{W}},\hat{\sigma}^2) \in \mb{R}^{p \times k} \times \mb{R}_{> 0}$ are the maximum likelihood estimates of the parameters under model $\mc{M}_k$. The first element of the pair $(\lambda_k(\bs{\Sigma}_0),\mf{m}_k(\bs{\Sigma}_0))$ is referred to as the \textit{learning coefficient} and the second element is its \textit{multiplicity}. 
Our later results show that this pair only depends on $\bs{\Sigma}_0$ through the integer $r$ defined in \eqref{eq:rSigma0}, and we denote it by $(\lambda_{kr},\mf{m}_{kr})$.

For models that are regularly parameterized, the learning coefficient equals half the dimension of the model and the multiplicity is one, which is the fact that yields the BIC \citep{SchwarzBIC}. As pointed out in \citep{sBICDrtonPlummer}, when $(\lambda_k(\bs{\Sigma}_0),\mf{m}_k(\bs{\Sigma}_0))$ is not constant and depends on the unknown $\bs{\Sigma}_0$ the expansion \eqref{eqn:LogMargLikeAsymptotics} cannot directly be used to form a model selection criterion. The sBIC was proposed as a methodology that uses knowledge of the values of $(\lambda_{kr},\mf{m}_{kr})$ to construct a model selection criterion that is both consistent and equal to the true marginal likelihood up to an error term of order $O_p(1)$ \citep{sBICDrtonPlummer}.

Assuming equal prior probability $p(\mc{M}_k) = \tfrac{1}{p}$ for each model, the sBIC is constructed iteratively, starting at the smallest model $\mc{M}_0$ and initializing a marginal likelihood proxy as
\begin{align}
\label{eqn:sBICL0Initialize}
    L_{0} = p(\bl{X}|\hat{\bl{W}},\hat{\sigma}^2,\mc{M}_0) \frac{\log(n)^{\mf{m}_{00}-1}}{n^{\lambda_{00}}}.
\end{align}
One then recursively computes further marginal likelihood proxies for $k=1,\dots,p-1$ as
\begin{align}
\label{eqn:sBICLi}
    L_k & = \tfrac{1}{2}\big(-b_k + \sqrt{b_k^2 + 4c_k}\big),\\
    \intertext{where}
\label{eqn:sBICbi}
b_k & = -p(\bl{X}|\hat{\bl{W}},\hat{\sigma}^2,\mc{M}_k) \frac{\log(n)^{\mf{m}_{kk}-1}}{n^{\lambda_{kk}}} + \sum_{j < k} L_j,
\\
\label{eqn:sBICci}
c_k & = \sum_{j < k} p(\bl{X}|\hat{\bl{W}},\hat{\sigma}^2,\mc{M}_k) \frac{\log(n)^{\mf{m}_{kj}-1}}{n^{\lambda_{kj}}} L_j.
\end{align}
The sBIC for model $\mc{M}_k$ is defined as $\log(L_k)$; for numerical stability, in practice the calculation should be carried out on the log scale. The resulting model selection criterion amounts to choosing the model with the largest sBIC. 

The calculation of sBIC for PPCA is straightforward due to the simple form of the maximum likelihood estimators (MLE) in $\mc{M}_k$ \citep{PPCATippingBishop}. 
\begin{lemma}
\label{PPCA:MLE}
Assume the observations to have zero mean.  Form the sample covariance matrix $\tfrac{1}{n}\bl{X}\bl{X}^\intercal$ and its  eigendecomposition $\tfrac{1}{n}\bl{X}\bl{X}^\intercal = \bl{U}\bl{T}\bl{U}^\intercal$ with $\bl{T}=\diag(t_1,\dots,t_p)$, where $t_1 \geq \cdots \geq t_p$, and $\bl{U}\in \mathrm{O}(p)$. Then the MLE of the covariance matrix under the PPCA model with $k$ PCs is equal to 
\begin{equation*}
\hat{\bs{\Sigma}} = \bl{U}\bl{T}_*\bl{U}^\intercal, \;\; \text{where} \ \bl{T}_* = \diag(t_1,\ldots,t_k,\Bar{t},\ldots,\Bar{t}) \text{ and } \Bar{t} = \textstyle\tfrac{1}{p-k}\sum_{i = k+1}^p t_i.
\end{equation*}
The corresponding maximized log-likelihood is
\begin{align*}
    \hat{\ell}_k = -\frac{n}{2}\bigg(\sum_{i = 1}^k \log(t_{i}) + (p-k) \log(\Bar{t})\bigg) -\frac{pn}{2}(\log(2\pi)+ 1).
\end{align*}
\end{lemma}

\begin{remark}
    The above formulas for sBIC also apply to the larger class of partitioned noise PPCA models introduced in Section \ref{sec:UnifyPPCAandFA}. The only distinction is that in \eqref{eqn:sBICbi}-\eqref{eqn:sBICci} the sums $\sum_{j < k}$ need to be replaced by sums that refer to a partial order among models. Specifically, one would use the partial order given by model inclusion, i.e., use an order $\prec$ on the model indices $\bl{k}$ such that $\sum_{\bl{j} \prec \bl{k}}$ becomes a sum of terms associated with the submodels $\mc{M}_{\bl{j}} \subsetneq \mc{M}_{\bl{k}}$.
\end{remark}

\section{Learning coefficients of PPCA models}
\label{sec:RLCTTheory}

We now turn to the core theoretical contribution of this paper, determining the value of $(\lambda_{k}(\bs{\Sigma}_0),\mf{m}_{k}(\bs{\Sigma}_0))$ appearing in \eqref{eqn:LogMargLikeAsymptotics}. Our results will leverage algebraic methods as developed in \citep{Lin2017Ideal, WatanabeSingularLearningBook}. The reader who is less interested in the algebraic machinery may directly focus on the following statement of our main theorem.

\begin{theorem}\label{thm:formula-learning-coeff}
Let $p \geq 1$ and $0 \leq r \leq k\leq p-1$.  Let the observations have true covariance matrix $\bs{\Sigma}_0 \in \mc{M}_r\setminus \mc{M}_{r-1}$, i.e., a PPCA covariance matrix with exactly $r$ PCs. Under any smooth and everywhere positive bounded prior density $\phi_k:\R^{p\times k} \times \R_{>0}\to[0,\infty)$, the model $\mc{M}_k$ with $k$ PCs admits the marginal likelihood expansion \eqref{eqn:LogMargLikeAsymptotics} with learning coefficient and multiplicity equal to
\begin{align}
    \label{eqn:LearningCoeffExpression}
    \lambda_k(\bs{\Sigma}_0)= \lambda_{kr} \coloneqq \frac{pk + r(p-k+1) + 2}{4} \quad \text{and} \quad \mf{m}_k(\bs{\Sigma}_0)= \mf{m}_{kr} \coloneqq 1.
\end{align}
\end{theorem}

The rest of this section reviews the needed algebraic tools, develops some preliminary lemmas, and then gives a proof of Theorem~\ref{thm:formula-learning-coeff}.

\subsection{Algebraic aspects and preliminary lemmas}

The learning coefficient-multiplicity pair corresponds to an algebraic quantity, the \textit{real log canonical threshold} (RLCT).  A general setup concerning the RLCT of an ideal of analytic functions is to let $\Omega \subseteq \R^d$ be a compact semianalytic subset and denote by $\mathcal{A}(\Omega)$ the ring of all functions $\Omega \to \R$ that are analytic at every point of $\Omega$. Consider a finitely generated ideal $\mathcal{I} = \langle f_1,\ldots,f_h \rangle \subseteq \mathcal{A}(\Omega)$ and a function $\phi: \Omega \to \R$ that is the pointwise product of an analytic map and an everywhere positive and smooth map. It can be shown that
\begin{equation}\label{equ:zeta}
    \zeta_{\mathcal{I}}(z) = \int_\Omega (f_1(\omega)^2 + \ldots + f_h(\omega)^2)^{-z/2} |\phi(\omega)| \;\mathrm{d}\omega
\end{equation}
extends to a meromorphic function $\mathbb C \to \mathbb C$ with poles that are rational numbers; cf.~\cite{WatanabeSingularLearningBook}. We denote the pair consisting of the smallest pole $\lambda$ and the multiplicity $m$ of this pole by
\[
\rlct_\Omega(\mathcal{I}; \phi) = (\lambda,m)
\]
and call it the real log canonical threshold of $\mathcal{I}$ on $\Omega$ with respect to $\phi$. 
This RLCT is independent of the choice of generators of $\mathcal{I}$; see~\cite[Prop.~6]{Lin2017Ideal}. Moreover, for every $x \in \Omega$ there exists a compact neighborhood $\Omega_x \subseteq \Omega$ of $x$ such that
\[
\rlct_{\Omega_x}(\mathcal{I};\phi) = \rlct_U(\mathcal{I};\phi)
\]
for every compact neighborhood $U \subseteq \Omega_x$ of $x$; see~\cite[Proposition 2]{Lin2017Ideal}. We define $\rlct_{x}(\mathcal{I};\phi) = \rlct_{\Omega_x}(\mathcal{I};\phi)$. It is shown in~\cite[Prop.~3]{Lin2017Ideal} that $\rlct_{\Omega}(\mathcal{I};\phi) = \min_{x \in \Omega} \rlct_{x}(\mathcal{I};\phi)$, where we order pairs by 
\[
 (\lambda,m) \leq (\lambda',m') \text{ if and only if } \lambda < \lambda' \text{ or } \lambda = \lambda' \text{ and } m \geq m'.
 \]
 
 The fiber ideal of a PPCA model that we will now describe is generated by polynomials. Consequently, all singularity types of the associated real analytic space appear in a compact subset of the parameter space $\mb{R}^{p \times k} \times \mb{R}_{> 0}$, and we can therefore extend our treatment from compact $\Omega$ to this non-compact space.

Recall that the covariance matrices of the $p$-dimensional PPCA model with $k$ PCs are parametrized by $\varphi_k(\bl{W},\sigma^2) = \bl{W}\bl{W}^\intercal + \sigma^2\bl{I}_p$. The \textit{fiber} at $\bs{\Sigma}_0 \in \mc{M}_k$ is the set of all $(\bl{W},\sigma^2)$ with $\bl{W}\bl{W}^\intercal + \sigma^2 \bl{I}_p = \bs{\Sigma}_0$, namely the inverse image of the parameterization map $\varphi_k^{-1}(\bs{\Sigma}_0)$. Corresponding to a fiber is the \textit{fiber ideal}
\[
\mathcal{I}_{p,k}(\bs{\Sigma}_0) = \langle \bl{W}\bl{W}^\intercal + \sigma^2\bl{I}_p - \bs{\Sigma}_0 \rangle,
\]
where $\langle \bl{M} \rangle$ denotes the ideal \citep[Chap.~1]{CoxLittleOSheaIdealsVarieties} generated by the entries of the matrix $\bl{M}$. The elements of $\mc{I}_{p,k}(\bs{\Sigma}_0)$ are analytic functions in the variables $(w_{11},\ldots,w_{pk},\sigma^2)$. Any function $f \in \mathcal{I}_{p,k}(\bs{\Sigma}_0)$ has the property that it vanishes on the fiber at $\bs{\Sigma}_0$: if $\bl{W}\bl{W}^\intercal + \sigma^2 \bl{I}_p = \bs{\Sigma}_0$ then $f(\bl{W},\sigma^2) = 0$. Intuitively, for large sample sizes, the marginal likelihood integral is determined by the behavior of the likelihood and prior near the fiber of $\bs{\Sigma}_0$, hence the pertinence of the fiber ideal.

The relevant function $\phi$ in the RLCT for the PPCA model will be the prior density $\phi_k$. As we focus on the case of a smooth and everywhere positive prior $\phi_k$, its specific form does not matter and we can assume $\phi_k = 1$; see~\cite[Lemma 1]{Lin2017Ideal}.

The following result is the key link between the learning coefficient, RLCT, and the fiber ideal, and is a direct consequence of~\cite[Thm.~2]{Lin2017Ideal}.

\begin{fact}\label{fact:fiber-ideal}
    Let $\bs{\Sigma}_0\in \mc{M}_k$ be a symmetric positive definite matrix lying in the PPCA model with $p$-dimensional observations and $k$ principal components. Then the following holds for the learning coefficient $\lambda_k(\bs{\Sigma}_0)$ and its multiplicity $\mf{m}_k(\bs{\Sigma}_0)$:
    \[
    (2\lambda_k(\bs{\Sigma}_0),\mf{m}_k(\bs{\Sigma}_0)) = \min_{(\bl{W}_0,\sigma_0^2) \in \varphi_k^{-1}(\bs{\Sigma}_0)} \rlct_{(\bl{W}_0,\sigma_0^2)}(\mathcal{I}_{p,k}(\bs{\Sigma}_0);1).
    \]
\end{fact}
We now detail a few further facts regarding RLCTs that will be needed to determine the learning coefficients. The following statement that we will use without further reference is a slight generalization of~\cite[Prop.~8]{Lin2017Ideal}, but the proof is exactly the same.

\begin{fact}[Chain rule \text{\cite[Prop.~8]{Lin2017Ideal}}]\label{fact:chain-rule}
    Let $\Omega \subseteq \R^d$  be open
     and $W \subseteq \Omega$ a compact semianalytic neighborhood of a point $x\in\Omega$. Let $\mathcal{I}=\langle f_1,\ldots,f_h\rangle$ be a finitely generated ideal of $\mathcal{A}(\Omega)$ and $ W = W_1 \cup \ldots \cup W_u \cup V$ a partition, where $V \subsetneqq W$ is an analytic variety and the $W_i$ are semianalytic subsets of dimension $d$ with compact closure. Let $M$ be a real analytic manifold and $\rho: M \to W$ a proper real analytic map whose restrictions $\rho^{-1}(W_i) \to W_i$ are real analytic isomorphisms, that is, bijective with real analytic inverse. Then
    \[
    \rlct_x(\mathcal{I};\phi) = \min_{y \in \rho^{-1}(x)}\rlct_y(\rho^*\mathcal{I};(\phi \circ \rho)\cdot\vert\det\jac\rho\vert),
    \]
    where $\rho^*\mathcal{I} = \{g \circ \rho \mid g \in \mathcal{I}\} = \langle f_1 \circ \rho, \ldots, f_h \circ \rho \rangle \subseteq \mathcal{A}(M)$ is the pullback of $\mathcal{I}$ under $\rho$ and $\jac \rho$ is the Jacobian matrix of~$\rho$.
\end{fact}

In the proof of Theorem~\ref{thm:formula-learning-coeff}, we apply Fact~\ref{fact:chain-rule} with a specific map $\rho$, namely the \textit{blow-up} of $\R^d$ along the linear subspace defined by the ideal $\langle x_1,\ldots,x_c \rangle$. Its domain is the union of $c$ different charts isomorphic to $\R^d$ and the restriction of $\rho$ to the $i$th chart is given in local coordinates by \[(x_1,\ldots,x_d) \mapsto (x_ix_1,\ldots,x_ix_{i-1},x_i,x_ix_{i+1},\ldots,x_ix_c,x_{c+1},\ldots,x_d).\] This map is a real analytic isomorphism away from $x_1 = \cdots = x_c = 0$ in the image space.

\begin{fact}[Sum rule \text{\cite[Prop.~7]{Lin2017Ideal}}]
\label{fact:sum-rule}
Let $\Omega_1 \subseteq \R^{d_1}$ and $\Omega_2 \subseteq \R^{d_2}$ be compact subsets and let $\mathcal{I} \subseteq \mathcal{A}(\Omega_1)$ and $\mathcal{J} \subseteq \mathcal{A}(\Omega_2)$ be finitely generated ideals. Then, composing with the canonical projections $\Omega_1 \times \Omega_2 \to \Omega_i$, we can consider $\mathcal{I}$ and $\mathcal{J}$ as ideals of $\mathcal{A}(\Omega_1 \times \Omega_2)$. Let $\phi_1: \Omega_1 \to \R$ and $\phi_2: \Omega_2 \to \R$ be pointwise products of an analytic map and an everywhere positive and smooth map. Denote $\rlct_{\Omega_1 }(\mathcal{I};\phi_1) = (\lambda_1,\mathfrak m_1)$ and $\rlct_{\Omega_2}(\mathcal{J};\phi_2) = (\lambda_2,\mathfrak m_2)$.  Then
    \[\rlct_{\Omega_1 \times \Omega_2}(\mathcal{I} + \mathcal{J};\phi_1\cdot \phi_2) = (\lambda_1+\lambda_2,\mathfrak m_1+ \mathfrak m_2-1).\]
\end{fact}
The statement below follows from~\cite[Ex.~2.4]{drton2025Factor} by transforming a neighborhood of a smooth point in a real analytic variety to a linear space.

\begin{fact}\label{fact:smooth-RLCT} 
   Let $\Omega \subseteq \R^d$ be full-dimensional and compact, and let $\mathcal{I} \subseteq \mathcal{A}(\Omega)$ be a finitely generated ideal. Let $V = \{\omega \in \Omega : f(\omega) = 0 \text{ for all } f \in \mathcal{I}\} \subseteq \Omega$ be the real analytic variety defined by $\mathcal{I}$, that is, the common zero set of its elements. Let $x \in V$ be a smooth point at which $V$ has codimension $c$ within $\Omega$, that is, there exist $f_1,\ldots,f_c \in \mathcal{A}(\Omega)$ such that $\mathcal{I} = \langle f_1,\ldots,f_c\rangle$ and the Jacobian matrix $(\partial f_i/\partial x_j (x))$ has rank $c$. Then $\rlct_x(\mathcal{I};1) = (c,1)$.
\end{fact}

Denote by $\mathcal{L}^{p\times k}_{r,+}$ the space of real $p\times k$ matrices of the form 
\[
    \bl{W} = \begin{bmatrix}
    \bl{W}_{11} & \bl{0} \\
    \bl{W}_{21} & \bl{W}_{22}
\end{bmatrix},
\]
where $\bl{W}_{11}$ is an $r \times r$ lower triangular matrix with positive diagonal entries, $\bl{W}_{21} \in \R^{(p-r)\times r}$ and $\bl{W}_{22} \in \R^{(p-r)\times(k-r)}$.
The proof of the following lemma is analogous to the one of~\cite[Fact~3.5]{drton2025Factor}.

\begin{lemma}\label{lem:QR-decomposition}
     Let $\bs{\Sigma}_0 \in \mc{M}_r\setminus \mc{M}_{r-1}$ be a symmetric positive definite matrix that has exactly $p-r$ eigenvalues equal to the smallest eigenvalue. Then
    \[
    \min_{(\bl{W},\sigma^2)\in \R^{p\times k}\times \R_{>0}} \rlct_{(\bl{W},\sigma^2)}(\mathcal{I}_{p,k}(\bs{\Sigma}_0);1) = \min_{(\bl{W},\sigma^2)\in \mathcal{L}^{p\times k}_{r,+}\times \R_{>0}} \rlct_{(\bl{W},\sigma^2)}(\mathcal{I}_{p,k}(\bs{\Sigma}_0);1),
    \]
     where we consider $\mathcal{I}_{p,k}(\bs{\Sigma}_0)$ as an ideal of $\mathcal{A}(\mathcal{L}^{p\times k}_{r,+}\times \R_{>0})$ on the right side of this equation.
\end{lemma}

In a final lemma, we show that without loss of generality it can be assumed that $\bs{\Sigma}_0$ is a diagonal matrix. The basic idea is that multiplying the generators in $\bl{W}\bl{W}^\intercal + \sigma^2\bl{I}_p - \bs{\Sigma}_0$ by an invertible matrix gives a different set of generators of $\mathcal{I}_{p,k}(\bs{\Sigma}_0)$,  but does not change the ideal as a whole.

\begin{lemma}\label{lem:diagonalization}
    Let $\bs{\Sigma}_0 \in \mc{M}_r$ be a symmetric positive definite $p\times p$ matrix and let $\bl{V} \in \mathrm{O}(p)$ be orthogonal. Then
    \[
    \min_{(\bl{W},\sigma^2)\in \R^{p\times k}\times \R_{>0}} \rlct_{(\bl{W},\sigma^2)}(\mathcal{I}_{p,k}(\bs{\Sigma}_0);1) = \min_{(\bl{W},\sigma^2)\in \R^{p\times k}\times \R_{>0}} \rlct_{(\bl{W},\sigma^2)}(\mathcal{I}_{p,k}(\bl{V}\bs{\Sigma}_0\bl{V}^\intercal);1).
    \]
    In particular, diagonalizing $\bl{V}\bs{\Sigma}_0\bl{V}^\intercal = \diag(\delta_1,\ldots,\delta_p) \eqcolon \Delta$ with $\bl{V} \in \mathrm{O}(p)$ and $\delta_i \in \R_{>0}$,
    \[
    \min_{(\bl{W},\sigma^2)\in \R^{p\times k}\times \R_{>0}} \rlct_{(\bl{W},\sigma^2)}(\mathcal{I}_{p,k}(\bs{\Sigma}_0);1) = \min_{(\bl{W},\sigma^2)\in \R^{p\times k}\times \R_{>0}} \rlct_{(\bl{W},\sigma^2)}(\mathcal{I}_{p,k}(\Delta);1).
    \]
\end{lemma}

\begin{proof}
    Let $\mathcal{I}$ be the ideal generated by the entries of the matrix $\bl{V}(\bl{W}\bl{W}^\intercal + \sigma^2\bl{I}_p - \bs{\Sigma}_0)\bl{V}^\intercal = (\bl{V}\bl{W})(\bl{V}\bl{W})^\intercal + \sigma^2\bl{I}_p - \bl{V}\bs{\Sigma}_0\bl{V}^\intercal$. Every generator of $\mathcal{I}$ is a linear combination of generators of $\mathcal{I}_{p,k}(\bs{\Sigma}_0)$, so $\mathcal{I} \subseteq \mathcal{I}_{p,k}(\bs{\Sigma}_0)$. The converse is shown analogously by applying $\bl{V}^\intercal$ from the left and $\bl{V}$ from the right to the matrix of generators of $\mathcal{I}$. Now, using Fact~\ref{fact:chain-rule} on $\mathcal{I}$ and the real analytic isomorphism $\rho: \bl{W} \mapsto \bl{V}^\intercal\bl{W}$ yields the statement of the lemma.
\end{proof}

\subsection{Proof of our main result}

We first restate our main result (Theorem~\ref{thm:formula-learning-coeff}) on learning coefficients of PPCA models as the equivalent statement on the RLCT.

\begin{theorem}\label{thm:formula-learning-coeff-rlct}
Let $p \geq 1$ and $0 \leq r \leq k\leq p-1$, and take $\bs{\Sigma}_0 \in \mc{M}_r\setminus \mc{M}_{r-1}$. For $\phi_k$ a smooth and everywhere positive prior density on $\R^{p\times k} \times \R_{>0}$, the fiber ideal $\mathcal{I}_{p,k}(\bs{\Sigma}_0)$ of the $p$-dimensional PPCA model with $k$ PCs generated by the entries of $\bl{W}\bl{W}^\intercal + \sigma^2 \bl{I}_p - \bs{\Sigma}_0$ has global RLCT
\[
\min_{(\bl{W}_0,\sigma_0^2) \in \varphi_k^{-1}(\bs{\Sigma}_0)} \rlct_{(\bl{W}_0,\sigma_0^2)}(\mathcal{I}_{p,k}(\bs{\Sigma}_0); \phi_k) = \left(\frac{pk + r(p-k+1) + 2}{2}, 1  \right).
\]
\end{theorem}

\begin{proof}
    Using Lemma~\ref{lem:diagonalization}, we may assume that our fixed covariance matrix is of the form $\bs{\Sigma}_0 = \diag(\tau_1^2,\ldots,\tau_r^2,0,\ldots,0) + \sigma_0^2 \bl{I}_p$, where $\tau_1^2 \geq \tau_2^2 \geq \cdots \geq \tau_r^2 >0$ and $\sigma_0^2 > 0$. We write $\bl{D} = \diag(\tau_1^2,\ldots,\tau_r^2)$.
     Furthermore, by Lemma~\ref{lem:QR-decomposition}, we can restrict the parameters in $\bl{W}$ to the space $\mathcal{L}^{p\times k}_{r,+}$ of all real $p\times k$ matrices of the form
\[
\bl{W} = \begin{bmatrix}
    \bl{W}_{11} & \bl{0} \\
    \bl{W}_{21} & \bl{W}_{22}
\end{bmatrix},
\]
where $\bl{W}_{11}$ is an $r \times r$ lower triangular matrix with positive diagonal entries. 
Now, the relevant fiber ideal is the ideal of real analytic functions on $\mathcal{L}^{p\times k}_{r,+} \times \R_{>0}$ generated by the entries of $\bl{W}\bl{W}^\intercal + \sigma^2 \bl{I}_p - \bs{\Sigma}_0$. As we put the prior distribution on $\sigma^2$, we have to treat it as a variable $v = \sigma^2$ and write $\bl{W}\bl{W}^\intercal + v \bl{I}_p - \bs{\Sigma}_0$. Applying the transformation $v \mapsto v + \sigma_0^2$ yields the ideal generated by the entries of 
\[
\begin{aligned}
&\bl{W}\bl{W}^\intercal + v\bl{I}_p - \diag(\tau_1^2,\ldots,\tau_r^2,0,\ldots,0) \\
&\qquad =
\begin{bmatrix}
    \bl{W}_{11}\bl{W}_{11}^\intercal + v\bl{I}_r - \bl{D} & \bl{W}_{11}\bl{W}_{21}^\intercal \\
    \bl{W}_{21}\bl{W}_{11}^\intercal & \bl{W}_{21}\bl{W}_{21}^\intercal + \bl{W}_{22}\bl{W}_{22}^\intercal + v\bl{I}_{p-r}
\end{bmatrix},
\end{aligned}
\]
where now $v$ ranges over all real numbers $>-\sigma_0^2$, that is, $v \in \R_{>-\sigma_0^2}$.
Leaving all other variables unchanged while applying the transformation $\bl{W}_{21} \mapsto \bl{W}_{21}\bl{W}_{11}^\intercal$, where we note that $\bl{W}_{11}$ is invertible, yields the pullback of the ideal
\[
\begin{bmatrix}
    \bl{W}_{11}\bl{W}_{11}^\intercal + v\bl{I}_r - \bl{D} & \bl{W}_{21}^\intercal \\
    \bl{W}_{21} & \bl{W}_{21}\bl{W}_{11}^{-T}\bl{W}_{11}^{-1}\bl{W}_{21}^\intercal + \bl{W}_{22}\bl{W}_{22}^\intercal + v\bl{I}_{p-r}
\end{bmatrix}.
\]
Since every entry of $\bl{W}_{21}$ is now a generator of the ideal, we can eliminate the terms coming from $ \bl{W}_{21}\bl{W}_{11}^{-T}\bl{W}_{11}^{-1}\bl{W}_{21}^\intercal$ in the lower right block of generators. So, the entries of the following matrix generate the same ideal:
\[
\begin{bmatrix}
    \bl{W}_{11}\bl{W}_{11}^\intercal + v\bl{I}_r - \bl{D} & \bl{W}_{21}^\intercal \\
    \bl{W}_{21} &  \bl{W}_{22}\bl{W}_{22}^\intercal + v\bl{I}_{p-r}
\end{bmatrix}.
\]
Using the sum rule (Fact~\ref{fact:sum-rule}), we can disregard the generators that are entries of $\bl{W}_{21}$ by adding $((p-r)r,0)$ to the RLCT in the end. Hence, we are left with $\mathcal{I} = \mathcal{I}_1 + \mathcal{I}_2$, where
\begin{align*}
    \mathcal{I}_1 &= \langle \bl{W}_{11}\bl{W}_{11}^\intercal + v \bl{I}_r - \bl{D} \rangle, \\
    \mathcal{I}_2 &= \langle \bl{W}_{22}\bl{W}_{22}^\intercal + v \bl{I}_{p-r}\rangle.
\end{align*}

We compute the RLCT of the ideal $\mathcal{I}$ at every fixed point $(\bl{W}_0, v_0) \in \mathcal{L}^{p\times k}_{r,+} \times \R_{>-\sigma_0^2}$ whose projection lies on the variety $V(\mathcal{I})$ defined by $\mathcal{I}$. If $v_0 \neq 0$ then $(\bl{W}_0,v_0)$ does not project into $V(\mathcal{I})$ because $\rank(\bl{W}_{22}\bl{W}_{22}^\intercal) \leq k-r < p-r = \rank(-v_0\bl{I}_{p-r})$ for every $\bl{W}_{22} \in \R^{(p-r) \times (k-r)}$. We can hence assume that $v_0 = 0$ and pick a small neighborhood $\Omega$ with compact closure around the projection to $(\bl{W}_{11},\bl{W}_{22},v)$-coordinates of the point $(\bl{W}_0,0)$. Suppose for simplicity that $\Omega = U \times [-\delta,\delta]$, where $U$ is a neighborhood of $(\bl{W}_{0,11}, \bl{W}_{0,22})$ and $\delta > 0$.

We now apply the change of $\sigma \mapsto \sigma^2 = v$ which gives a Jacobian term of $2\sigma$ and the two relevant fiber ideals $\mathcal{I}^+ = \mathcal{I}^+_1 + \mathcal{I}^+_2$ and $\mathcal{I}^- = \mathcal{I}^-_1 + \mathcal{I}^-_2$, where
\begin{align*}
    \mathcal{I}^+_1 &= \langle \bl{W}_{11}\bl{W}_{11}^\intercal + \sigma^2 \bl{I}_r - \bl{D} \rangle, \\
    \mathcal{I}^+_2 &= \langle \bl{W}_{22}\bl{W}_{22}^\intercal + \sigma^2 \bl{I}_{p-r}\rangle, \\
    \mathcal{I}^-_1 &= \langle \bl{W}_{11}\bl{W}_{11}^\intercal - \sigma^2 \bl{I}_r - \bl{D} \rangle, \\
    \mathcal{I}^-_2 &= \langle \bl{W}_{22}\bl{W}_{22}^\intercal - \sigma^2 \bl{I}_{p-r}\rangle.
\end{align*}

For deducing the RLCT of the original ideal from these, note that,
writing for $\vert \vert \ . \ \vert \vert_F$ the Frobenius norm for matrices, the zeta integrals (\ref{equ:zeta}) relate as follows, where we write $d\bl{W}= d\bl{W}_{11} d\bl{W}_{22}$ by abuse of notation:
\begin{align*}
    \int_{U}\int_{-\delta}^\delta (\vert \vert \mathcal{I}_1 \vert \vert_F^2 &+ \vert \vert \mathcal{I}_2 \vert \vert_F^2)^{-\frac{z}{2}} dv  d\bl{W} \\
    =  \int_{U}\int_{0}^{\sqrt\delta}\left(\vert \vert \mathcal{I}_1^+ \vert \vert_F^2 + \vert \vert \mathcal{I}_2^+ \vert \vert_F^2\right)^{-\frac{z}{2}} \vert 2\sigma \vert d\sigma  d\bl{W}  &+   \int_{U}\int_{0}^{\sqrt\delta}\left(\vert \vert \mathcal{I}_1^- \vert \vert_F^2 + \vert \vert \mathcal{I}_2^- \vert \vert_F^2\right)^{-\frac{z}{2}} \vert 2\sigma \vert d\sigma  d\bl{W}  \\
    =  \int_{U}\int_{-\sqrt\delta}^{\sqrt\delta}\left(\vert \vert \mathcal{I}_1^+ \vert \vert_F^2 + \vert \vert \mathcal{I}_2^+ \vert \vert_F^2\right)^{-\frac{z}{2}} \vert \sigma \vert d\sigma  d\bl{W} &+ \int_{U}\int_{-\sqrt\delta}^{\sqrt\delta}\left(\vert \vert \mathcal{I}_1^- \vert \vert_F^2 + \vert \vert \mathcal{I}_2^- \vert \vert_F^2\right)^{-\frac{z}{2}} \vert \sigma \vert d\sigma  d\bl{W}
\end{align*}
We compute the RLCT of $\mathcal{I}^e$ at the relevant point $\bl{W}_{11} = 0$, $\bl{W}_{22}=0$, $\sigma = 0$ simultaneously for $e \in \{+, -\}$.
For this, we perform the blow-up along the linear subspace defined by $\bl{W}_{22} = 0$ and $\sigma= 0$. We have to distinguish the charts associated to $\sigma$ and $w_{pk}$. This is sufficient because the generators of $\mathcal{I}^e$ are symmetric with respect to variables in $\bl{W}_{22}$.  

We first consider the $\sigma$-chart whose Jacobian determinant is $2\sigma^{(p-r)(k-r)+1}$, taking into account the Jacobian term of $2\sigma$ from before. The pullback of $\mathcal{I}_1^e$ remains unchanged while the pullback of $\mathcal{I}_2^e$ is $\sigma^2\langle \bl{W}_{22}\bl{W}_{22}^\intercal + e \bl{I}_{p-r}\rangle $. The generators of the second factor have no common complex zeros, because $\rank(\bl{W}_{22}\bl{W}_{22}^\intercal) \leq k-r < p-r = \rank(e\bl{I}_{p-r})$ for every $\bl{W}_{22} \in \R^{(p-r) \times (k-r)}$, and hence the ideal contains the constant function $1$ by Hilbert's Nullstellensatz. Consequently, the pullback of $\mathcal{I}$ under the $\sigma$-chart is $\langle \bl{W}_{11}\bl{W}_{11}^\intercal + e\sigma^2 \bl{I}_r - \bl{D}\rangle  + \langle \sigma^2 \rangle = \langle \bl{W}_{11}\bl{W}_{11}^\intercal - \bl{D}\rangle + \langle \sigma^2 \rangle$. After rescaling the $\ell$-th row of $\bl{W}_{11}$ by $\tau_\ell >0$ for $\ell \in \{1,\ldots, r\}$, we arrive at the ideal $\langle \bl{W}_{11}\bl{W}_{11}^\intercal - \bl{I}_r\rangle + \langle \sigma^2 \rangle$, where $\bl{W}_{11}$ still varies over the lower triangular $r \times r$ matrices with positive diagonal entries. The only real zero of this ideal is $\bl{W}_{11} = \bl{I}_r$, $\sigma = 0$. Using the sum rule, we can disregard $\langle \sigma^2 \rangle$ by adding $\left(\frac{(p-r)(k-r)+2}{2},0\right)$ to the RLCT in the end, which comes from $\rlct_0(\langle \sigma^2\rangle ;2\sigma^{(p-r)(k-r)+1})$. 
So, we are left with computing $\rlct_{\bl{I}_r}(\langle \bl{W}_{11}\bl{W}_{11}^\intercal - \bl{I}_r\rangle; 1)$. This ideal defines a smooth real analytic variety of codimension $\frac{r(r+1)}{2}$ and, hence, its RLCT is $\left( \frac{r(r+1)}{2}, 1 \right)$, see Fact~\ref{fact:smooth-RLCT}. Consequently, we compute the RLCT of the pullback of $\mathcal{I}$ on the $\sigma$-chart as
\[
\left(\frac{r(r+1)}{2} + \frac{(p-r)(k-r)+2}{2}, 1  \right).
\]

Finally, we consider the chart for $w_{pk}$. The Jacobian determinant is $2w_{pk}^{(p-r)(k-r)+1} \sigma$. Factoring out $w_{pk}^2$ in the pullback of $\mathcal{I}^e_2$, the cofactor contains the generator $ \left(\sum_{\ell = r+1}^{k-1} w_{p\ell}^2\right) + 1 + e\sigma^2$ which is a unit in the ring of analytic functions in a small neighborhood of the origin. Using the remaining factor $(w_{pk}^2)$, we can eliminate the $w_{pk}^2\sigma^2 \bl{I}_r$ terms in the pullback of $\mathcal{I}^e_1$. Again, this pullback defines a smooth real analytic variety of codimension $\frac{r(r+1)}{2}$. Hence, using the sum rule, the RLCT of the pullback of $\mathcal{I}^e$ on this chart is the same as on the $\sigma$-chart.
 Adding the contribution of $\bl{W}_{21}$, we see that the RLCT of the fiber ideal of the PPCA model is
\[
\left((p-r)r + \frac{r(r+1)}{2} + \frac{(p-r)(k-r)+2}{2}, 1  \right) = \left(\frac{pk + r(p-k+1) + 2}{2}, 1  \right).
\]
The statement on learning coefficients follows immediately from Fact~\ref{fact:fiber-ideal}.
\end{proof}

As a partial confirmation of formula \eqref{eqn:LearningCoeffExpression} for the learning coefficient we note that when $\bs{\Sigma}_0 \in \mc{M}_{k} \setminus \mc{M}_{k-1}$ is a non-singular point in the PPCA model with $k$ principal components the learning coefficient is $\lambda_{kk} = \tfrac{1}{2}(k(p-k) + k(k+1)/2 + 1)$. The learning coefficient in this case is one half of the dimension of the manifold $\mc{M}_k \setminus \mc{M}_{k-1}$, which is the usual BIC based penalty and is typical for non-singular points \citep[Chap.~6]{WatanabeSingularLearningBook}.

\begin{remark}
Rather than assuming a zero mean, a more general PPCA model has $\bl{X}_i \overset{i.i.d.}{\sim} \mc{N}_p(\bs{\mu},\bs{\Sigma}_0)$ with unknown mean $\bs{\mu} \in \mb{R}^p$ and $\bs{\Sigma}_0 \in \mc{M}_r \setminus \mc{M}_{r-1}$. The only modification needed in this setting is that $\tfrac{p}{2}$ must be added to every learning coefficient so that $\lambda_{kr} = \tfrac{1}{4}(pk + r(p-k+1) + 2 + 2p)$ with multiplicity $\mf{m}_{kr} = 1$. This is a consequence of the sum rule with respect to the ideals $\mathcal{I}_{p,k}(\bs{\Sigma}_0) = \langle \bl{W}\bl{W}^\intercal + \sigma^2\bl{I}_p - \bs{\Sigma}_0 \rangle$ and $\mc{J}_p(\bs{\mu}_0) = \langle \bs{\mu} - \bs{\mu}_0 \rangle$, as well as the fact that $\rlct_{\mb{R}^p}(\mathcal{J}_p(\bs{\mu}_0);1) = (p,1)$. Adding the constant $\tfrac{p}{2}$ to every learning coefficient will simply result in adding $-\tfrac{p\log(n)}{2}$ to every sBIC term. Model selection with the sBIC is therefore essentially the same whether one assumes that the mean is zero or not; the only modification needed is that the value of the maximized log-likelihood differs in these two cases.  
\end{remark}

\begin{remark}\label{rmk:singularity-structure}
    Theorem~\ref{thm:formula-learning-coeff} and its proof imply that the learning coefficients and, more generally, the real singularity types of the $p$-dimensional PPCA model with $k$ principal components are independent of the choice of $\bs{\Sigma}_0 \in \mc{M}_r \setminus \mc{M}_{r-1}$, and only change when $\bs{\Sigma}_0$ resides in the model with $r-1$ principal components. This is a special property of the PPCA model; compare, for instance,~\cite[Sec.~5]{drton2025Factor}.
\end{remark}

\section{Partitioned noise PPCA models}
\label{sec:UnifyPPCAandFA}

We introduce a class of models that is a common generalization of factor analysis and PPCA. As before, we consider a $p$-dimensional Gaussian with mean $\bl{0}$ and a covariance matrix that is parameterized by the function
\begin{align*}
    \varphi_{k,\bl{d}}: \R^{p\times k} \times \R_{>0}^s &\to \mathcal{S}^p_{++}\\
    (\mathbf{W},\bs{\psi}) &\mapsto \bl{W}\bl{W}^\intercal + \diag( \underbrace{\psi_1,\ldots,\psi_1}_{d_1},
        \ldots,
        \underbrace{\psi_s,\ldots,\psi_s}_{d_s}),
\end{align*}
where $0 \leq k \leq p$ and $\bl{d} = (d_1,\ldots,d_s) \in \Z_{>0}^s$ with $\sum_{i = 1}^s d_i = p$ is an (ordered) integer partition of $p$. We denote the image of $\varphi_{k,\bl{d}}$ by $\mathcal{M}_{k,\bl{d}}$ and call it the \emph{$\bl{d}$-partitioned noise PPCA model with $k$ principal components}. 
Note that the cases $s = 1$ and $s = p$ correspond to the PPCA and factor analysis models respectively. In particular, $\mathcal{M}_{k,(p)} = \mathcal{M}_k$ and $\mathcal{M}_k$ is a submodel of the factor analysis model $\mc{F}_k = \mathcal{M}_{k,\underbrace{(1,\ldots,1)}_p}$. Figure \ref{fig:HasseDiagramPPCAFA} displays the relationship between the PPCA and factor analysis models when $p = 6$. Arrows in the diagram represent model containment and point from smaller models to larger ones. The models $\mc{M}_5$ and $\mc{F}_5$ parameterize the entire space of positive definite matrices. Like PPCA, the factor analysis model $\mc{F}_k$ equals the entire positive definite cone only when $k \geq p-1$; for any $k$ smaller than $p-1$, the space $\mc{S}^p_{++}\setminus \mc{F}_k$ has a non-empty interior \citep[Thm 2]{BoydReiersolThmRankDimension}. However, one distinction between these two models is that the factor analysis model can be full-dimensional even when $k < p-1$. For example, the codimension, $\dim(\mc{S}^p_{++}) - \dim(\mc{F}_k)$, of $\mc{F}_3$ is zero when $p = 6$ \citep[Thm.~2]{AlgFactorAnalysisDrton}.

\begin{figure}[t]
    \centering
    \begin{tikzpicture}[
        >=Stealth, 
        every node/.style={
            circle, 
            draw=black, 
            thick, 
            minimum size=10mm,  
            text width=7mm,     
            align=center        
        },
        directed edge/.style={
            ->, 
            thick, 
            shorten >=1pt, 
            shorten <=1pt
        }
    ]
        \foreach \i in {0,...,5} {
            \node (T\i) at (\i*2.2, 2) {$\mathcal{F}_{\i}$};
        }
        
        \foreach \i in {0,...,5} {
            \node (B\i) at (\i*2.2, 0) {$\mathcal{M}_{\i}$};
        }

        \foreach \i in {0,...,5} {
            \draw[directed edge] (B\i) -- (T\i);
            
            \ifnum\i<5
                \draw[directed edge] (T\i) -- (T\the\numexpr\i+1\relax);
                \draw[directed edge] (B\i) -- (B\the\numexpr\i+1\relax);
            \fi
        }
    \end{tikzpicture}
    \caption{Model containment between the PPCA ($\mc{M}_i$) and factor analysis ($\mc{F}_i$) models when $p = 6$. Arrows point from smaller to larger models.}
    \label{fig:HasseDiagramPPCAFA}
\end{figure}
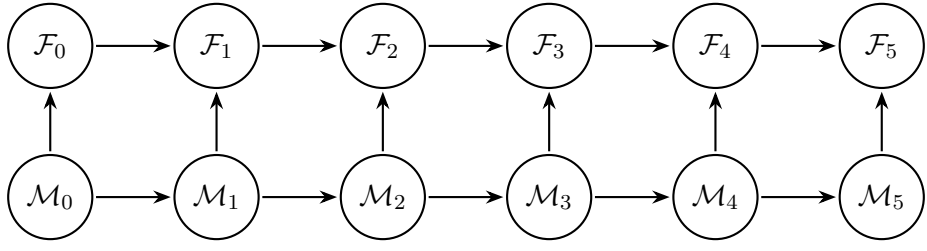

More generally, for $\bl{e} = (e_1,\ldots,e_t) \in \Z_{>0}^t$, there is an inclusion of models  $\mathcal{M}_{r,\bl{e}} \subseteq \mathcal{M}_{k,\bl{d}}$ if $r \leq k$ and the noise partition $\bl{d}$ is a refinement of $\bl{e}$.
More formally, the latter is equivalent to $r \leq k$ and
the existence of $1 = i_1 < \cdots < i_{t+1} = s+1$ such that for all $j \in \{1,\ldots,t\}$, we have $e_j = d_{i_j} + \ldots + d_{i_{j+1} - 1}$.
In this case, we write $\bl{d} \preceq \bl{e}$, so that $\mathcal{M}_{r,\bl{e}} \subseteq \mathcal{M}_{k,\bl{d}}$ if $r \leq k$ and $\bl{d} \preceq \bl{e}$.

\begin{theorem}\label{thm:generalized-formula-learning-coeff}
Let $\bl{d} \preceq \bl{e}$ be integer partitions of $p\ge 1$, and let $0 \leq r \leq k\leq p-1$. Write $s$ for the length of $\bl{d}$ and suppose
\begin{align}
\label{eqn:FADimensionRequirement}
    \binom{p-k+1}{2} \geq s.
\end{align}
Assume that $\bs{\Sigma}_0$ is chosen generically from  $\mc{M}_{r,\bl{e}}$ and that $\mc{M}_{r,\bl{e}}$ is minimal among the partitioned noise PPCA models containing $\bs{\Sigma}_0$.
Let $\phi_k$ be any smooth and everywhere positive bounded prior distribution on parameters $(\bl{W},\bs{\psi}) \in \R^{p\times k} \times (\R_{>0})^s$. Then the fiber ideal $\mathcal{I}_{p,k,\bl{d}}(\bs{\Sigma}_0)$ of the $p$-dimensional $\bl{d}$-partitioned noise PPCA model with $k$ principal components generated by the entries of 
\[
\bl{W}\bl{W}^\intercal + \diag( \underbrace{\psi_1,\ldots,\psi_1}_{d_1},
        \ldots,
        \underbrace{\psi_s,\ldots,\psi_s}_{d_s}) - \bs{\Sigma}_0
\]
has global RLCT
\[
\min_{(\bl{W},\bs{\psi}) \in \R^{p\times k} \times (\R_{>0})^s}\rlct_{(\bl{W},\bs{\psi})}(\mathcal{I}_{p,k,\bl{d}}(\bs{\Sigma}_0); \phi_k) = \left(\frac{pk + r(p-k+1) + 2s}{2}, 1  \right).
\]
In particular, the learning coefficient and its multiplicity of this model $\mc{M}_{k,\bl{d}}$ along the submodel $\mc{M}_{r,\bl{e}}$ are
\begin{align}
    \label{eqn:LearningCoeffExpressionGeneral}
    \lambda_{kr}^{\bl{d}\bl{e}} = \frac{pk + r(p-k+1) + 2s}{4} \quad \text{and} \quad \mf{m}_{kr}^{\bl{d}\bl{e}} = 1.
\end{align}
\end{theorem}

\begin{remark}
    Note that $\lambda_{kr}^{\bl{d}\bl{e}}$ in Theorem~\ref{thm:generalized-formula-learning-coeff} is completely independent of $\bl{e}$ and only depends on the length $s$ of the partition $\bl{d}$.
\end{remark}

Note that the condition
$\binom{p-k+1}{2}\geq s$ is equivalent to
\[
pk-\binom{k}{2}+s\leq\binom{p+1}{2}.
\]
After accounting for the orthogonal invariance of the
loading matrix $\bl{W}$, this is precisely the range in which the
model with $s$ noise variance parameters is not
overparametrized relative to the ambient space of
covariance matrices. Equality turns up at the boundary at which
the model becomes full-dimensional.

The learning coefficients from Theorem \ref{thm:generalized-formula-learning-coeff} in the sBIC can be used as a model selection criterion for choosing between the PPCA and factor analysis models with an unknown number of components.  If we make the simplifying assumption that the true number of components is equal to $k$, the expression for the sBICs \eqref{eqn:sBICLi} takes the simple form given in the following corollary. 

\begin{corollary}
    Assume that only the PPCA and factor analysis models with $k$ components, $\mc{M}_k$ and $\mc{F}_k$, are under consideration for $k$ with $\binom{p-k+1}{2} \geq p$. The sBICs for the PPCA model and the factor analysis model are equal to
    \begin{align}
    \label{eqn:sBICPPCA}
        \log(L_{PPCA}) &= \log\big(p_{PPCA}(\bl{X}|\hat{\bl{W}}, \hat{\sigma}^2)\big) -  \frac{pk + k(p-k+1) + 2 }{4} \log(n)
        \\
        \label{eqn:sBICFA}
        \log(L_{FA}) &= \log\big( p_{FA}(\bl{X}|\hat{\bl{W}},\hat{\bs{\psi}}) \big) -  \frac{pk + k(p-k+1) + 2p }{4} \log(n).
    \end{align}
    The PPCA model will have a higher sBIC than the factor analysis model if and only if 
    \begin{align*}
        \log\bigg( \frac{p_{FA}(\bl{X}|\hat{\bl{W}}, \hat{\bs{\psi}})}{ p_{PPCA}(\bl{X}|\hat{\bl{W}},\hat{\sigma}^2)} \bigg) < \frac{p-1}{2} \log(n).
    \end{align*}
\end{corollary}
\begin{proof}
    Indexing the PPCA model by $0$ and the factor analysis model by $1$, we note that $\lambda_{11} = \lambda_{10}$ by Theorem \ref{thm:generalized-formula-learning-coeff}. If we define $\tilde{L}_1 = n^{-\lambda_{11}} p_{FA}(\bl{X}|\hat{\bl{W}},\hat{\bs{\psi}})$, the formula \eqref{eqn:sBICLi} yields
    \begin{align*}
        L_1 = \tfrac{1}{2}\big( \tilde{L}_1 - L_0 + \sqrt{(\tilde{L}_1 - L_0)^2 + 4\tilde{L}_1L_0}\big) = \tilde{L}_1.
    \end{align*}
\end{proof}
The difference in the model complexity penalties in the sBICs \eqref{eqn:sBICPPCA} and \eqref{eqn:sBICFA} amounts to the difference in the dimensions of the parameter spaces of the models; PPCA has a single variance parameter $\sigma^2$, while factor analysis has a variance parameter $\bs{\psi} \in \mb{R}^p$. As soon as models with a differing number of components are considered the sBIC formulas become more complicated and it becomes necessary to make extra adjustments beyond performing a dimension count. 

The proof of Theorem~\ref{thm:generalized-formula-learning-coeff} for the partitioned noise model needs considerably more effort than Theorem~\ref{thm:formula-learning-coeff} for the PPCA model. We begin with two genericity statements.

\subsection{Genericity conditions for partitioned noise}
\label{subsec:partition-genericity}

Throughout, for an ordered
partition $\bl{d}=(d_1,\ldots,d_s)$ of $p$, write
\[
 P_\nu=\diag(0,\ldots,0,\bl{I}_{d_\nu},0,\ldots,0)
 \quad(1\leq\nu\leq s),
\]
where the nonzero block of $P_\nu$ is at position $\nu$. Thus, the noise matrix in the definition of
$\varphi_{k,\bl{d}}$ is $\sum_{\nu=1}^s\psi_\nu P_\nu$.
For $J\subseteq\{1,\ldots,s\}$, set
\[
 d_J=\sum_{\nu\in J}d_\nu,\qquad
 P_J=\sum_{\nu\in J}P_\nu.
\]
We denote by $\mc{S}^q$ the vector space of symmetric $q\times q$
matrices. We say that a property holds for a generic real matrix if
it holds outside a proper real algebraic subset.

\begin{lemma}\label{lem:partition-generic-orthonormal}
Let $p\geq q>m\geq0$, let $\bl{d}$ be a partition of $p$ of length $s$, and suppose that
\begin{equation}\label{eq:partition-codim}
  \binom{q-m+1}{2}\geq s.
\end{equation}
There is a nonempty Zariski open set of matrices $V\in\R^{q\times p}$
with the following properties:
\begin{enumerate}[label=(\arabic*)]
\item For every $J\subseteq\{1,\ldots,s\}$ with $d_J\leq q$,
      the $q\times d_J$ matrix $V_J$ formed by the columns in
      the blocks indexed by $J$ has full column rank.
\item If $a=(a_1,\ldots,a_s)\in\R^s$ satisfies
\[
 \rank\left(\sum_{\nu=1}^s a_\nu V P_\nu V^\intercal\right)\leq m,
\]
then $d_J\leq m$, where $J=\{\nu:a_\nu\neq0\}$.
\item For each $J$ with $d_J\leq m$, the columns outside $J$,
      projected to $\R^q/\operatorname{im}(V_J)$, satisfy the
      analogous assertions (1)--(3) for the partition obtained by
      deleting the blocks in $J$, and for
      $(p,q,m)$ replaced by $(p-d_J,q-d_J,m-d_J)$.
\end{enumerate}
The assertions hold also for generic orthonormal row matrices $V$
and their left multiples by invertible matrices.
\end{lemma}

\begin{proof}
The first assertion is a non-empty Zariski open condition. For
(2), fix a non-empty set $J$ and write $j=|J|$. If
$d_J\leq q$, the matrix $V_J$ generically has full column rank, and
so
\[
 \rank\left(V_J\diag(a_\nu\bl{I}_{d_\nu}:\nu\in J)
                    V_J^\intercal\right)= d_J
\]
whenever all $a_\nu$ are non-zero for $\nu \in J$. In particular,
the required rank inequality is impossible if $m<d_J$.

Suppose that $d_J>q$. We may work over $\mathbb{C}$, since a generic
complex algebraic condition induces a generic real one. For fixed
$[a]\in(\mathbb{C}^*)^j/\mathbb{C}^*$, consider
\[
 V_J\longmapsto V_J\diag(a_\nu\bl{I}_{d_\nu}:\nu\in J)
                          V_J^\intercal.
\]
At a matrix $V_J$ of full row rank this is a submersion onto
$\mc{S}^q(\mathbb C)$. Indeed, its differential maps $H$ to
$HDV_J^\intercal+V_JDH^\intercal$, where $D= \diag(a_\nu\bl{I}_{d_\nu}:\nu\in J)$, and $DV_J^\intercal$ has a left
inverse, which makes the differential surjective. The variety of
symmetric $q\times q$ matrices of rank at most $m$ has codimension
$c=\binom{q-m+1}{2}$ which can be shown using the Schur complement
with respect to some $m\times m$ rank $m$ submatrix. 
Thus, the variety of pairs
$(V_J,[a])$ satisfying the rank inequality has dimension at most
\[
 qd_J+(j-1)-c<qd_J,
\]
because $j\leq s\leq c$. Hence, its projection is contained in a proper
algebraic subset of the space of $V_J$. There are only finitely many
choices of $J$, proving genericity of (2).

For (3), choose a basis of the quotient by the independent columns
$V_J$. For fixed such columns, the projected remaining columns are
generic, and their block sizes are unchanged. Moreover,
\[
 (q-d_J)-(m-d_J)=q-m,
\]
so the hypothesis on the binomial coefficient remains true, while
the number of remaining blocks is at most $s$. We apply (1) and (2)
to this smaller problem. As there are only finitely many sets of deleted blocks, the same
argument applies simultaneously after every successive deletion.
This proves the recursive assertion (3).

Finally, the assertions are invariant under left multiplication by
an invertible matrix. Moreover, every matrix of full row rank can be normalized
to have orthonormal rows by such a multiplication, and genericity
therefore holds on the space of orthonormal row matrices as well.
\end{proof}

The second genericity concerns the possible loading matrices 
$\bl{W}$ in a fiber of a given covariance matrix.

\begin{lemma}\label{lem:partition-fiber-classification}
Let $0\leq r\leq k<p$ and suppose
\begin{equation}\label{eq:partition-main-assumption}
 c:=\binom{p-k+1}{2}\geq s.
\end{equation}
For generic $\bl{L}\in\R^{p\times r}$, suppose that
\[
 M=\bl{L}\bl{L}^\intercal+\sum_{\nu=1}^s\beta_\nu P_\nu
\]
is positive semidefinite and has rank at most $k$. Set
$J=\{\nu:\beta_\nu\neq0\}$. Then one of the following holds:
\begin{enumerate}[label=(\arabic*)]
\item $J=\varnothing$ and $M=\bl{L}\bl{L}^\intercal$ has rank $r$.
\item $J\neq\varnothing$, $d_J\leq k-r$, and $\beta_\nu>0$
      for every $\nu\in J$. Moreover, $\rank M=r+d_J$, and
      \[
       M=(\bl{L}\ \bl{D}_J)(\bl{L}\ \bl{D}_J)^\intercal,
      \]
      where $\bl{D}_J$ has columns
      $\sqrt{\beta_\nu}\,e_i$, for all coordinate indices $i$
      belonging to a block $\nu\in J$.
\item $J=\{1,\ldots,s\}$, $c=s$, and $\rank M=k$.
      Any corresponding factorization $M=\bl{W}\bl{W}^\intercal$
      with $\bl{W}\in\R^{p\times k}$, the parameterization
      $\varphi_{k,\bl{d}}$ has a regular point.
\end{enumerate}
\end{lemma}

\begin{proof}
Let $T$ be the union of the coordinate blocks not indexed by $J$,
write $t=|T|=p-d_J$, and let $\bl{L}_T$ be the matrix of rows in $T$.
We may choose $\bl{L}$ generically so that every submatrix with at
least $r$ rows has column rank $r$.

\textbf{Case 1.} Suppose that $t\geq r$. Choose $r$ rows
$R\subseteq T$ with $\bl{L}_R$ invertible. After permuting rows,
$M$ has upper left block $\bl{L}_R\bl{L}_R^\intercal$. Its Schur
complement is the restriction of
$\sum_\nu\beta_\nu P_\nu$ to the remaining rows which all lie in $J$. Therefore,
\[
 \rank M=r+d_J.
\]
Since $M$ is positive semidefinite, its Schur complement is
positive semidefinite, which forces $\beta_\nu>0$ for
$\nu\in J$. The upper bound on the rank $r \leq t$ gives $d_J\leq k-r$. 
The asserted factorization now follows immediately.
This includes the case
$J=\varnothing$.

\textbf{Case 2.} Suppose that $t<r$. Set $h=r-t$ and write,
after a permutation of rows,
\[
 \bl{L}=\begin{bmatrix}\bl{L}_T\\\bl{L}_{T^c}\end{bmatrix},
 \qquad
 B=\diag(\beta_\nu\bl{I}_{d_\nu}:\nu\in J).
\]
Generically $\bl{L}_T$ has full row rank. The Schur complement
of $\bl{L}_T\bl{L}_T^\intercal$ is now of the form
\[
 B+HH^\intercal,\qquad
 H=\bl{L}_{T^c}N\in\R^{d_J\times h},
\]
where the columns of $N$ form an orthonormal basis of
$\ker\bl{L}_T$. For generic $\bl{L}$, the matrix $H$ is generic
and has column rank $h$. Since $B$ is invertible, the kernel
identity for Schur complements gives
\begin{equation}\label{eq:partition-rank-reduction}
 \rank M=p-h+\rank\big(\bl{I}_h+H^\intercal B^{-1}H\big).
\end{equation}
Put $d=p-k>0$. If $h<d$, the rank bound is impossible. Otherwise,
\eqref{eq:partition-rank-reduction} says that
\begin{equation}\label{eq:partition-rank-condition}
 \rank\left(\bl{I}_h+
 \sum_{\nu\in J}\beta_\nu^{-1}H_\nu^\intercal H_\nu\right)
 \leq h-d,
\end{equation}
where $H_\nu$ consists of the $d_\nu$ rows of $H$ in block $\nu$.
The determinantal variety of symmetric $h\times h$ matrices of rank
at most $h-d$ has codimension $c=\binom{d+1}{2}$, again by using Schur complements.

For fixed nonzero $(\beta_\nu^{-1}:\nu\in J)$, the map
$H\mapsto H^\intercal B^{-1}H$ is a submersion at every
$H$ of full column rank, by the same differential calculation as in
the proof of Lemma~\ref{lem:partition-generic-orthonormal}.
Consequently, the variety defined by
\eqref{eq:partition-rank-condition}, in the variables $H$ and
the $|J|$ non-zero scalars $\beta_\nu^{-1}$, has codimension $c$
and dimension at most
\[
 d_Jh+|J|-c.
\]
When $|J|<c$, its projection to $H$-space is contained in a proper
algebraic subset. In particular, all proper non-empty subsets $J$
are excluded. The same argument excludes the full set if $c>s$.
Thus the only remaining possibility is $J=\{1,\ldots,s\}$ and
$c=s$.

At points of this variety of rank exactly $h-d$, the map just considered is a submersion and
the determinantal subvariety of rank exactly $h-d$ is smooth of
codimension $c$. The subvarieties corresponding to smaller rank have larger codimension
and can be excluded for generic $H$. By generic smoothness for morphism of schemes
applied to the projection to $H$-space, either its image is not dense (in which
case (3) is impossible for generic $H$), or all fibers over
a non-empty Zariski open set are reduced and zero dimensional.
In the latter case,
\eqref{eq:partition-rank-reduction} gives $\rank M=k$ which
implies the smoothness condition on the parametrization map.
\end{proof}

\subsection{The blow-up computation}
\label{subsec:partition-blowup}
As for the PPCA model, the central transformation of the proof is a single blow-up step which we will construct in the more general setting of the lemma below.

\begin{lemma}\label{lem:partition-uniform-blowup}
Let $p\geq q>m\geq0$, let $\bl{d}$ be a partition of $p$ into $s$
blocks, and suppose $\binom{q-m+1}{2}\geq s$. Let $V(u,t)$ be a
real analytic family of $q\times p$ matrices, defined in a
neighborhood of $(u,t)=(0,0)$, such that $V(0,0)$ satisfies the
conditions in Lemma~\ref{lem:partition-generic-orthonormal}. For
arbitrary signs $\epsilon_\nu\in\{\pm1\}$, set
\begin{equation}\label{eq:partition-standard-generators}
 S(Y,u,t)=YY^\intercal+
   V(u,t)\left(\sum_{\nu=1}^s
                  \epsilon_\nu u_\nu^2P_\nu\right)V(u,t)^\intercal,
 \qquad Y\in\R^{q\times m}.
\end{equation}
Then
\[
 \rlct_{(Y,u)=(0,0)}
   \left(\langle S(Y,u,0)\rangle;\prod_{\nu=1}^s u_\nu \right)
 =\left(s+\frac{qm}{2},1\right).
\]
Here $\langle S\rangle$ denotes the ideal generated by all entries
of $S$. Moreover, for every $z<s+qm/2$, there are sufficiently
small neighborhoods $U$ of $(Y,u)=0$ and $O$ of $t=0$ such that
\begin{equation}\label{eq:partition-uniform-integral}
 \sup_{t\in O}\int_U\|S(Y,u,t)\|_F^{-z}
                     \prod_{\nu=1}^s|u_\nu|\,dY\,du<\infty.
\end{equation}
\end{lemma}

\begin{proof}
The assertion on the RLCT follows from the uniform assertion on the integrals
which we prove by induction on $m$. If $m=0$,
Lemma~\ref{lem:partition-generic-orthonormal}(2) implies that
$VP_1V^\intercal,\ldots,VP_sV^\intercal$ are linearly
independent at $u=t=0$ and hence remain so in a neighborhood.
The ideal generated by the entries of $S(0,u,t)$, which is exactly
the relevant ideal in the case $m = 0$, is therefore
$\langle u_1^2,\ldots,u_s^2\rangle$. Its RLCT with amplitude
$\prod u_\nu$ is $(s,1)$, and the same reduction proves uniformity
in $t$.

Suppose that $m\geq1$. We blow up along $Y=0$, $u=0$. In a chart
write
\[
 Y=xZ,\qquad u_\nu=xb_\nu,
\]
where one coordinate of $(Z,b)$ equals $1$ depending on in which 
chart we are working. The generator matrix
pulls back to
\begin{equation}\label{eq:partition-blowup-pullback}
 x^2\left[ZZ^\intercal+V(xb,t)
  \left(\sum_{\nu=1}^s\epsilon_\nu b_\nu^2P_\nu\right)
          V(xb,t)^\intercal\right].
\end{equation}
Taking into account both the Jacobian determinant and the original amplitude,
the new amplitude is, up to a non-vanishing factor,
\begin{equation}\label{eq:partition-blowup-amplitude}
 x^{qm+2s-1}\prod_{\nu=1}^sb_\nu.
\end{equation}
Where the expression in brackets in
\eqref{eq:partition-blowup-pullback} does not vanish on $x=0$, the
smallest pole is
\[
 \lambda_0=\frac{qm+2s}{2}=s+\frac{qm}{2}
\]
with multiplicity one.

Now suppose that the expression in brackets vanishes at a point of
the exceptional divisor defined by $x = 0$, for $t=0$. Put
$J=\{\nu:b_\nu\neq0\}$ and $D=d_J$. By
Lemma~\ref{lem:partition-generic-orthonormal}(2), $D\leq m$.
We also have $J\neq\varnothing$. Indeed, in a $u$-chart this is
immediate while in a $Y$-chart, $J=\varnothing$ would imply
$ZZ^\intercal=0$ and hence $Z=0$, contrary to the chart condition.
Furthermore, the columns of $V_J(0,0)$ are independent. As $ZZ^\intercal$
is positive semidefinite, the
equality
\[
 ZZ^\intercal=-V_J(0,0)
           \diag(\epsilon_\nu b_\nu^2\bl{I}_{d_\nu}:\nu\in J)
                          V_J(0,0)^\intercal
\]
implies that every $\epsilon_\nu$ with $\nu\in J$ equals $-1$
and $\rank Z=D$. We may, consequently, work in a $u_\nu$-chart
with $\nu\in J$ and treat the $b_\nu$ with $\nu\in J$ as
non-vanishing parameters.

Multiplying the matrix of generators by suitable invertible
analytic matrices on the left and right, and changing $Z$
accordingly, we may assume that the columns
\[
V = \begin{bmatrix}
    \bl{I}_D & R_1\\
    \bl{0} & R_2
\end{bmatrix},
\]
where the columns of the left blocks form $V_J$.
Since $\rank Z=D$, LQ-coordinates as in~\cite[Fact 3.5]{drton2025Factor} give
\[
 Z=\begin{bmatrix}P&0\\B&G\end{bmatrix},
\]
where $P$ is lower triangular with positive diagonal entries. Here
$B\in\R^{(q-D)\times D}$ and
$G\in\R^{(q-D)\times(m-D)}$. Put
\[
 \Delta_J(b)=\diag(b_\nu^2\bl{I}_{d_\nu}:\nu\in J),
 \qquad
 \Delta_R(b)=\diag(\epsilon_\nu b_\nu^2\bl{I}_{d_\nu}:\nu\notin J).
\]
The matrix in brackets in \eqref{eq:partition-blowup-pullback}
becomes
\begin{equation}\label{eq:partition-block-after-blowup}
 \begin{bmatrix}
 PP^\intercal-\Delta_J+R_1\Delta_RR_1^\intercal
   &PB^\intercal+R_1\Delta_RR_2^\intercal\\
 BP^\intercal+R_2\Delta_RR_1^\intercal
   &BB^\intercal+GG^\intercal+R_2\Delta_RR_2^\intercal
 \end{bmatrix}.
\end{equation}
We now apply the real analytic coordinate change
\begin{align*}
 X&=PP^\intercal-\Delta_J+R_1\Delta_RR_1^\intercal,\\
 Z_0&=BP^\intercal+R_2\Delta_RR_1^\intercal,
\end{align*}
leaving the other variables fixed. It is an isomorphism near our
point because $P$ is recovered by the analytic Cholesky
decomposition of $\Delta_J-R_1\Delta_RR_1^\intercal+X$, and then
$B$ is recovered from $Z_0$. Thus, the entries of $X$ and $Z_0$
are regular generators. Eliminating them from the remaining
entries leaves
\begin{equation}\label{eq:partition-residual}
 GG^\intercal+R_2\left[\Delta_R+
 \Delta_R R_1^\intercal
  (\Delta_J-R_1\Delta_RR_1^\intercal)^{-1}
      R_1\Delta_R\right]R_2^\intercal.
\end{equation}
The regular generators contribute
\begin{equation}\label{eq:partition-normal-count}
 N_D=\frac{D(D+1)}2+D(q-D)
      =Dq-\frac{D(D-1)}2
\end{equation}
to the first component of the RLCT.

Set
\[
 K=R_1^\intercal(\Delta_J-R_1\Delta_RR_1^\intercal)^{-1}R_1,
 \qquad T=(\bl{I}+\Delta_RK)^{1/2}.
\]
The analytic matrix square root is well defined in a neighborhood
of $\Delta_R=0$. As $K$ and $\Delta_R$ are symmetric, its
power series gives
\[
 T\Delta_RT^\intercal=\Delta_R+\Delta_RK\Delta_R.
\]
Consequently, \eqref{eq:partition-residual} equals
\[
 GG^\intercal+(R_2T)\Delta_R(R_2T)^\intercal.
\]
At $b_\nu=0$ for $\nu\notin J$ and $x=0$, the matrix $R_2T$
is the matrix of the remaining columns of $V(0,0)$ projected to
the quotient by $\operatorname{im}(V_J)$. By
Lemma~\ref{lem:partition-generic-orthonormal}(3), these columns
satisfy the genericity conditions for the smaller partition, with
$(p,q,m,s)$ replaced by $(p-D,q-D,m-D,s-|J|)
$.
The $b_\nu$ with $\nu \in J$ are analytic units in the amplitude. 
The induction hypothesis, including
its uniform assertion, therefore applies to the residual ideal.
It follows that the relevant first component before integrating
over $x$ is
\begin{align*}
 \lambda_J
 &=N_D+(s-|J|)+\frac{(q-D)(m-D)}2\\
 &=\lambda_0+\frac{D(q-m+1)}2-|J|.
\end{align*}
This number is strictly greater than $\lambda_0$. Indeed, if
$q-m\geq2$, then $D\geq|J|$ implies
$\lambda_J-\lambda_0\geq|J|/2>0$. If $q-m=1$,
assumption \eqref{eq:partition-codim} forces $s=1$ and its only
non-empty block has $D=p>m$, so the exceptional case under
consideration cannot occur.

We have thus shown that in a neighborhood of every point where the
proper transform meets the exceptional divisor, integration in
$(Z,b)$ is uniformly finite for $z<\lambda_J$, with
$\lambda_J>\lambda_0$. Hence, the remaining $x$-integral, using
\eqref{eq:partition-blowup-amplitude}, is finite for
$z<\lambda_0$, which implies the assertion.
\end{proof}

We are now ready for the proof of the main theorem on learning coefficients of partitioned noise models.

\begin{proof}[Proof of Theorem~\ref{thm:generalized-formula-learning-coeff}]
Write $\bs{\Sigma}_0=\bl{L}\bl{L}^\intercal+
\sum_{\nu=1}^s\alpha_\nu P_\nu$, where
$\bl{L}\in\R^{p\times r}$ is generic with respect to the conditions in Lemma~\ref{lem:partition-fiber-classification}
and $\alpha_\nu>0$. The $\alpha_\nu$ are constant within the
parts of $\bl{e}$ because $\bl{d}\preceq\bl{e}$.

First we compute the RLCT at a parameter point with loading matrix
of rank $r$ and noise parameters $\bs{\psi}=\bs{\alpha}$.
Choose $\bl{Q}\in\mathrm{O}(p)$ so that
\[
 \bl{E}_r:=\bl{Q}\bl{L}\bl{L}^\intercal\bl{Q}^\intercal
      =\diag(\eta_1,\ldots,\eta_r,0,\ldots,0),
       \qquad \eta_i>0.
\]
Multiplying the matrix of generators on the left by $\bl{Q}$ and
on the right by $\bl{Q}^\intercal$ does not change its ideal.
We divide the noise-parameter neighborhood into finitely many
orthants and in each orthant make the substitutions
\begin{equation}\label{eq:partition-root-substitution}
 \psi_\nu-\alpha_\nu=\epsilon_\nu u_\nu^2,
       \qquad\epsilon_\nu\in\{\pm1\}.
\end{equation}
These substitutions preserve the RLCT similarly as in the proof
of Theorem~\ref{thm:formula-learning-coeff-rlct} when the Jacobian determinant
$\prod_\nu u_\nu$ is included as an amplitude.
Write
\[
 \Delta(u)=\sum_{\nu=1}^s\epsilon_\nu u_\nu^2P_\nu,
 \qquad \bl{A}(u)=\bl{Q}\Delta(u)\bl{Q}^\intercal.
\]

By the LQ-decomposition as in~\cite[Fact 3.5]{drton2025Factor},
we can write the loading
matrix in the form
\[
 \bl{W}=\begin{bmatrix}
    \bl{W}_{11}&0\\ \bl{W}_{21}&\bl{W}_{22}
 \end{bmatrix},
\]
where $\bl{W}_{11}$ is an $r\times r$ lower triangular matrix
with positive diagonal entries and
$\bl{W}_{22}\in\R^{(p-r)\times(k-r)}$.
We divide $\bl{A}$ into blocks similarly, so the matrix of generators is
\begin{equation}\label{eq:partition-initial-blocks}
 \begin{bmatrix}
 \bl{W}_{11}\bl{W}_{11}^\intercal+A_{11}-\bl{E}_{r,11}
 &\bl{W}_{11}\bl{W}_{21}^\intercal+A_{21}^\intercal\\
 \bl{W}_{21}\bl{W}_{11}^\intercal+A_{21}
 &\bl{W}_{21}\bl{W}_{21}^\intercal+
        \bl{W}_{22}\bl{W}_{22}^\intercal+A_{22}
 \end{bmatrix}.
\end{equation}

We apply the real analytic coordinate change
\[
 \bl{Z}=\bl{W}_{21}\bl{W}_{11}^\intercal+A_{21}(u),
\]
leaving the other variables fixed. It is invertible because
$\bl{W}_{11}$ is invertible. Since all entries of $\bl{Z}$ are
now generators, we can eliminate terms containing $\bl{Z}$ from
the lower right block. Its remaining entries are those of
\[
 \bl{W}_{22}\bl{W}_{22}^\intercal+A_{22}+
    A_{21}(\bl{W}_{11}\bl{W}_{11}^\intercal)^{-1}
        A_{21}^\intercal.
\]
Next set
\[
 \bl{X}=\bl{W}_{11}\bl{W}_{11}^\intercal
                 +A_{11}-\bl{E}_{r,11}.
\]
This is again a real analytic isomorphism in the
$\bl{W}_{11}$-coordinates. Indeed, its inverse is the analytic Cholesky
decomposition of $\bl{E}_{r,11}-A_{11}+\bl{X}$. Since the entries
of $\bl{X}$ are generators, we may eliminate them from the
remaining generators. Thus, by the sum rule and elimination of
regular coordinates, $\bl{X}$ and $\bl{Z}$ contribute
\begin{equation}\label{eq:partition-regular-contribution}
 N_r=\frac{r(r+1)}2+(p-r)r
    =pr-\frac{r(r-1)}2
\end{equation}
to the first component of the RLCT, and the remaining generators
are the entries of
\begin{equation}\label{eq:partition-Schur-remainder}
 H(\bl{W}_{22},u)=\bl{W}_{22}\bl{W}_{22}^\intercal+
       A_{22}+A_{21}(\bl{E}_{r,11}-A_{11})^{-1}
                                     A_{21}^\intercal.
\end{equation}
Note that the same statements hold for $r=0$ after omitting the empty
blocks and setting $N_0=0$.

Denote by $U$ the first $r$ rows and by $V$ the last $p-r$ rows
of $\bl{Q}$. Setting
\[
 K(u)=U^\intercal
        (\bl{E}_{r,11}-U\Delta(u)U^\intercal)^{-1}U,
\]
we can rewrite \eqref{eq:partition-Schur-remainder} as
\[
 H(\bl{W}_{22},u)=\bl{W}_{22}\bl{W}_{22}^\intercal+
                 V\big(\Delta+\Delta K\Delta\big)V^\intercal.
\]
As $\Delta$ tends to zero, the matrix
$T(u)=(\bl{I}_p+\Delta(u)K(u))^{1/2}$ is real analytic and
satisfies
\[
 T\Delta T^\intercal=\Delta+\Delta K\Delta
\]
as we argued already in the proof of Lemma~\ref{lem:partition-uniform-blowup}.
Hence
\begin{equation}\label{eq:partition-final-central}
 H(\bl{W}_{22},u)=\bl{W}_{22}\bl{W}_{22}^\intercal+
                       (VT(u))\Delta(u)(VT(u))^\intercal.
\end{equation}
For generic $\bl{L}$, the orthonormal row matrix $V$, whose
rows span $\ker\bl{L}^\intercal$, satisfies
Lemma~\ref{lem:partition-generic-orthonormal} with
 $q=p-r$ and $m=k-r$.
Indeed, $q-m=p-k$ and our assumption is precisely
$\binom{q-m+1}{2}\geq s$.
We can therefore apply Lemma~\ref{lem:partition-uniform-blowup}
to \eqref{eq:partition-final-central}. With the Jacobian determinant
from \eqref{eq:partition-root-substitution}, it gives
\[
 \rlct_{(0,0)}\big(\langle H\rangle;\prod_\nu u_\nu \big)
     =\left(s+\frac{(p-r)(k-r)}2,1\right).
\]
Adding \eqref{eq:partition-regular-contribution}, we obtain
\begin{equation}\label{eq:partition-lambda-base}
 \lambda_0=N_r+s+\frac{(p-r)(k-r)}2
     =\frac{pk+r(p-k+1)+2s}{2}
\end{equation}
with multiplicity one. We now show that this is the smallest RLCT
on the entire fiber.

Fix any $(\bl{W}_*,\bs{\psi}_*)$ in the original fiber and put
\[
 \beta_\nu=\alpha_\nu-\psi_{*,\nu},
 \qquad
 \bl{W}_*\bl{W}_*^\intercal
     =\bl{L}\bl{L}^\intercal+\sum_\nu\beta_\nu P_\nu.
\]
Lemma~\ref{lem:partition-fiber-classification} leaves three cases.

\textbf{Case (1).} If all $\beta_\nu=0$, then
$\bs{\psi}_*=\bs{\alpha}$ and $\rank\bl{W}_*=r$.
The preceding computation applies (up to a right orthogonal
change of loading coordinates), giving $(\lambda_0,1)$.

\textbf{Case (2).} Let
$J=\{\nu:\beta_\nu\neq0\}$ be nonempty, and set
\[
 D=d_J,\qquad a=r+D,\qquad j=|J|.
\]
The classification lemma tells us that $D\leq k-r$,
$\beta_\nu>0$ for $\nu\in J$, and
$\rank\bl{W}_*=a$. We repeat the coordinate transformations
leading to \eqref{eq:partition-final-central}, now around
$\bl{W}_*$ with $r$ replaced by $a$ and with the noise centered
at $\bs{\psi}_*$.
Choose the diagonalizing orthogonal matrix so that its last
$p-a$ rows which we denote by say $V_*$ span the space
$ \operatorname{im}(\bl{L}\ \bl{D}_J)^\perp$, where $\bl{D}_J$ is 
the matrix coming from Lemma~\ref{lem:partition-fiber-classification}(2).

Because the columns of $\bl{D}_J$ are nonzero multiples of every
coordinate vector in the blocks in $J$, all columns of $V_*$ in
those blocks vanish. Write $V_*=(0\ \ V_{J^c})$ after grouping the
columns. The $(p-a)\times(p-D)$ matrix $V_{J^c}$ is a generic
orthonormal row matrix. Its row space is
$\ker(\bl{L}_{J^c})^\intercal$ and the remaining rows of
$\bl{L}$ are generic.

For $\nu\notin J$ make the signed square-root substitutions
$\psi_\nu-\psi_{*,\nu}=\epsilon_\nu u_\nu^2$.
For $\nu\in J$ leave the centered noise variables as ordinary
analytic parameters $t_\nu=\psi_\nu-\psi_{*,\nu}$.
In the reduction to the lower right block, the terms involving
$t_\nu$ vanish after projection by $V_*$, because
$V_*P_\nu=0$ for $\nu\in J$. However, they survive as analytic
parameters in the coefficient matrix $K$. The reduced generator
matrix is consequently of the form
\[
 \bl{W}_{22}\bl{W}_{22}^\intercal+
     V_{J^c}T(u,t)
       \left(\sum_{\nu\notin J}\epsilon_\nu u_\nu^2
                                       P_\nu^{J^c}\right)
             T(u,t)^\intercal V_{J^c}^\intercal,
\]
where $P_\nu^{J^c}$ is the projection onto the corresponding
remaining block, $T$ is analytic, and $T(0,t)=\bl{I}_{p-D}$.
Here $\bl{W}_{22}\in\R^{(p-a)\times(k-a)}$.
Since
\[
 \binom{(p-a)-(k-a)+1}{2}
    =\binom{p-k+1}{2}\geq s-j,
\]
we can apply the uniform assertion in
Lemma~\ref{lem:partition-uniform-blowup}, taking the variables
$t_\nu$ as auxiliary parameters. The regular coordinates give
$N_a=a(p-a)+a(a+1)/2$. Therefore, the RLCT at this point is
\begin{align*}
 \lambda_J
  &=N_a+(s-j)+\frac{(p-a)(k-a)}2\\
  &=\lambda_0+\frac{D(p-k+1)}2-j,
\end{align*}
with multiplicity one. It is strictly larger than $\lambda_0$.
Indeed, if $p-k\geq2$, use $D\geq j$, and if $p-k=1$, our hypothesis
forces $s=1$, in which case $D=p>k-r$ makes Case (2) impossible.

\textbf{Case (3).} Finally, suppose that all $s$ noise shifts
are nonzero and $c=s$. Lemma~\ref{lem:partition-fiber-classification}
gives $\rank\bl{W}_*=k$ and states that the parametrization
$\varphi_{k,\bl{d}}$ is regular at this point. Hence, the full
fiber ideal is locally generated by $p(p+1)/2$ independent
regular coordinates, so its RLCT equals
\[
 \left(\frac{p(p+1)}2,1\right).
\]
Using $\binom{p-k+1}{2}=s$, we compute
\[
 \frac{p(p+1)}2-\lambda_0
    =\frac{(k-r)(p-k+1)}2\geq0.
\]
Thus, also in this case, the local RLCT is not smaller than
$(\lambda_0,1)$.

We have found a point on the fiber with RLCT $(\lambda_0,1)$,
while every other point has at least this value in the RLCT order.
The locality property of the RLCT therefore proves the
asserted value of the RLCT. The learning coefficient is
half the first component of this ideal RLCT, proving
\eqref{eqn:LearningCoeffExpressionGeneral}.
\end{proof}


\section{High-Dimensional Consistency of the sBIC}
\label{sec:HighDimensional}

Bai, Choi, and Fujikoshi \citep{BaiAICBICConsistency}, hereafter referred to as BCF, outline conditions for the BIC to be asymptotically consistent in high-dimensional PPCA models. In this section, we extend these results to the sBIC, which reveals interesting differences: asymptotically, the sBIC can be proven to require a smaller signal than the BIC for the true model to be selected with probability one. 

The high-dimensional setup of BCF considers a sequence of PPCA problems, where the $n$th problem pertains to $n$ i.i.d.~observations  $\bl{X}_{1n},\dots,\bl{X}_{nn} \sim \mc{N}_{p_n}(\bl{0},\bs{\Sigma}_n)$, with all $\bl{X}_{in} \in \mb{R}^{p_n}$. The following assumptions are made on the dimensions $p_n$ and the eigenstructure of $\bs{\Sigma}_n$:
\begin{conditions}
    \item $p_n/n \rightarrow c \in (0,1)$. \label{condition:C1}
    \item The eigenvalues of the sequence of true covariance matrices $\bs{\Sigma}_n$ have the form
    \begin{align*}
        \underbrace{\sigma^2_n + \tau_{1n}^2 \geq \cdots \geq  \sigma^2_n + \tau_{rn}^2}_{r\text{ eigenvalues}} > \underbrace{\sigma^2_n = \cdots = \sigma^2_n}_{p_n - r \text{ eigenvalues }},
    \end{align*}
    where $r \geq 1$ is fixed and does not depend on $n$. \label{condition:C2}
\end{conditions}
We denote the PPCA models selected by the BIC and sBIC  when the sample size is $n$ respectively by $\hat{k}_{\mathrm{BIC},n}$ and $\hat{k}_{\mathrm{sBIC},n}$. The next theorem extends Theorem 3.2 of \citep{BaiAICBICConsistency}. 

\begin{theorem}
\label{thm:BICConsistency}
    Under conditions \eqref{condition:C1} and \eqref{condition:C2},
    \begin{itemize}
        \item[(i)] if $\liminf_n \frac{\tau_{rn}^2}{\sigma^2_n \log(n)} > c$, then $\mathrm{Pr}(\hat{k}_{\mathrm{BIC},n} = r) \rightarrow 1$;
        \item[(ii)]
     if $\limsup_n \frac{\tau_{rn}^2}{\sigma^2_n \log(n)} < c$, then $\mathrm{Pr}(\hat{k}_{\mathrm{BIC},n} = r) \rightarrow 0$. 
     \end{itemize}
\end{theorem}
According to this result, if the $r$th eigenvalue of $\bs{\Sigma}_n$ has the form $\sigma_n^2(1 + \beta \log(n))$ then the BIC is able to consistently estimate the true model $r$ if $\beta > c$, but does not consistently estimate the true model if $\beta < c$. More generally, if the gap between the $r$th and $(r+1)$st largest eigenvalues is greater than $\sigma^2_n \,\beta \log(n)$, the BIC remains consistent. Observe that this gap must increase with $n$ to account for the instability of the sample eigenvalues in high-dimensional settings.

The sBIC counterpart to Theorem \ref{thm:BICConsistency} is:

\begin{theorem}
\label{thm:sBICConsistency}
    Under conditions \eqref{condition:C1} and \eqref{condition:C2},
    \begin{itemize}
        \item[(i)] if $\liminf_n \frac{\tau_{rn}^2}{\sigma^2_n \log(n)} > \frac{c}{2}$, then $\mathrm{Pr}(\hat{k}_{\mathrm{sBIC},n} = r) \rightarrow 1$;
        \item[(ii)]
     if $\limsup_n \frac{\tau_{rn}^2}{\sigma^2_n \log(n)} < \frac{c}{2}$, then $\mathrm{Pr}(\hat{k}_{\mathrm{sBIC},n} = r) \rightarrow 0$. 
     \end{itemize}
\end{theorem}

The only difference between Theorem~\ref{thm:sBICConsistency} and Theorem~\ref{thm:BICConsistency} is the smaller detection threshold $\frac{c}{2}$. Hence, there are cases where the sBIC is consistent while the BIC is not. 

\begin{corollary}
    If the $r$th largest eigenvalue of $\bs{\Sigma}_n$ has the form $\sigma_n^2(1 + \beta \log(n))$ with $\beta \in (\frac{c}{2},c)$ then $\mathrm{Pr}(\hat{k}_{\mathrm{sBIC},n} = r) \rightarrow 1$ but $\mathrm{Pr}(\hat{k}_{\mathrm{BIC},n} = r) \rightarrow 0$. 
\end{corollary}

\begin{remark} 
\label{rmk:AIC}
The detection thresholds $c$ and $c/2$ are a consequence of the price for adding the $r$th principal component: $\frac{p_n-r+1}{2}\log(n)$ for the BIC, and the learning coefficient difference $(\lambda_{r,r-1}-\lambda_{r-1,r-1})\log(n)=\frac{p_n-r+1}{4}\log(n)$ for the sBIC. The improvement from $c$ to $\frac{c}{2}$, thus, stems from the marginal likelihood reflecting the singularities of the PPCA models. Other criteria with smaller penalties could detect weaker spikes in the high-dimensional regime. For example, BCF show that the AIC is consistent under a bounded gap condition. However, the AIC lacks a Bayesian justification and is not consistent for fixed $p$, whereas the sBIC is \citep{sBICDrtonPlummer}.
\end{remark}

A more involved proof is required for Theorem \ref{thm:sBICConsistency} as compared to Theorem \ref{thm:BICConsistency} since the sBIC is defined recursively. The proof involves an inductive argument that identifies the dominant terms in the expression \eqref{eqn:sBICLi}. Random matrix theory, following the approach of \citep{BaiAICBICConsistency}, is used to explicitly analyze the asymptotic behaviour of the maximized likelihoods, which are functions of the spectrum of the sample covariance matrix. Proofs of Theorems \ref{thm:BICConsistency} and \ref{thm:sBICConsistency} are provided in Appendix \ref{app:Proofs}.

\section{Simulations and Data Examples}
\label{sec:Simulations}

\subsection{Marginal Log-Likelihood Comparison}
Empirical confirmation for the formula provided in Theorem \ref{thm:formula-learning-coeff} is provided in this section in a low-dimensional ($p = 2$) case. The value of the marginal log-likelihood $\log\big( p(\bl{X}|\mc{M}_k)\big)$ is computed under the assumption that a model with one PC ($k = 1$) is used with the prior distribution $\bl{W} \sim \mc{N}(\bl{0},\bl{I}_2)$ and $\sigma^2 \sim \text{Exp}(1)$ independently. For sample sizes ranging from $50$ to $500$ the marginal log-likelihood is computed by Monte Carlo, where $7000$ draws from the prior are used to estimate the marginal likelihood integral \eqref{eqn:MarginalLikelihoodforPPCA}. From the asymptotic expansion \eqref{eqn:LogMargLikeAsymptotics} the marginal log-likelihood is approximately a linear function of the log-sample size, where the slope is given by the negative of the learning coefficient $\lambda(\bs{\Sigma}_0)$. In our simulations we consider two settings for the true covariance matrix from which the data are drawn: the first takes $\bs{\Sigma}_0 = \bl{I}_2 \in \mc{M}_0$ and the second $\bs{\Sigma}_0 = \diag(2,1) \in \mc{M}_1$. Theorem \ref{thm:formula-learning-coeff} asserts that the learning coefficient in the singular first case is equal to $1$, while the learning coefficient in the second case is equal to $\tfrac{3}{2}$. Note that the second case is a non-singular case where the learning coefficient is equal to the dimension of the model $\mc{M}_1$, equaling $\text{dim}(\mc{S}_{++}^2) = 3$, divided by two.   

Figure \ref{fig:MargLikelihoodVsSampleSize} displays the observed marginal log-likelihoods against the log sample sizes; more precisely, we also consider the difference of marginal log-likelihood and log-likelihood under the true distribution. In each panel, the fitted line is the line of best fit to the simulated values, with respective slopes of $-1.00$ and $-1.44$, while the second line has slope exactly $-\lambda(\bs{\Sigma}_0)$, namely $-1$ and $-1.5$. As is apparent in the plots, these two lines are nearly identical, illustrating that the theoretically justified learning coefficients can be observed in practice and that the marginal log-likelihood asymptotics differ in the singular and non-singular regimes.  
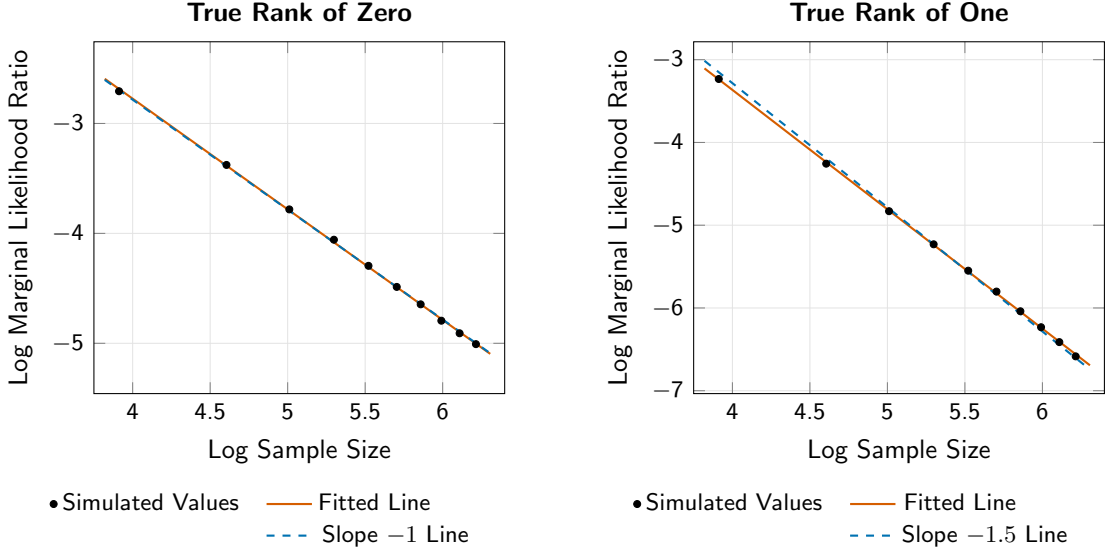
\begin{figure}
    \centering
\begin{tikzpicture}
\begin{groupplot}[
  group style={group size=2 by 1, horizontal sep=2.5cm},
  width=0.45\linewidth, height=0.40\linewidth,
  xmin=3.75, xmax=6.4,
  xtick={4.0,4.5,5.0,5.5,6.0},
  xticklabel style={/pgf/number format/fixed, /pgf/number format/precision=1},
  yticklabel style={/pgf/number format/fixed, /pgf/number format/precision=1},
  xlabel={Log Sample Size}, ylabel={Log Marginal Likelihood Ratio},
  legend columns=2,
  label style={font=\small\sffamily\sansmath},
  tick label style={font=\footnotesize\sffamily\sansmath},
  title style={font=\small\sffamily\bfseries, yshift=-2pt},
  grid=major, grid style={line width=.2pt, draw=gray!22},
  legend style={font=\footnotesize\sffamily, draw=none, fill=none,
                /tikz/every even column/.append style={column sep=8pt}},
]
\nextgroupplot[title={True Rank of Zero}, ymin=-5.4577, ymax=-2.2568, legend to name=leg:figtwoA]
\addplot[color=sbiccol, thick, forget plot] coordinates {(3.8199,-2.5935) (6.3067,-5.0962)};
\addplot[color=biccol, thick, dashed, forget plot] coordinates {(3.8199,-2.6037) (6.3067,-5.0906)};
\addplot[color=black, only marks, mark=*, mark size=1.3pt, forget plot] coordinates {(3.912,-2.7068) (4.6052,-3.3773) (5.0106,-3.7819) (5.2983,-4.0578) (5.5214,-4.2952) (5.7037,-4.4871) (5.858,-4.6455) (5.9914,-4.7945) (6.1092,-4.9089) (6.2146,-5.0077)};
\addlegendimage{color=black, only marks, mark=*, mark size=1.3pt}\addlegendentry{Simulated Values}
\addlegendimage{color=sbiccol, thick}\addlegendentry{Fitted Line \phantom{asd}}
\addlegendimage{empty legend}\addlegendentry{}
\addlegendimage{color=biccol, thick, dashed}\addlegendentry{Slope $-1$ Line}

\nextgroupplot[title={True Rank of One}, ymin=-7.0347, ymax=-2.7834, legend to name=leg:figtwoB]
\addplot[color=sbiccol, thick, forget plot] coordinates {(3.8199,-3.1053) (6.3067,-6.6928)};
\addplot[color=biccol, thick, dashed, forget plot] coordinates {(3.8199,-3.0133) (6.3067,-6.7436)};
\addplot[color=black, only marks, mark=*, mark size=1.3pt, forget plot] coordinates {(3.912,-3.2334) (4.6052,-4.2564) (5.0106,-4.8305) (5.2983,-5.2306) (5.5214,-5.5495) (5.7037,-5.8019) (5.858,-6.0399) (5.9914,-6.2326) (6.1092,-6.4125) (6.2146,-6.5847)};
\addlegendimage{color=black, only marks, mark=*, mark size=1.3pt}\addlegendentry{Simulated Values}
\addlegendimage{color=sbiccol, thick}\addlegendentry{Fitted Line \phantom{asdf}}
\addlegendimage{empty legend}\addlegendentry{}
\addlegendimage{color=biccol, thick, dashed}\addlegendentry{Slope $-1.5$ Line}
\end{groupplot}
\end{tikzpicture}

\vspace{1mm}
\makebox[\linewidth]{\hfill\ref*{leg:figtwoA}\hfill\hfill\ref*{leg:figtwoB}\hfill}
    \caption{Plot of the marginal log-likelihood ratio against the log-sample size for a model with $p = 2$, one PC $k = 1$, and true covariance matrix $\bs{\Sigma}_0 = \bl{I}_2$ (left) and $\bs{\Sigma}_0 = \diag(2,1)$ (right). }
    \label{fig:MargLikelihoodVsSampleSize}
\end{figure}

\subsection{Comparison of the sBIC to Other Approaches}
To begin our simulations comparing the sBIC to alternative methods, we illustrate in Figure \ref{fig:sBIcvsBICProportion} the proportion of times across $2000$ Monte Carlo runs that the sBIC and BIC correctly estimate the number of PCs of the underlying covariance matrix. This is done for data of dimension $p = 20$, a sample size of $n = 500$ from a Gaussian distribution, and the true number of PCs ranging from $0$ to $19$, where $\bs{\Sigma}_0 = \diag(c,\ldots,c,1,\ldots,1)$. From the figure it is readily apparent that the sBIC dominates the BIC in that it estimates the correct model at least as often as the BIC, and in many cases estimates the correct model with a much higher probability. When $c = 2$ and the true number of PCs is $9$, the sBIC found the correct model $75.6\%$ of the time while the BIC did so only $0.3\%$ of the time. Both model selection procedures tend to favor smaller, more parsimonious models. For example, when the signal to noise ratio (SNR) $c$ is equal to two, the sBIC is able to select the correct model with probability almost one when the number of PCs is in the range of $0$--$5$. When the number of PCs ranges between $15$--$19$, both methods can have difficulty selecting the correct model; in such a setting a larger sample size or a larger SNR is needed in order for these methods to find the correct model with high probability.
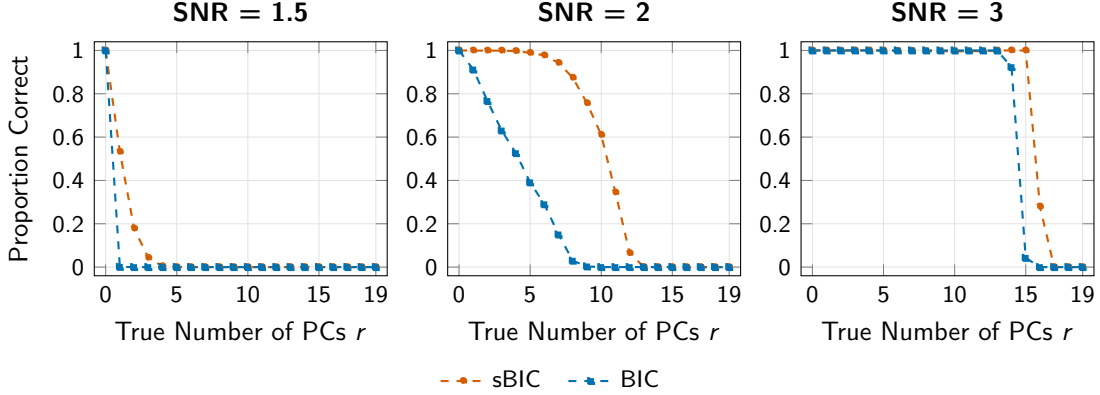
\begin{figure}[ht]
    \centering
\begin{tikzpicture}
\begin{groupplot}[
  group style={group size=3 by 1, horizontal sep=0.8cm, ylabels at=edge left},
  width=0.35\linewidth, height=0.30\linewidth,
  xmin=-0.8, xmax=19.8, ymin=-0.04, ymax=1.04,
  xtick={0,5,10,15,19}, ytick={0,0.2,0.4,0.6,0.8,1.0},
  yticklabel style={/pgf/number format/fixed, /pgf/number format/precision=1},
  xlabel={True Number of PCs $r$}, ylabel={Proportion Correct},
  legend columns=2,
  label style={font=\small\sffamily\sansmath},
  tick label style={font=\footnotesize\sffamily\sansmath},
  title style={font=\small\sffamily\bfseries, yshift=-2pt},
  grid=major, grid style={line width=.2pt, draw=gray!22},
  legend style={font=\footnotesize\sffamily, draw=none, fill=none,
                /tikz/every even column/.append style={column sep=8pt}},
]
\nextgroupplot[title={SNR = 1.5}, legend to name=leg:figthree]
\addplot[color=sbiccol, mark=*, mark size=1.1pt, thick, dashed] coordinates {(0,0.9995) (1,0.532) (2,0.1785) (3,0.043) (4,0.004) (5,0) (6,0) (7,0) (8,0) (9,0) (10,0) (11,0) (12,0) (13,0) (14,0) (15,0) (16,0) (17,0) (18,0) (19,0)};
\addplot[color=biccol, mark=square*, mark size=1.1pt, thick, dashed] coordinates {(0,1) (1,0.001) (2,0) (3,0) (4,0) (5,0) (6,0) (7,0) (8,0) (9,0) (10,0) (11,0) (12,0) (13,0) (14,0) (15,0) (16,0) (17,0) (18,0) (19,0)};
\legend{sBIC,BIC}

\nextgroupplot[title={SNR = 2}]
\addplot[color=sbiccol, mark=*, mark size=1.1pt, thick, dashed] coordinates {(0,1) (1,0.9995) (2,0.999) (3,0.999) (4,0.997) (5,0.9885) (6,0.977) (7,0.9425) (8,0.874) (9,0.7565) (10,0.61) (11,0.3445) (12,0.0645) (13,0.0005) (14,0) (15,0) (16,0) (17,0) (18,0) (19,0)};
\addplot[color=biccol, mark=square*, mark size=1.1pt, thick, dashed] coordinates {(0,1) (1,0.912) (2,0.7665) (3,0.6305) (4,0.526) (5,0.391) (6,0.29) (7,0.1505) (8,0.0295) (9,0.003) (10,0) (11,0) (12,0) (13,0) (14,0) (15,0) (16,0) (17,0) (18,0) (19,0)};

\nextgroupplot[title={SNR = 3}]
\addplot[color=sbiccol, mark=*, mark size=1.1pt, thick, dashed] coordinates {(0,0.9995) (1,1) (2,0.999) (3,1) (4,1) (5,0.9995) (6,0.998) (7,0.998) (8,0.9965) (9,0.997) (10,0.9955) (11,0.9965) (12,0.9965) (13,0.9995) (14,1) (15,0.9995) (16,0.2795) (17,0) (18,0) (19,0)};
\addplot[color=biccol, mark=square*, mark size=1.1pt, thick, dashed] coordinates {(0,1) (1,1) (2,1) (3,1) (4,1) (5,1) (6,1) (7,1) (8,1) (9,1) (10,1) (11,1) (12,1) (13,0.9995) (14,0.923) (15,0.0425) (16,0) (17,0) (18,0) (19,0)};
\end{groupplot}
\end{tikzpicture}

\vspace{1mm}
\ref*{leg:figthree}
    \caption{The proportion of simulation runs where the sBIC and BIC correctly estimate the true number of PCs. The true covariance matrices have the form $\diag(c,\ldots,c,1,\ldots,1)$ with a signal to noise ratio (SNR) of $c$.}
    \label{fig:sBIcvsBICProportion}
\end{figure}

In our second comparison study of model selection methods we contrast the performance of the sBIC, BIC, two variants of the normal-Gamma (NG) method of \citep{BouveyronBayesianPCA}, and the generalized cross-validation (GCV) method from \citep{GenCrossValidJosseHusson}. Across the simulations, the dimension of $\bs{\Sigma}_0$ is taken to be $p = 5$ or $p = 20$ and the true number of PCs ranges from $0$ to $p-1$. Two different covariance structures are considered: the first is an isotropic PPCA model \citep{bouveyron2011IsotropicPPCA} of the form $\bs{\Sigma}_0 = \diag(5,\ldots,5,1,\ldots,1)$ where the PC eigenvalues are all equal, the second has the true covariance matrix $\bs{\Sigma}_0 = \diag(r+1,\ldots,2,1,\ldots,1)$, with the PC eigenvalues increasing linearly. The sample size takes values $n \in \{30,50,250,4000\}$. For each of these parameter settings $1000$ simulations were run and for each method the proportion of the runs that the true model was correctly estimated was recorded, along with the absolute distance between the true number of PCs and the estimated number of PCs. The proportions obtained for the linearly increasing eigenvalues with $p = 20$ are displayed in Figure \ref{fig:LinearEigenvalueComparisonPlot} and discussed below; the results for the remaining settings are tabulated in Appendix \ref{app:SimulationResults}. There, the proportions of correct selection are reported in Tables \ref{tab:PropIsotropicP20} and \ref{tab:PropLinearP20} for $p = 20$ and in Table \ref{tab:PropP5} for $p = 5$, while the corresponding average absolute distances are reported in Tables \ref{tab:DistIsotropicP20} and \ref{tab:DistLinearP20} for $p = 20$ and in Table \ref{tab:DistP5} for $p = 5$.  

Figure \ref{fig:LinearEigenvalueComparisonPlot} illustrates the proportion of the time that the various methods correctly estimate the true model when $p = 20$ and the PC eigenvalues are linearly increasing. For small sample sizes both the BIC and the sBIC tend to select small models while the other three methods select from a wider spectrum of model sizes and thus are more likely to find the correct model when the true number of PCs is large. The BIC and sBIC exhibit a kind of thresholding behavior where the probability that the sBIC selects a large model will jump from approximately zero to one once the sample size is large enough (see Table \ref{tab:PropP5}). In the $n = 250$ plot in Figure \ref{fig:LinearEigenvalueComparisonPlot} it is seen that the sBIC outperforms the other methods. At a sample size of $4000$ both the sBIC and BIC find the correct model essentially $100\%$ of the time.  Although the GCV performs reasonably in the isotropic covariance setting (Table \ref{tab:PropIsotropicP20}), it breaks down in the linearly increasing eigenvalue setting and almost never selects larger models, even when $n = 4000$. The NG approach is based on approximating the marginal likelihood of a PPCA model with a prior on $\bs{\Sigma}$ given by
\begin{align*}
    \bs{\Sigma} = \bl{W}\bl{W}^\intercal + \sigma^2 \bl{I}_p, \;\; w_{ij} \overset{i.i.d.}{\sim} \mc{N}(0,\phi^{-1}),\;\; \sigma^2 \sim \text{Gamma}(a,\phi/2). 
\end{align*}
The NG1 and NG2 approaches differ by how we choose the hyperparameters $(a,\phi)$ of the prior. A favorable choice of hyperparameters is used in NG1 where we set 
\begin{align*}
    \phi = \frac{p(r+0.1)}{1.05(\text{tr}(\bs{\Sigma}_0)-p)}, \; a = \frac{\phi}{2}, 
\end{align*}
with $\bs{\Sigma}_0$ and $r$ equaling the true underlying covariance matrix and number of PCs. In practice, these quantities would not be known. This choice of hyperparameters ensures that marginally $\mb{E}(\sigma^2) = 1$ and $\mb{E}(\tr( \bl{W}\bl{W}^\intercal + \sigma^2 \bl{I}_p)) \approx \tr(\bs{\Sigma}_0)$, meaning that the prior is concentrated near the true covariance matrix from which we are drawing observations. The NG2 hyperparameters are the same as those in NG1 except that $r$ is taken to be equal to $p-1$; the idea is to choose a prior that is concentrated near models that have a large number of PCs. In Figure \ref{fig:LinearEigenvalueComparisonPlot} NG1 performs well for small sample sizes. This can be attributed to the fact that by construction the prior is centered about the true covariance matrix. However, NG2 performs poorly except for large and small models, even when the sample size is $4000$. This highlights the sensitivity of the NG method to the choice of prior hyperparameters. Another issue with the NG approach is that the marginal likelihood $\mb{E}_{\bs{\Sigma}}(p(\bl{X}_1,\ldots,\bl{X}_n|\bs{\Sigma}))$ is not exactly computed, instead the approximation $\prod_{i = 1}^n \mb{E}_{\bs{\Sigma}}(p(\bl{X}_i|\bs{\Sigma}))$ to the marginal likelihood that assumes independence is used for model selection. Asymptotically this criterion will select the model $\argmax_{k} \mb{E}_{\bl{X}_1}\big(\log(p_k(\bl{X}_1))\big)$, where $p_k(\bl{X}_1) = \mb{E}_{\bs{\Sigma}}(p(\bl{X}_1|\bs{\Sigma}))$ is the marginal likelihood of a single observation from the $k$th model $\mc{M}_k$ with respect to the normal-Gamma prior. That is, this criterion asymptotically selects the model $\mc{M}_k$ with the minimum KL-divergence between $p_k(x)$ and $p(x|\bs{\Sigma}_0)$. There is no guarantee that such a procedure is consistent, with the consistency depending heavily on the choice of prior. Moreover, the influence from the prior distribution does not diminish as the sample size increases. From the simulation results we conclude that while the NG method can perform adequately when the prior is chosen judiciously, the performance of this procedure is suspect for large sample sizes.

In summary, we recommend using the sBIC as it is more sensitive in detecting large models than the BIC. For small sample sizes the sBIC remains conservative, favoring parsimonious models. Other competing methods perform erratically in simulations and for large sample sizes are outperformed by the sBIC; the sBIC is able to accurately find the correct model regardless of the eigenstructure of $\bs{\Sigma}_0$.

\begin{figure}[ht]
    \centering
\begin{tikzpicture}
\begin{groupplot}[
  group style={group size=2 by 2, horizontal sep=1.3cm, vertical sep=1.5cm,
               xlabels at=edge bottom, ylabels at=edge left},
  width=0.505\linewidth, height=0.40\linewidth,
  xmin=-0.6, xmax=19.6, ymin=-0.04, ymax=1.04,
  xtick={0,5,10,15,19}, ytick={0,0.2,0.4,0.6,0.8,1.0},
  yticklabel style={/pgf/number format/fixed, /pgf/number format/precision=1},
  xlabel={True Number of PCs $r$}, ylabel={Proportion Correct},
  label style={font=\small\sffamily\sansmath},
  tick label style={font=\footnotesize\sffamily\sansmath},
  title style={font=\small\sffamily\bfseries, yshift=-2pt},
  grid=major, grid style={line width=.2pt, draw=gray!22},
  legend style={font=\footnotesize\sffamily, draw=none, fill=none, legend columns=5,
                /tikz/every even column/.append style={column sep=8pt}},
]
\nextgroupplot[title={n = 30}, legend to name=grouplegend]
\addplot[color=sbiccol, mark=*, mark size=1.1pt, thick, dashed] coordinates {(0,0.80) (1,0.31) (2,0.21) (3,0.20) (4,0.14) (5,0.10) (6,0.08) (7,0.09) (8,0.06) (9,0.04) (10,0.02) (11,0.02) (12,0.01) (13,0.00) (14,0.00) (15,0.00) (16,0.00) (17,0.00) (18,0.00) (19,0.00)};
\addplot[color=biccol, mark=square*, mark size=1.1pt, thick, dashed] coordinates {(0,1.00) (1,0.00) (2,0.00) (3,0.00) (4,0.00) (5,0.00) (6,0.00) (7,0.00) (8,0.00) (9,0.00) (10,0.00) (11,0.00) (12,0.00) (13,0.00) (14,0.00) (15,0.00) (16,0.00) (17,0.00) (18,0.00) (19,0.00)};
\addplot[color=ngonecol, mark=triangle*, mark size=1.1pt, thick, dashed] coordinates {(0,0.00) (1,0.39) (2,0.41) (3,0.43) (4,0.45) (5,0.46) (6,0.46) (7,0.46) (8,0.43) (9,0.40) (10,0.40) (11,0.37) (12,0.35) (13,0.35) (14,0.33) (15,0.32) (16,0.32) (17,0.31) (18,0.31) (19,0.61)};
\addplot[color=ngtwocol, mark=diamond*, mark size=1.1pt, thick, dashed] coordinates {(0,0.49) (1,0.02) (2,0.01) (3,0.00) (4,0.00) (5,0.00) (6,0.00) (7,0.00) (8,0.00) (9,0.00) (10,0.00) (11,0.00) (12,0.00) (13,0.00) (14,0.00) (15,0.01) (16,0.05) (17,0.14) (18,0.25) (19,0.61)};
\addplot[color=gcvcol, mark=pentagon*, mark size=1.1pt, thick, dashed] coordinates {(0,0.99) (1,0.04) (2,0.16) (3,0.22) (4,0.20) (5,0.14) (6,0.14) (7,0.15) (8,0.12) (9,0.11) (10,0.10) (11,0.08) (12,0.07) (13,0.07) (14,0.06) (15,0.03) (16,0.03) (17,0.02) (18,0.01) (19,0.01)};
\legend{sBIC,BIC,NG1,NG2,GCV}

\nextgroupplot[title={n = 50}]
\addplot[color=sbiccol, mark=*, mark size=1.1pt, thick, dashed] coordinates {(0,0.94) (1,0.36) (2,0.25) (3,0.19) (4,0.19) (5,0.14) (6,0.14) (7,0.11) (8,0.09) (9,0.06) (10,0.06) (11,0.03) (12,0.02) (13,0.01) (14,0.00) (15,0.00) (16,0.00) (17,0.00) (18,0.00) (19,0.00)};
\addplot[color=biccol, mark=square*, mark size=1.1pt, thick, dashed] coordinates {(0,1.00) (1,0.00) (2,0.00) (3,0.00) (4,0.00) (5,0.00) (6,0.00) (7,0.00) (8,0.00) (9,0.00) (10,0.00) (11,0.01) (12,0.00) (13,0.01) (14,0.00) (15,0.00) (16,0.00) (17,0.00) (18,0.00) (19,0.00)};
\addplot[color=ngonecol, mark=triangle*, mark size=1.1pt, thick, dashed] coordinates {(0,0.00) (1,0.43) (2,0.46) (3,0.47) (4,0.54) (5,0.56) (6,0.56) (7,0.52) (8,0.53) (9,0.53) (10,0.51) (11,0.45) (12,0.47) (13,0.44) (14,0.42) (15,0.40) (16,0.40) (17,0.38) (18,0.37) (19,0.63)};
\addplot[color=ngtwocol, mark=diamond*, mark size=1.1pt, thick, dashed] coordinates {(0,0.51) (1,0.02) (2,0.00) (3,0.00) (4,0.00) (5,0.00) (6,0.00) (7,0.00) (8,0.00) (9,0.00) (10,0.00) (11,0.00) (12,0.00) (13,0.00) (14,0.00) (15,0.00) (16,0.02) (17,0.10) (18,0.30) (19,0.63)};
\addplot[color=gcvcol, mark=pentagon*, mark size=1.1pt, thick, dashed] coordinates {(0,0.99) (1,0.10) (2,0.29) (3,0.28) (4,0.26) (5,0.27) (6,0.25) (7,0.21) (8,0.20) (9,0.17) (10,0.17) (11,0.17) (12,0.14) (13,0.13) (14,0.12) (15,0.10) (16,0.06) (17,0.04) (18,0.00) (19,0.00)};

\nextgroupplot[title={n = 250}]
\addplot[color=sbiccol, mark=*, mark size=1.1pt, thick, dashed] coordinates {(0,1.00) (1,0.95) (2,0.96) (3,0.96) (4,0.99) (5,0.98) (6,0.98) (7,0.99) (8,0.99) (9,0.99) (10,0.99) (11,0.99) (12,0.99) (13,0.99) (14,0.98) (15,0.98) (16,0.97) (17,0.96) (18,0.96) (19,0.78)};
\addplot[color=biccol, mark=square*, mark size=1.1pt, thick, dashed] coordinates {(0,1.00) (1,0.18) (2,0.20) (3,0.26) (4,0.31) (5,0.39) (6,0.45) (7,0.54) (8,0.62) (9,0.69) (10,0.73) (11,0.82) (12,0.85) (13,0.91) (14,0.92) (15,0.96) (16,0.97) (17,0.98) (18,0.99) (19,0.98)};
\addplot[color=ngonecol, mark=triangle*, mark size=1.1pt, thick, dashed] coordinates {(0,0.00) (1,0.33) (2,0.49) (3,0.57) (4,0.69) (5,0.75) (6,0.82) (7,0.87) (8,0.89) (9,0.89) (10,0.87) (11,0.85) (12,0.82) (13,0.79) (14,0.78) (15,0.77) (16,0.74) (17,0.71) (18,0.72) (19,0.75)};
\addplot[color=ngtwocol, mark=diamond*, mark size=1.1pt, thick, dashed] coordinates {(0,0.53) (1,0.01) (2,0.00) (3,0.00) (4,0.00) (5,0.00) (6,0.00) (7,0.00) (8,0.00) (9,0.00) (10,0.00) (11,0.00) (12,0.00) (13,0.00) (14,0.00) (15,0.00) (16,0.00) (17,0.01) (18,0.28) (19,0.75)};
\addplot[color=gcvcol, mark=pentagon*, mark size=1.1pt, thick, dashed] coordinates {(0,1.00) (1,0.21) (2,0.32) (3,0.32) (4,0.31) (5,0.31) (6,0.29) (7,0.28) (8,0.27) (9,0.25) (10,0.23) (11,0.20) (12,0.19) (13,0.16) (14,0.13) (15,0.10) (16,0.07) (17,0.03) (18,0.00) (19,0.00)};

\nextgroupplot[title={n = 4000}]
\addplot[color=sbiccol, mark=*, mark size=1.1pt, thick, dashed] coordinates {(0,1.00) (1,1.00) (2,1.00) (3,1.00) (4,1.00) (5,1.00) (6,1.00) (7,1.00) (8,1.00) (9,1.00) (10,1.00) (11,1.00) (12,1.00) (13,1.00) (14,1.00) (15,1.00) (16,1.00) (17,0.99) (18,0.99) (19,1.00)};
\addplot[color=biccol, mark=square*, mark size=1.1pt, thick, dashed] coordinates {(0,1.00) (1,1.00) (2,1.00) (3,1.00) (4,1.00) (5,1.00) (6,1.00) (7,1.00) (8,1.00) (9,1.00) (10,1.00) (11,1.00) (12,1.00) (13,1.00) (14,1.00) (15,1.00) (16,1.00) (17,1.00) (18,1.00) (19,1.00)};
\addplot[color=ngonecol, mark=triangle*, mark size=1.1pt, thick, dashed] coordinates {(0,0.00) (1,0.05) (2,0.48) (3,0.76) (4,0.97) (5,1.00) (6,1.00) (7,1.00) (8,1.00) (9,1.00) (10,1.00) (11,1.00) (12,1.00) (13,1.00) (14,1.00) (15,1.00) (16,1.00) (17,1.00) (18,1.00) (19,1.00)};
\addplot[color=ngtwocol, mark=diamond*, mark size=1.1pt, thick, dashed] coordinates {(0,0.56) (1,0.00) (2,0.00) (3,0.00) (4,0.00) (5,0.00) (6,0.00) (7,0.00) (8,0.00) (9,0.00) (10,0.00) (11,0.00) (12,0.00) (13,0.00) (14,0.00) (15,0.00) (16,0.00) (17,0.00) (18,0.02) (19,1.00)};
\addplot[color=gcvcol, mark=pentagon*, mark size=1.1pt, thick, dashed] coordinates {(0,1.00) (1,0.12) (2,0.11) (3,0.09) (4,0.08) (5,0.08) (6,0.07) (7,0.04) (8,0.03) (9,0.01) (10,0.02) (11,0.01) (12,0.00) (13,0.00) (14,0.00) (15,0.00) (16,0.00) (17,0.00) (18,0.00) (19,0.00)};
\end{groupplot}
\end{tikzpicture}

    \vspace{1mm}
    \ref*{grouplegend}
    \caption{A comparison of the proportion of $1000$ simulations where the correct model is identified for various model selection procedures. The covariance structure has the form $\bs{\Sigma}_0 = \diag(r+1,r,r-1,\ldots,1,\ldots,1) \in \mb{R}^{20 \times 20}$, with $r$ ranging from $0$ to $19$ ($x$-axes). Each plot has a different sample size of $n \in \{30,50,250,4000\}$.
    }
    \label{fig:LinearEigenvalueComparisonPlot}
\end{figure}
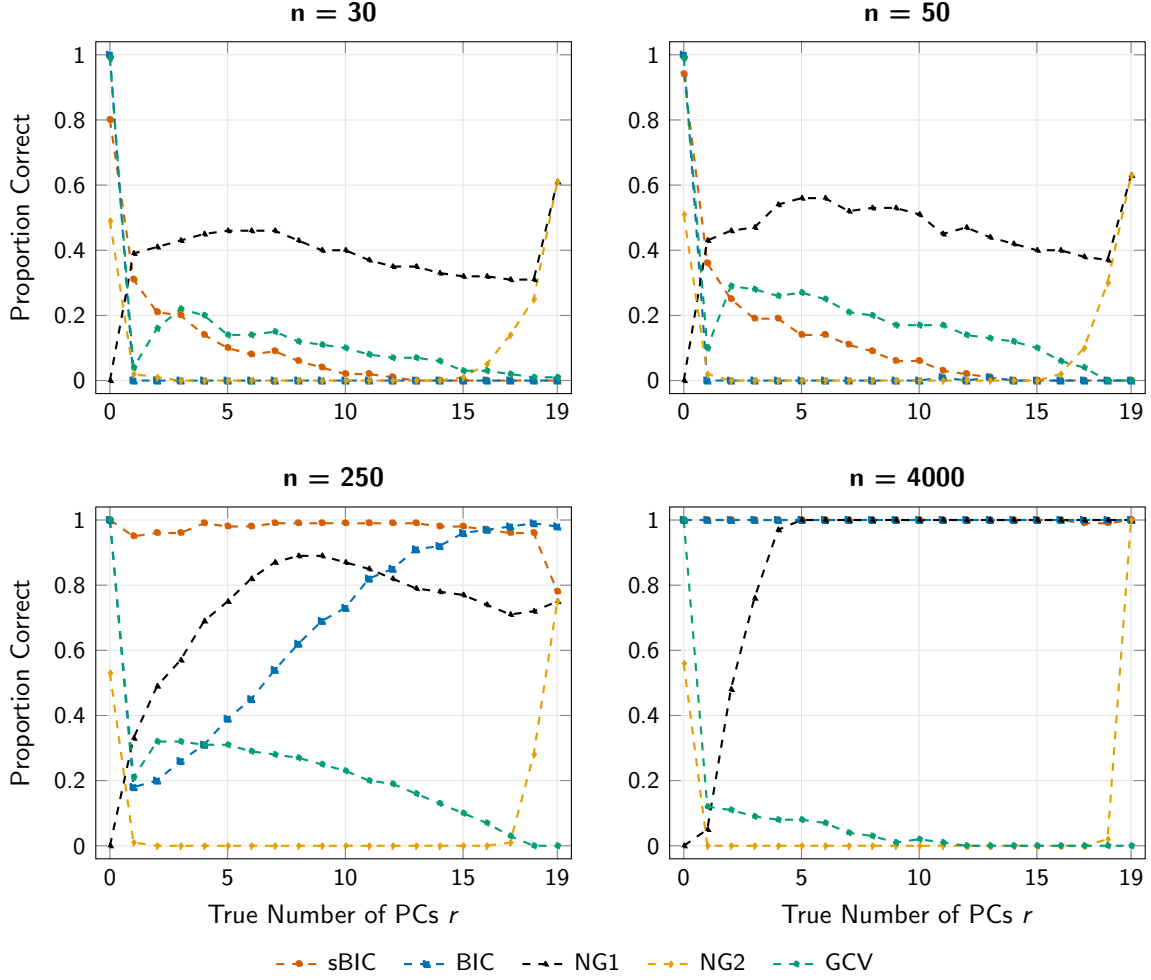

\subsection{Applications of PPCA and Factor Analysis}
\label{sec:Applications}

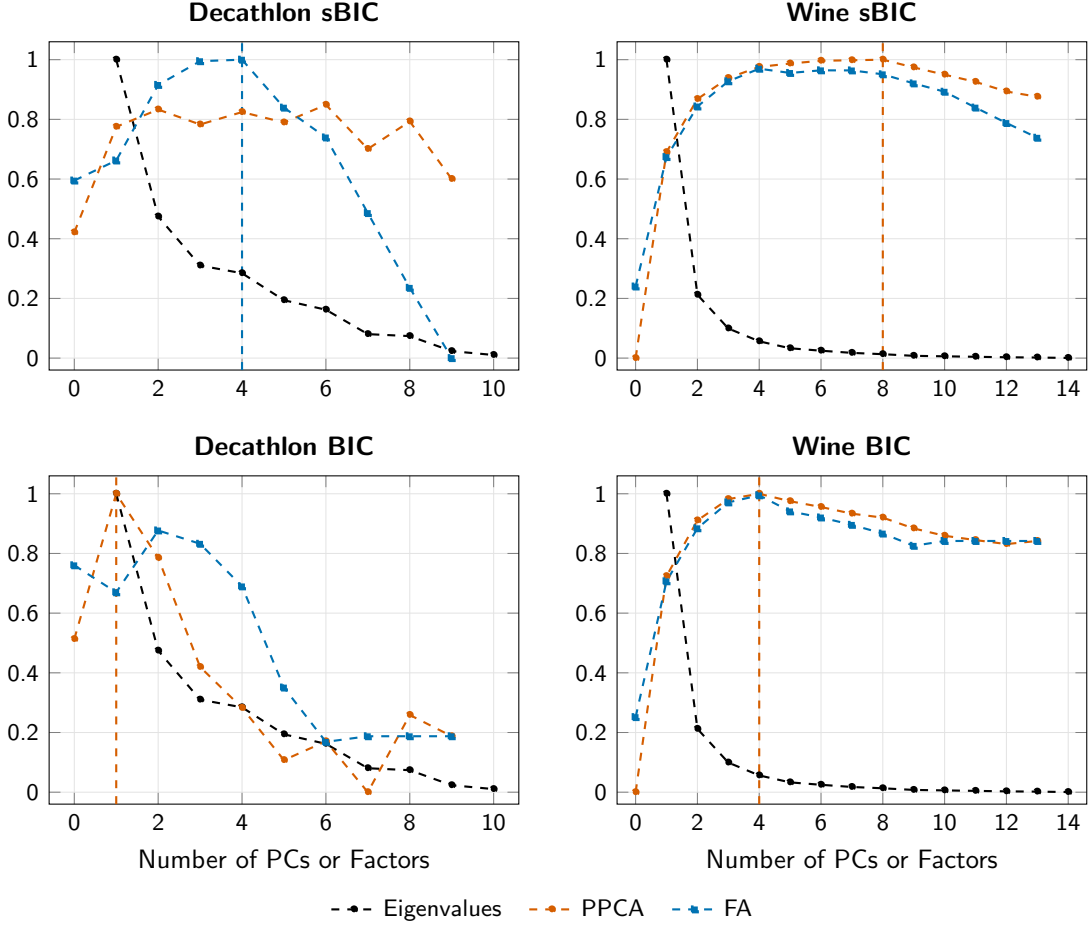
\begin{figure}[ht]
    \centering
\begin{tikzpicture}
\begin{groupplot}[
  group style={group size=2 by 2, horizontal sep=1.3cm, vertical sep=1.4cm,
               xlabels at=edge bottom, ylabels at=edge left},
  width=0.50\linewidth, height=0.38\linewidth,
  xmin=-0.6, ymin=-0.04, ymax=1.06,
  ytick={0,0.2,0.4,0.6,0.8,1.0},
  yticklabel style={/pgf/number format/fixed, /pgf/number format/precision=1},
  xlabel={Number of PCs or Factors}, ylabel={},
  legend columns=3,
  label style={font=\small\sffamily\sansmath},
  tick label style={font=\footnotesize\sffamily\sansmath},
  title style={font=\small\sffamily\bfseries, yshift=-2pt},
  grid=major, grid style={line width=.2pt, draw=gray!22},
  legend style={font=\footnotesize\sffamily, draw=none, fill=none,
                /tikz/every even column/.append style={column sep=8pt}},
]
\nextgroupplot[title={Decathlon sBIC}, xmax=10.6, xtick={0,2,4,6,8,10}, ylabel={}, legend to name=leg:figfive]
\draw[biccol, dashed, thick] ({axis cs:4,0}|-{rel axis cs:0,0}) -- ({axis cs:4,0}|-{rel axis cs:0,1});
\addplot[color=black, mark=*, mark size=1.1pt, thick, dashed] coordinates {(1,1) (2,0.4744) (3,0.3101) (4,0.2844) (5,0.1939) (6,0.1621) (7,0.0803) (8,0.0737) (9,0.0233) (10,0.0104)};
\addplot[color=sbiccol, mark=*, mark size=1.1pt, thick, dashed] coordinates {(0,0.4212) (1,0.775) (2,0.8329) (3,0.7828) (4,0.8244) (5,0.7895) (6,0.8488) (7,0.7008) (8,0.793) (9,0.6002)};
\addplot[color=biccol, mark=square*, mark size=1.1pt, thick, dashed] coordinates {(0,0.595) (1,0.6622) (2,0.9154) (3,0.9949) (4,1) (5,0.8393) (6,0.7394) (7,0.486) (8,0.2355) (9,0)};
\legend{Eigenvalues,PPCA,FA}

\nextgroupplot[title={Wine sBIC}, xmax=14.6, xtick={0,2,4,6,8,10,12,14}]
\draw[sbiccol, dashed, thick] ({axis cs:8,0}|-{rel axis cs:0,0}) -- ({axis cs:8,0}|-{rel axis cs:0,1});
\addplot[color=black, mark=*, mark size=1.1pt, thick, dashed] coordinates {(1,1) (2,0.2122) (3,0.0993) (4,0.0562) (5,0.0329) (6,0.0248) (7,0.0175) (8,0.0131) (9,0.0076) (10,0.006) (11,0.0044) (12,0.0028) (13,0.0019) (14,0.0005)};
\addplot[color=sbiccol, mark=*, mark size=1.1pt, thick, dashed] coordinates {(0,0) (1,0.6902) (2,0.8681) (3,0.9385) (4,0.9758) (5,0.9871) (6,0.9963) (7,0.998) (8,1) (9,0.974) (10,0.9497) (11,0.9256) (12,0.8934) (13,0.8757)};
\addplot[color=biccol, mark=square*, mark size=1.1pt, thick, dashed] coordinates {(0,0.2399) (1,0.6736) (2,0.8431) (3,0.9276) (4,0.9697) (5,0.9555) (6,0.9643) (7,0.964) (8,0.951) (9,0.9205) (10,0.8926) (11,0.8395) (12,0.7878) (13,0.7384)};

\nextgroupplot[title={Decathlon BIC}, xmax=10.6, xtick={0,2,4,6,8,10}, ylabel={}]
\draw[sbiccol, dashed, thick] ({axis cs:1,0}|-{rel axis cs:0,0}) -- ({axis cs:1,0}|-{rel axis cs:0,1});
\addplot[color=black, mark=*, mark size=1.1pt, thick, dashed] coordinates {(1,1) (2,0.4744) (3,0.3101) (4,0.2844) (5,0.1939) (6,0.1621) (7,0.0803) (8,0.0737) (9,0.0233) (10,0.0104)};
\addplot[color=sbiccol, mark=*, mark size=1.1pt, thick, dashed] coordinates {(0,0.5131) (1,1) (2,0.7856) (3,0.4196) (4,0.2826) (5,0.1075) (6,0.1705) (7,0) (8,0.2588) (9,0.1875)};
\addplot[color=biccol, mark=square*, mark size=1.1pt, thick, dashed] coordinates {(0,0.7603) (1,0.6698) (2,0.8771) (3,0.8324) (4,0.6897) (5,0.3507) (6,0.1686) (7,0.1875) (8,0.1875) (9,0.1875)};

\nextgroupplot[title={Wine BIC}, xmax=14.6, xtick={0,2,4,6,8,10,12,14}]
\draw[sbiccol, dashed, thick] ({axis cs:4,0}|-{rel axis cs:0,0}) -- ({axis cs:4,0}|-{rel axis cs:0,1});
\addplot[color=black, mark=*, mark size=1.1pt, thick, dashed] coordinates {(1,1) (2,0.2122) (3,0.0993) (4,0.0562) (5,0.0329) (6,0.0248) (7,0.0175) (8,0.0131) (9,0.0076) (10,0.006) (11,0.0044) (12,0.0028) (13,0.0019) (14,0.0005)};
\addplot[color=sbiccol, mark=*, mark size=1.1pt, thick, dashed] coordinates {(0,0) (1,0.724) (2,0.9107) (3,0.9821) (4,1) (5,0.975) (6,0.9556) (7,0.9333) (8,0.92) (9,0.8838) (10,0.8592) (11,0.8448) (12,0.8314) (13,0.8418)};
\addplot[color=biccol, mark=square*, mark size=1.1pt, thick, dashed] coordinates {(0,0.2516) (1,0.7067) (2,0.8844) (3,0.9711) (4,0.9948) (5,0.941) (6,0.9206) (7,0.8959) (8,0.8667) (9,0.8256) (10,0.8418) (11,0.8418) (12,0.8418) (13,0.8418)};
\end{groupplot}
\end{tikzpicture}

\vspace{1mm}
\ref*{leg:figfive}
    \caption{Plots of the eigenvalue ratios of the sample covariance matrix, $t_i/t_{\max}$, in decreasing order for each dataset. The sBIC values for the PPCA and factor analysis (FA) models are plotted against the number of PCs or factors, and are normalized as $(\text{sBIC}_i - \min(\text{sBIC}))/(\max(\text{sBIC}) - \min(\text{sBIC}))$. The maximum sBIC is indicated by a vertical dashed line in the style of the corresponding model. The BIC values are plotted in the same way in the lower row.}
                \label{fig:RealDatasBIC}
\end{figure}

Finally, we apply the sBIC to two different real data examples. We use the setup outlined in Section \ref{sec:UnifyPPCAandFA} where we jointly examine candidate PPCA and factor analysis models. By the construction of the sBIC, the sBIC values for the PPCA model do not depend on whether the factor analysis models are included as candidate models. However, the factor analysis sBIC will be different based on whether the PPCA models are also considered because quantities from PPCA appear in the sums \eqref{eqn:sBICbi} and \eqref{eqn:sBICci}. It should be noted that Theorem \ref{thm:generalized-formula-learning-coeff} and the sBIC computations to follow only hold for factor analysis models with $k$ small enough to satisfy condition \eqref{eqn:FADimensionRequirement}.

The first example consists of the decathlon scores across $p = 10$ events for the top $n = 24$ men competing in the $2008$ Summer Olympics \citep{marden2015multivariate}. The second dataset records $p = 14$ gustatory qualities of $n = 21$ different wines \citep{WineDatahusson2016package}. Figure \ref{fig:RealDatasBIC} displays the scree plot of the eigenvalues of the sample covariance matrices along with normalized values of the sBIC and BIC for models with $0,\ldots,p-1$ principal components and factors. The factor analysis model becomes full-dimensional when there are $6$ and $10$ factors respectively for the decathlon and wine data and condition \eqref{eqn:FADimensionRequirement} holds when the number of factors is at most $6$ and $9$, respectively. Thus, the sBIC values displayed in Figure \ref{fig:RealDatasBIC} are theoretically valid only up to these thresholds. All quantities are normalized to take values in $[0,1]$. The sBIC selects the factor analysis model with $4$ factors while the BIC selects the PPCA model with $r = 1$ principal component for the decathlon data. For the wine data the sBIC and BIC select PPCA models with $8$ and $4$ principal components respectively.
These results are consistent with the simulation studies showing that the sBIC tends to select larger models as compared to the BIC. While there is no known ground truth for the given examples, based on the scree plot of the decathlon example it appears that a model with only a single PC is overly conservative because many of the eigenvalues beyond the largest constitute a significant proportion of the total variance. It should be noted that while the sBIC is less conservative than the BIC, both procedures will tend to select relatively small models.

\section{Conclusion and Future Directions}

The present work fills a gap in the literature: obtaining the correct large sample asymptotics for approximate Bayesian model selection in PPCA, as well as providing a unified framework for performing model selection jointly on PPCA and factor analysis models. At a technical level, we derive learning coefficients that provide a Bayesian measure of effective dimension for overparametrized PPCA models, and we use their values to form an sBIC for selection of the number of principal components/latent factors. Adopting high-dimensional asymptotics, we show that the sBIC is consistent yet more sensitive to eigenvalue spikes than the ordinary BIC.
There are a number of related open questions, a few of which we now mention. 

The learning coefficients calculated in this paper are derived from fixed-dimensional asymptotics for marginal likelihood integrals. For regular settings, recent work examines conditions for the validity of Laplace approximations to the marginal likelihood in high-dimensional settings \citep{katsevich2024laplace,katsevich2026high,ReidTangLaplaceApprox}. Obtaining asymptotic log-marginal likelihood expansions in high dimensions when singularities are present is a challenging open direction. 

Other future avenues of study include determining the learning coefficients for models that are related to PPCA. As indicated in Section \ref{sec:UnifyPPCAandFA}, the learning coefficients of both the factor analysis and PPCA models can be determined by the techniques introduced here and in \citep{drton2025Factor}. Rather than partitioning the noise variances as in the partitioned noise model, an alternative modeling framework is to assume that the noise is isotropic and that the principal components are partitioned according to their norms. In the special case where the partition is of the form $(k,p-k)$ the resulting model is the isotropic PCA model with covariance matrix $\bs{\Sigma} = \tau^2 \bl{W}\bl{W}^\intercal + \sigma^2 \bl{I}$, where the columns of $\bl{W}$ are orthonormal \citep{bouveyron2011IsotropicPPCA}. Maximum likelihood alone is sufficient for consistent model selection in isotropic PPCA simply because no model is contained in another. Generalizing both PPCA and isotropic PCA is stratified PCA, where the covariance matrices are parameterized according to $\bs{\Sigma} = \sum_{i = 1}^m \sigma_i^2 \bl{W}_i\bl{W}_i^\intercal$, with $\sigma^2_1 \geq \cdots \geq \sigma^2_m$, each $\bl{W}_i$ having orthonormal columns, and $\sum_{i = 1}^m \bl{W}_i\bl{W}_i^\intercal = \bl{I}$ \citep{SPCAPennec}. Determining the learning coefficients for the larger class of stratified PCA models remains an open question. Lastly, we mention that canonical correlation analysis (CCA) has a probabilistic formulation akin to PPCA \citep{bach2005CCA}. As this model also has singularities, a study of learning coefficients is warranted and can be used for model selection on the number of canonical directions.

\section*{Acknowledgments}

A.~McCormack acknowledges the support of the Natural Sciences and Engineering Research Council of Canada (NSERC), [RGPIN-2025-03968, DGECR-2025-00237]. D.~Windisch was supported by the FWO grants G0F5921N (Odysseus) and G023721N, and by the KU Leuven grant iBOF/23/064.  M.~Drton received funding from the European Research Council (ERC) under the European Union’s Horizon 2020 research and innovation programme (grant agreement No. 883818).

\section*{Disclosure on the use of large language models}
The material in Section \ref{sec:HighDimensional} and the corresponding proofs in Appendix \ref{sec:appHighDimProofs} were influenced by exchanges with Opus 5 and GPT-5.6 Sol. Large language models were also used to assist with coding tasks, and assisted with copy-editing of the manuscript. All of the writing and remaining results in this paper were human-generated. The authors take full responsibility for the content of the manuscript.

\bibliographystyle{abbrv}
\bibliography{biblio}

\clearpage

\appendix 
\numberwithin{table}{section}
\numberwithin{figure}{section}

\section{Supplementary Material: Proofs}
\label{app:Proofs}

\subsection{Geometry of the PPCA Model}

\begin{lemma}[Dimension of PPCA Tangent Spaces]
\label{lem:ModelDimension}
When $0 \leq k \leq p-1$, the dimension of \[\mathrm{span}\bigg(\{\bl{W}\bl{V}^\intercal + \bl{V}\bl{W}^\intercal : \bl{V} \in \mb{R}^{p \times k} \} \cup \{\bl{I}_p\} \bigg)\] is equal to $r(p-r) + r(r+1)/2 + 1$ when $\mathrm{rank}(\bl{W}) = r$. 
\end{lemma}
\begin{proof}
The dimension of this subspace is unchanged after it is transformed by an invertible linear transformation. It is also unchanged when $\bl{W}$ is replaced by $\bl{W}\bl{O}$ for an orthogonal $\bl{O} \in \mathrm{O}(k)$, because $\bl{V} \mapsto \bl{V}\bl{O}$ is a bijection of $\mb{R}^{p \times k}$. Applying such an $\bl{O}$ obtained from a singular value decomposition of $\bl{W}$, we may assume that the last $k-r$ columns of $\bl{W}$ are zero. Let $\bl{U}$ be an invertible $p \times p$ matrix with rows that are orthogonal to each other such that
\begin{align*}
    \bl{U}\bl{W} = \begin{bmatrix}
        \bl{I}_r & \bl{0}_{r \times k-r}
        \\
        \bl{0}_{p-r \times r} & \bl{0}_{p-r \times k-r}
    \end{bmatrix}.
\end{align*}
The first $r$ rows of $\bl{U}$ consist of the first $r$ scaled left singular vectors of $\bl{W}$, of which there are only $r$ with non-zero singular values due to the condition $\text{rank}(\bl{W}) = r$. It suffices to compute the dimension of
\begin{align*}
\text{span}\bigg(\{\bl{U}\bl{W}\bl{V}^\intercal + \bl{V}(\bl{U}\bl{W})^\intercal : \bl{V} \in \mb{R}^{p \times k} \} \cup \{\bl{U}\bl{U}^\intercal\}\bigg).
\end{align*}
The matrix $\bl{U}\bl{W}\bl{V}^\intercal + \bl{V}(\bl{U}\bl{W})^\intercal$ has the form 
\begin{align*}
    \begin{bmatrix}
        \bl{V}_{1:r,1:r} + \bl{V}_{1:r,1:r}^\intercal & \bl{V}_{1:r,r+1:p}^\intercal
        \\
        \bl{V}_{r+1:p,1:r} & \bl{0}_{r+1:p \times r+1:p}
    \end{bmatrix}.
\end{align*}
All such matrices are linearly independent of the full-rank diagonal matrix $\bl{U}\bl{U}^\intercal$ due to the $\bl{0}_{r+1:p \times r+1:p}$ matrix in the lower-right block. The upper-right block can vary freely and contributes a dimension of $r(p-r)$, while the upper-left block can be any symmetric matrix and so it has dimension $\binom{r+1}{2} = r(r+1)/2$. Adding these dimensions, along with the extra dimension from the $\bl{U}\bl{U}^\intercal$ term, yields the result.
\end{proof}

\begin{lemma}
    The set $\mb{R}^{p \times k}_* / \mathrm{O}(k)$ has the structure of a smooth manifold. 
\end{lemma}
\begin{proof}
The group action $\bl{W} \mapsto \bl{W}\bl{U}$ on $\mb{R}^{p \times k}_*$ is free since $\bl{W} = \bl{W}\bl{U}$ if and only if $\bl{U} = \bl{I}_k$ because the columns of $\bl{W}$ are linearly independent. This group action is proper since $\text{O}(k)$ is compact \citep[Cor.~21.6]{LeeSmoothManifolds}. By the quotient manifold theorem \citep[Thm.~21.10]{LeeSmoothManifolds} the orbit space $\mb{R}_*^{p \times k} / \text{O}(k)$ is a smooth manifold. A related result can be found in \citep{FixedRankPDManifoldGeomvandereycken}. 
\end{proof}

\begin{lemma}
\label{lem:NonSingularPointsareManifold}
    The set $\mc{M}_{k} \setminus \mc{M}_{k-1}$ is a smooth manifold when $1\leq k \leq p-1$. It is diffeomorphic to $\mb{R}^{p \times k}_* /\mathrm{O}(k) \times \mb{R}_{> 0}$ via the parameterization map $\varphi_k$. 
\end{lemma}
\begin{proof}
    The parameterization map $\varphi_k(\bl{W},\sigma^2) = \bl{W}\bl{W}^\intercal + \sigma^2 \bl{I}_p$ is a bijection from $\mb{R}^{p \times k}_* / \text{O}(k) \times \mb{R}_{> 0}$ to $\mc{M}_k \setminus \mc{M}_{k-1}$. Moreover, the Jacobian of $\varphi_k$ has full rank by the same computation as in Lemma \ref{lem:ModelDimension} and so it is an injective, smooth immersion from $\mb{R}^{p \times k}_* / \text{O}(k) \times \mb{R}_{> 0}$ into $\mc{S}^p_{++}$.
    By Prop.~4.22 of \citep{LeeSmoothManifolds} this is a smooth embedding if it is shown that $\varphi^{-1}$ is continuous. Assume that $\bs{\Sigma}_n \coloneqq \bl{W}_n\bl{W}_n^\intercal + \sigma^2_n \bl{I}_p \rightarrow \bs{\Sigma} \coloneqq \bl{W}\bl{W}^\intercal + \sigma^2\bl{I}_p$. By the continuity of eigenvalues the smallest eigenvalue of $\bs{\Sigma}_n$ converges to the smallest eigenvalue of $\bs{\Sigma}$: $\sigma_n^2 \rightarrow \sigma^2$. It follows that $\Vert \bl{W}_n \Vert_F^2 = \text{tr}(\bs{\Sigma}_n) - p\sigma^2_n \rightarrow \text{tr}(\bs{\Sigma}) - p\sigma^2 $, and hence that the sequence $\bl{W}_n$ is bounded. Writing $[\bl{W}]$ for the orbit containing $\bl{W}$ in the quotient space $\mb{R}^{p \times k}_* /\mathrm{O}(k)$, we claim that $[\bl{W}_n] \rightarrow [\bl{W}]$. If this were not the case there would exist a  neighbourhood $U$ of $[\bl{W}]$ and a subsequence with $[\bl{W}_{n_j}] \notin U$ for all $j$. However, since $\bl{W}_{n_{j}}$ is bounded there exists a further subsequence with $\bl{W}_{n_{j_k}} \rightarrow \tilde{\bl{W}}$ and $\tilde{\bl{W}}\tilde{\bl{W}}^\intercal = \bl{W}\bl{W}^\intercal$. By how the quotient space is defined $[\tilde{\bl{W}}] = [\bl{W}]$. This is a contradiction because $[\bl{W}_{n_{j_k}}] \rightarrow [\bl{W}]$ but $[\bl{W}_{n_j}] \notin U$ for all $j$. We conclude that $([\bl{W}_n],\sigma_n^2) \rightarrow ([\bl{W}],\sigma^2)$.  
\end{proof}

\subsection{High-Dimensional Consistency}
\label{sec:appHighDimProofs}
The proofs of  Theorems~\ref{thm:BICConsistency} and~\ref{thm:sBICConsistency} are provided in this section. The notation, Gaussian sampling assumption, and conditions
\eqref{condition:C1}--\eqref{condition:C2} appearing in Section~\ref{sec:HighDimensional} are used throughout this section. 

Following Lemma~\ref{PPCA:MLE},
write $t_{1n}\geq\cdots\geq t_{p_n,n}>0$ for the eigenvalues of the sample covariance matrix $ \bl{S}_n=\frac{1}{n}\sum_{i=1}^n\bl{X}_{in}\bl{X}_{in}^{\intercal}$. For $0\leq k\leq p_n-1$, define
\begin{align}
    \hat{\ell}_{kn}
   &= \max_{\bs{\Sigma}\in\mc{M}_k}
       \log\big(p(\bl{X}_{1n},\ldots,\bl{X}_{nn}\mid\bs{\Sigma})\big) \nonumber
       \\
       &= -\frac{n}{2}\bigg(\sum_{i = 1}^k \log(t_{in}) + (p_n-k) \log(\Bar{t}_{kn})\bigg) -\frac{p_n n}{2}(\log(2\pi)+ 1), \label{eqn:MaximizedLikelihood}
\end{align}
where $\Bar{t}_{kn}=\frac{1}{p_n-k}\sum_{i=k+1}^{p_n}t_{in}$. Write
\begin{align*}
    g_{kn}=\hat{\ell}_{kn}-\hat{\ell}_{k-1,n}, \qquad k=1,\dots, p_n-1,
\end{align*}
for the log-likelihood differences. We likewise append $n$ to the learning coefficients and marginal likelihood
proxies to record their dependence on dimension and sample size.
Theorem~\ref{thm:formula-learning-coeff} and
\eqref{eqn:LearningCoeffExpression} give $\mf{m}_{kj,n}=1$ and
\begin{align}
 \lambda_{kj,n}=\frac{p_nk+j(p_n-k+1)+2}{4}, \qquad \text{ for } \quad
    0\leq j\leq k\leq p_n-1.\label{eqn:HDCoefficients}
\end{align}
Define
\begin{align}
 F_{kj,n}=\exp\big(\hat{\ell}_{kn}-\lambda_{kj,n}\log(n)\big), \qquad
 B_{kn}=\log(F_{kk,n}).\label{eqn:HDScoreDefinitions}
\end{align}
Then $\hat{k}_{\mathrm{BIC},n}$ maximizes $B_{kn}$, and
$\hat{k}_{\mathrm{sBIC},n}$ maximizes $\log(L_{kn})$ with $L_{kn}$
specified by \eqref{eqn:sBICL0Initialize}--\eqref{eqn:sBICci}.

The following expressions for the learning coefficient differences with $0 \leq i \leq j \leq k \leq p_n - 1$ are needed in subsequent results:
\begin{align}
 \lambda_{ki,n}-\lambda_{ji,n}
   &=\frac{(k-j)(p_n-i)}{4}, \label{eqn:HDCrossDifference}\\
 \lambda_{kk,n}-\lambda_{jj,n}
   &=\frac{(k-j)(2p_n-k-j+1)}{4}.\label{eqn:HDDiagonalDifference}
\end{align}
These expressions are immediate from \eqref{eqn:HDCoefficients}.

For simplicity, we scale the data as $\bl{X}_{in} \mapsto \bl{X}_{in}/\sigma_n$ and thus assume without loss of generality that $\sigma^2_n = 1$ for all $n$. This is justified because scaling the data has no impact on the model selected by either the BIC or the sBIC. Indeed, if $\bl{X}_{in} \mapsto a\bl{X}_{in}$ for some $a>0$, then the sample eigenvalues are multiplied by $a^2$, that is, $t_{in} \mapsto a^2 t_{in}$, and the maximized log-likelihood in \eqref{eqn:MaximizedLikelihood} changes from $\hat{\ell}_{kn}$ to $\hat{\ell}_{kn} - np_n\log(a)$ for every $k$. Hence, $B_{kn}$ changes to $B_{kn} - np_n\log(a)$, and the sBIC recursion \eqref{eqn:sBICLi} shows that $\log(L_{kn})$ likewise changes to $\log(L_{kn}) - np_n \log(a)$ for every $k$. Consequently, scaling leaves the relative ordering of both the BIC and the sBIC values across models unaltered.

We begin by importing results appearing in Lemma 2.2 of \citep{BaiAICBICConsistency} that describe the asymptotic behaviour of the sample eigenvalues. These results are based on the observation that the empirical distribution of the smallest $p_n - r$ sample eigenvalues converges to a Mar\v{c}enko--Pastur distribution, with the extreme ones converging to the edges $(1\pm\sqrt{c})^2$ of its support, while the largest $r$ sample eigenvalues are ratio-consistent for the corresponding population eigenvalues if these diverge to $\infty$.

\begin{lemma}[Sample eigenvalue convergence]\label{lem:HDSpectralBounds}
Under conditions~\eqref{condition:C1}--\eqref{condition:C2} and the assumption $\sigma_n^2 = 1$, it holds that
\begin{align}
 \Bar{t}_{rn}
   \convp1.\label{eqn:HDResidualMean}
\end{align}
Moreover, there exist constants $0<a< (1 - \sqrt{c})^2 < (1 + \sqrt{c})^2 < b< \infty$ such that
\begin{align}
 \mathrm{Pr}\left(a\leq t_{p_n,n}
       \leq t_{r+1,n}\leq b\right)
       \longrightarrow1,\label{eqn:HDNoiseEvent}
\end{align}
and $t_{rn}/(1+\tau_{rn}^2)$ is bounded in probability.
If $\tau_{rn}^2 \to\infty$, then
\begin{align}
 \max_{1\leq j\leq r}
 \left|\frac{t_{jn}}{1 +\tau_{jn}^2}-1\right|
       \convp0.\label{eqn:HDSpikeRatio}
\end{align}
\end{lemma}

The next lemma controls the behaviour of the likelihood increments $g_{kn}$ by expressing the increments as a function of the eigenvalue ratio
\begin{align*}
    R_{kn} = \frac{t_{kn}}{\Bar{t}_{kn}} \geq 1, \qquad k=1,\dots, p_n-1;
\end{align*}
larger values of this ratio correspond to larger likelihood increments.   

\begin{lemma}[Likelihood increments]
\label{lem:HDLikelihoodIncrements}
For $k=1,\dots, p_n-1$, let $q_{kn}=p_n-k+1$.
Then
\begin{align}
 g_{kn} =\frac{n}{2}\,H_{q_{kn}}(R_{kn}), \quad \text{ where } \quad
 H_q(x) =q\log\left(1+\frac{x-1}{q}\right)-\log(x).
 \label{eqn:HDGainIdentity}
\end{align}
Suppose condition~\eqref{condition:C1} holds and $k$ does not depend on $n$. If
$R_{kn}/\log(n)\convp\beta\in[0,\infty)$, then
\begin{align}
 \frac{2g_{kn}}{n\log(n)}
       \convp\beta.\label{eqn:HDHLimit}
\end{align}
Moreover, for $q>1$ and $x\geq1$, the function $H_q$ is nondecreasing and approximates
$\Phi(x)=x-1-\log(x)$ in the sense that
\begin{align}
 0\leq H_q(x)\leq\Phi(x),\qquad
 0\leq\Phi(x)-H_q(x)\leq\frac{(x-1)^2}{2q}.
 \label{eqn:HDRemainderBound}
\end{align}
\end{lemma}
\begin{proof}
Lemma~\ref{PPCA:MLE} gives, up to a term independent of $k$,
\begin{align*}
 \hat{\ell}_{kn}
 =-\frac{n}{2}\left(\sum_{i=1}^{k}\log(t_{in})
                     +(p_n-k)\log(\Bar{t}_{kn})\right).
\end{align*}
The averaged eigenvalues satisfy the relation
\begin{align*}
 q_{kn}\Bar{t}_{k-1,n}= \sum_{i = k}^{p_n} t_{in} = t_{kn} + \sum_{i = k+1}^{p_n} t_{in} =  t_{kn}+(q_{kn}-1)\Bar{t}_{kn}.
\end{align*}
Subtracting successive log-likelihoods therefore gives
\eqref{eqn:HDGainIdentity}.

We next prove \eqref{eqn:HDRemainderBound}, which will then be used to derive \eqref{eqn:HDHLimit}. Note that $H_q(1) = 0$.
Direct differentiation yields
\begin{align*}
 H'_q(x)=\frac{(q-1)(x-1)}{x(q+x-1)}\geq 0
\end{align*}
when $q > 1$ and $x \geq 1$. This shows that $0 \leq H_q(x)$.   
For every $u\geq0$,
\begin{align*}
 0\leq u-\log(1+u)
   =\int_0^u\frac{v}{1+v}\,\mathrm{d}v\leq \int_0^u v \,\mathrm{d}v = \frac{u^2}{2}.
\end{align*}
Take $u=(x-1)/q$ and multiply by $q$ to obtain
\eqref{eqn:HDRemainderBound}. 

For \eqref{eqn:HDHLimit}, abbreviate $R_n=R_{kn}$ and $q_n=q_{kn}$, so that $2g_{kn}/(n\log(n))=H_{q_n}(R_n)/\log(n)$ by \eqref{eqn:HDGainIdentity}. Since $k$ is fixed, condition~\eqref{condition:C1} gives $q_n/n\to c>0$. Convergence of $R_n/\log(n)$ to a finite
constant gives $R_n=O_p(\log(n))$. Thus, by \eqref{eqn:HDRemainderBound},
\begin{align*}
 0\leq\frac{\Phi(R_n)-H_{q_n}(R_n)}{\log(n)}
 \leq\frac{(R_n-1)^2}{2q_n\log(n)}
 =O_p\left(\frac{\log(n)}{n}\right)\convp0.
\end{align*}
Thus, $H_{q_n}(R_n)/\log(n)$ has the same limit in probability as  $\Phi(R_n)/\log(n)$ if the limit of the latter exists. We have
\begin{align*}
    \frac{\Phi(R_n)}{\log(n)} = \frac{R_n}{\log(n)} - \frac{1}{\log(n)} - \frac{\log(R_n)}{\log(n)} \convp \beta,
\end{align*}
because $R_n/\log(n) \convp \beta$ by assumption and  $\log(R_n)/\log(n)\convp 0$. 
\end{proof}

The goal of the next two lemmas is to express BIC and sBIC differences in terms of log-likelihood differences and $O(\log(n))$ penalties. This is straightforward for the BIC, but for the sBIC a careful analysis of the sBIC recursion is needed to bound the recursive formula of $L_{kn}$ in terms of more tractable log-likelihood terms.   

\begin{lemma}[BIC Increment Bound]
\label{lem:HDBICIncrements}
 For $0\leq j<k\leq p_n-1$ we have that
    \begin{align}
        B_{kn}-B_{jn} \;\leq\; \hat{\ell}_{kn}-\hat{\ell}_{jn}
       -\frac{(k-j)(p_n-j)}{4}\log(n), \label{eqn:HDBICScoreUpper}
    \end{align}
    and 
    \begin{align}
        B_{kn}-B_{k-1,n} \;=\; g_{kn}-\frac{p_n-k+1}{2}\log(n).\label{eqn:HDBICIncrement}
    \end{align}
\end{lemma}
\begin{proof}
    The formula \eqref{eqn:HDBICIncrement} is a straightforward calculation, while \eqref{eqn:HDBICScoreUpper} follows from \eqref{eqn:HDDiagonalDifference} and $(p_n-j) \leq (2p_n - k - j + 1)$.
\end{proof}

\begin{lemma}[sBIC Increment Bound]\label{lem:HDScoreComparisons}
For
$0\leq j<k\leq p_n-1$,
\begin{align}
\log(L_{kn}) - \log({L_{jn}}) \;\leq\; \hat{\ell}_{kn}-\hat{\ell}_{jn}
       -\frac{(k-j)(p_n-j)}{4}\log(n),\label{eqn:HDScoreUpper}
\end{align}
and
\begin{align}
 g_{kn} \;>\; \frac{p_n}{4}\log(n)+\log(2)
 &\quad\Longrightarrow\quad L_{kn} \;>\; L_{k-1,n}.\label{eqn:HDScoreLower}
\end{align}
\end{lemma}
\begin{proof}
To ease notation we suppress the sample-size index on the scores and put $p=p_n$ in this proof. For $k\geq 1$, let $T_k=\sum_{i<k}L_i$ and
$C_k=\sum_{i<k}F_{ki}L_i$. By
\eqref{eqn:sBICLi}--\eqref{eqn:sBICci}, $L_k$ is the positive root of
\begin{align}
 Q_k(x)=x^2+(T_k-F_{kk})x-C_k.\label{eqn:HDQuadratic}
\end{align}
Starting from $L_0=F_{00}>0$, induction gives $C_k>0$, so $Q_k$
has exactly one positive and one negative root. In particular, a point $x > 0$ satisfies $Q_k(x)<0$ if and only if $x<L_k$.

The two claims of the lemma are proved as follows. The bound \eqref{eqn:HDScoreUpper} is obtained by comparing the sBIC scores through the weighted-average form \eqref{eqn:HDWeightedIdentity} of the recursion derived below. For the implication \eqref{eqn:HDScoreLower}, we show that $Q_k(L_{k-1})<0$ under its premise, which by the preceding observation yields $L_{k-1}<L_k$.

Rearranging the equation $Q_k(L_k)=0$, and using $L_0=F_{00}$ at $k=0$, gives
\begin{align}
 L_k\sum_{i=0}^kL_i=\sum_{i=0}^kF_{ki}L_i.\label{eqn:HDWeightedIdentity}
\end{align}
This is exactly the formulation (3.16) of the sBIC appearing in \citep{sBICDrtonPlummer}. Also note that for every $k$ the $F_{ki}$ are decreasing in the index $i$ because $\lambda_{ki,n}$ increases with $i$. As $L_k$ is a weighted average of the $F_{ki}$ in \eqref{eqn:HDWeightedIdentity} it follows that $F_{kk} \leq L_k \leq F_{k0}$.

We now prove \eqref{eqn:HDScoreUpper}. Put $h=\exp\big(\hat{\ell}_k-\hat{\ell}_j-(k-j)(p-j)\log(n)/4\big)$.
For $i\leq j$, \eqref{eqn:HDCrossDifference} gives
$F_{ki}\leq hF_{ji}$. For $j<i\leq k$, monotonicity in the second
index gives $F_{ki}\leq F_{kj}=hF_{jj}\leq hL_j$.
Consequently,
\begin{align*}
 L_k\sum_{i=0}^kL_i
 &=\sum_{i=0}^jF_{ki}L_i+\sum_{i=j+1}^kF_{ki}L_i\\
 &\leq h\sum_{i=0}^jF_{ji}L_i+hL_j\sum_{i=j+1}^kL_i
 =hL_j\sum_{i=0}^kL_i,
\end{align*}
where the last equality uses \eqref{eqn:HDWeightedIdentity} at rank
$j$. Division by the positive $\sum_{i=0}^kL_i$ proves the sBIC comparison in
\eqref{eqn:HDScoreUpper}.

Finally, we prove \eqref{eqn:HDScoreLower}. Put $h=\exp\big(g_k-p\log(n)/4\big)$, so that the premise in \eqref{eqn:HDScoreLower} is that $h > 2$. As explained above, it suffices to show that $Q_k(L_{k-1}) < 0$. 

For every $i<k$,
\begin{align*}
 \frac{F_{ki}}{F_{k-1,i}}
 =\exp\big(g_k-(p-i)\log(n)/4\big)\geq h.
\end{align*}
The preceding inequality and equation~\eqref{eqn:HDWeightedIdentity} at rank $k-1$ therefore gives
\begin{align*}
    C_k = \sum_{i<k}\frac{F_{ki}}{F_{k-1,i}} F_{k-1,i} L_i \geq h \sum_{i = 0}^{k-1} F_{k-1,i} L_i = h L_{k-1} T_k.
\end{align*}
 As $T_k\geq L_{k-1}$, $F_{kk}>0$, and $h > 2$
\begin{align*}
 Q_k(L_{k-1}) = L_{k-1}^2 + (T_k - F_{kk})L_{k-1} - C_k
 &\leq L_{k-1}^2+(1-h)L_{k-1}T_k
 \leq(2-h)L_{k-1}T_k < 0.
\end{align*}
\end{proof}

We now prove one-half of the consistency of both the BIC and sBIC: models that are larger than the true model are never selected by either criterion.

\begin{lemma}[BIC and sBIC do not overfit asymptotically]
\label{lem:HDNoOverfitting}
Under conditions~\eqref{condition:C1}--\eqref{condition:C2}, with
probability tending to one,
\begin{align}
 B_{kn}<B_{rn}\quad\text{and}\quad L_{kn}<L_{rn}
 \qquad\text{for all }k=r+1,\dots,p_n-1.\label{eqn:HDNoOverfitting}
\end{align}
\end{lemma}
\begin{proof}
On the event in \eqref{eqn:HDNoiseEvent}, all
$j=r+1,\dots,p_n-1$ satisfy $1\leq R_{jn}\leq b/a$. The bound $H_q(x) \leq \Phi(x)$ in
Lemma~\ref{lem:HDLikelihoodIncrements} therefore gives
$0\leq g_{jn}\leq Cn$ simultaneously for these indices, where
$C=\Phi(b/a)/2<\infty$. Summing increments gives 
\begin{align*}
\hat{\ell}_{kn} - \hat{\ell}_{rn} = \sum_{j = r+1}^k g_{jn} \leq (k-r)Cn.
\end{align*}
The identical BIC and sBIC increment bounds \eqref{eqn:HDBICScoreUpper} and \eqref{eqn:HDScoreUpper}
with $j=r$ then yield
\begin{align*}
    \log(L_{kn}) - \log({L_{rn}}) &\leq (k-r)\left(Cn-\frac{p_n-r}{4}\log(n)\right),
       \\
       B_{kn}-B_{rn} &\leq (k-r)\left(Cn-\frac{p_n-r}{4}\log(n)\right),
\end{align*}
for all $k=r+1,\dots,p_n-1$.  As $Cn-\frac{p_n-r}{4}\log(n) < 0$ for sufficiently large $n$ because $p_n/n \rightarrow c$ we conclude that with probability converging to one $L_{kn} < L_{rn}$ and $B_{kn} < B_{rn}$ for \textit{all} $k=r+1,\dots,p_n-1$.
\end{proof}

Next, we provide the core lemma needed to show that the BIC and sBIC select a large enough model when the $r$th eigenvalue gap is sufficiently large. It uses Lemma~\ref{lem:HDLikelihoodIncrements} to describe the probabilistic behaviour of the likelihood increments for the models with at most $r$ PCs. 

\begin{lemma}[Weakest-spike likelihood bounds]
\label{lem:HDLikelihoodBounds}
Let $\beta_n=\tau_{rn}^2/\log(n)$.
For every constant $b\in(0,\liminf_n\beta_n)$,
\begin{align}
 \mathrm{Pr}\left(\min_{1\leq k\leq r}
      \frac{2g_{kn}}{n\log(n)}> b\right)
       \longrightarrow1.\label{eqn:HDGainLower}
\end{align}
Along any subsequence on which
$\beta_n\to\beta\in[0,\infty)$,
\begin{align}
 \frac{2g_{rn}}{n\log(n)}
       \convp \beta.
 \label{eqn:HDWeakGainLimit}
\end{align}
\end{lemma}
\begin{proof}
For \eqref{eqn:HDGainLower}, fix constants $b_1,b_2$ with
$b<b_1<b_2<\liminf_n\beta_n$; each strict inequality leaves room for an error that vanishes asymptotically.
Lemma~\ref{lem:HDSpectralBounds} gives $\Bar{t}_{rn}\to1$ in probability
and $t_{rn} \geq b_2\log(n)$ with probability tending to
one. The latter assertion follows since
\begin{align*}
    \frac{t_{rn}}{\log(n)} = \frac{t_{rn}}{1 + \tau_{rn}^2}\cdot\frac{1 + \tau_{rn}^2}{\log(n)},
\end{align*}
where the first factor converges in probability to one by \eqref{eqn:HDSpikeRatio}, and the second factor satisfies $\liminf_n \frac{1+\tau_{rn}^2}{\log(n)}\geq\liminf_n\beta_n>b_2$.  

Let $1\leq k\leq r$. Splitting the sum defining $\Bar{t}_{kn}$ at index $r$ gives
$(p_n-k)\Bar{t}_{kn}=\sum_{i=k+1}^r t_{in}+(p_n-r)\Bar{t}_{rn}$, which yields the equality in \eqref{eqn:HDInverseRatio} below. For the inequality, we use the ordering of the sample eigenvalues twice: $t_{in}\leq t_{kn}$ for $k<i\leq r$ bounds the sum $\sum_{i=k+1}^r t_{in}$ by $(r-k)t_{kn}$, and $t_{kn}\geq t_{rn}$ gives $1/t_{kn}\leq 1/t_{rn}$ in the remaining term. Hence,
\begin{align}
 \frac{\log(n)}{R_{kn}}
 &=\frac{\log(n)}{p_n-k}
     \frac{(p_n-r)\Bar{t}_{rn}+\sum_{i=k+1}^r t_{in}}
          {t_{kn}}\nonumber\\
 &\leq\frac{(r-k)\log(n)}{p_n-k}
    +\frac{p_n-r}{p_n-k}
       \frac{\Bar{t}_{rn}\log(n)}{t_{rn}}.
 \label{eqn:HDInverseRatio}
\end{align}
Since $r$ is fixed, the first term on the right side converges to zero. On the event $t_{rn}\geq b_2\log(n)$, the second term is at most $\Bar{t}_{rn}/b_2$ because $p_n-r\leq p_n-k$, and $\Bar{t}_{rn}/b_2$ converges in probability to $1/b_2<1/b_1$. Hence $\max_{k\leq r}\log(n)/R_{kn}\leq 1/b_1$, that is, $\min_{k\leq r}R_{kn}\geq b_1\log(n)$, with probability tending to one. When this event occurs the fact that $H_q(x)$ is nondecreasing
(Lemma~\ref{lem:HDLikelihoodIncrements}) gives
\begin{align*}
 \min_{k\leq r}\frac{2g_{kn}}{n\log(n)}
 \geq\min_{k\leq r}
       \frac{H_{p_n-k+1}(b_1\log(n))}{\log(n)}.
\end{align*}
The expression on the right converges to $b_1$ by \eqref{eqn:HDRemainderBound} because $\Phi(b_1\log(n))/\log(n) \rightarrow b_1$ and $(b_1\log(n) - 1)^2/\big(2(p_n - k +1)\big) \rightarrow 0$ for every $k \leq r$. This proves \eqref{eqn:HDGainLower} because $b_1 > b$.

For \eqref{eqn:HDWeakGainLimit}, we consider the specified subsequence.
If $\beta>0$, \eqref{eqn:HDResidualMean} and \eqref{eqn:HDSpikeRatio}
 imply
\begin{align*}
 \frac{R_{rn}}{\log(n)}
 =\frac{t_{rn}}{1+\tau_{rn}^2}
   \frac{1+\tau_{rn}^2}{\log(n)\Bar{t}_{rn}}
 \convp \beta.
\end{align*}
If $\beta=0$, use the fact that $t_{rn}/(1 + \tau_{rn}^2)$ is bounded in probability (Lemma~\ref{lem:HDSpectralBounds}) in the above expression to obtain $R_{rn}/\log(n) \convp 0$. In either case \eqref{eqn:HDHLimit} with $k=r$
proves \eqref{eqn:HDWeakGainLimit}.
\end{proof}

The next two lemmas show that if the $r$th and $(r+1)$st eigenvalues are sufficiently spread apart, then, with probability tending to one, the BIC and the sBIC do not select a model with fewer than $r$ PCs; together with Lemma~\ref{lem:HDNoOverfitting}, they select precisely $r$ PCs. If the respectively needed gaps are too small, both criteria fail to select $r$. 

\begin{lemma}[BIC does not underfit under an eigengap condition]
\label{lem:HDBICUnderfitting}
    If $\liminf_n \tau_{rn}^2/\log(n) >c$ then $B_{0n}<\cdots<B_{rn}$ with probability tending to one. If $\limsup_n \tau_{rn}^2/\log(n)<c$ then $B_{r-1,n} > B_{rn}$ with probability tending to one. 
\end{lemma}
\begin{proof}
Suppose $\liminf_n\beta_n>c$, with $\beta_n = \tau_{rn}^2/\log(n)$ as in
Lemma~\ref{lem:HDLikelihoodBounds}, and choose
$c<b<\liminf_n\beta_n$.
On the event in \eqref{eqn:HDGainLower},
\eqref{eqn:HDBICIncrement} gives
\begin{align*}
 \min_{1\leq k\leq r}
 \frac{B_{kn}-B_{k-1,n}}{n\log(n)}
 >\frac{b}{2}-\frac{p_n}{2n}>0
\end{align*}
for all sufficiently large $n$. Therefore
$B_{0n}<\cdots<B_{rn}$ with probability tending to one.
Lemma~\ref{lem:HDNoOverfitting} excludes all larger ranks on another
event whose probability tends to one. Their intersection proves
$\mathrm{Pr}(\hat{k}_{\mathrm{BIC},n}=r)\to1$.

Now suppose that $\limsup_n\beta_n<c$. By considering a subsequence we assume that $\beta_n \to \beta < c$. If
$\mathrm{Pr}(\hat{k}_{\mathrm{BIC},n}=r)$ did not tend to zero, there would
be an $\varepsilon>0$ and a subsequence along which these
probabilities were at least $\varepsilon$. Along it,
\eqref{eqn:HDWeakGainLimit} and \eqref{eqn:HDBICIncrement} yield
\begin{align*}
 \frac{B_{rn}-B_{r-1,n}}{n\log(n)}
 \convp\frac{\beta-c}{2}<0.
\end{align*}
Hence $\mathrm{Pr}(B_{rn}\geq B_{r-1,n})\to0$ along this
subsequence, contradicting the assumption that $\mathrm{Pr}(B_{rn}\geq B_{r-1,n}) \geq \mathrm{Pr}(\hat{k}_{\mathrm{BIC},n}=r) \geq \varepsilon$ along the subsequence.
\end{proof}

\begin{lemma}[sBIC does not underfit under an eigengap condition]
\label{lem:HDsBICUnderfitting}
    If $\liminf_n \tau_{rn}^2/\log(n) >\frac{c}{2}$ then $L_{0n}<\cdots<L_{rn}$ with probability tending to one. If $\limsup_n \tau_{rn}^2/\log(n)<\frac{c}{2}$ then $L_{r-1,n} > L_{rn}$ with probability tending to one. 
\end{lemma}
\begin{proof}
Suppose $\liminf_n\beta_n>c/2$ and choose
$c/2<b<\liminf_n\beta_n$.
On the event in \eqref{eqn:HDGainLower},
\begin{align*}
 \min_{1\leq k\leq r}
 \frac{g_{kn}-p_n\log(n)/4-\log(2)}{n\log(n)}
 >\frac{b}{2}-\frac{p_n}{4n}-\frac{\log(2)}{n\log(n)}>0
\end{align*}
for all sufficiently large $n$. The implication in 
\eqref{eqn:HDScoreLower} shows that $L_{0n}<\cdots<L_{rn}$.

If $\limsup_n\beta_n<c/2$, as in the proof of Lemma \ref{lem:HDBICUnderfitting}, we work along a subsequence with $\beta_n\to \beta<c/2$. Suppose for contradiction that there is a further subsequence with $\mathrm{Pr}(\hat{k}_{\mathrm{sBIC},n}=r) \geq \varepsilon$. Equation~\eqref{eqn:HDScoreUpper} at $j=r-1$ and $k=r$ gives
\begin{align*}
 \frac{\log(L_{rn})-\log(L_{r-1,n})}{n\log(n)}
 \leq\frac{g_{rn}}{n\log(n)}-\frac{p_n-r+1}{4n}.
\end{align*}
By \eqref{eqn:HDWeakGainLimit}, the right side converges in
probability to $\beta/2-c/4<0$, so
$\mathrm{Pr}(L_{rn}\geq L_{r-1,n})\to0$. This provides the needed contradiction that 
\begin{align*}
\mathrm{Pr}(L_{rn}\geq L_{r-1,n})  \geq    \mathrm{Pr}(\hat{k}_{\mathrm{sBIC},n}=r) \geq \varepsilon.
\end{align*}
\end{proof}

Taken together, Lemmas \ref{lem:HDNoOverfitting} and \ref{lem:HDBICUnderfitting} prove Theorem \ref{thm:BICConsistency} while Lemmas \ref{lem:HDNoOverfitting} and \ref{lem:HDsBICUnderfitting} prove Theorem \ref{thm:sBICConsistency}.

\newpage

\section{Supplementary Material: Simulation Results}
\label{app:SimulationResults}
\raggedbottom

In all simulations reported below, $1000$ data sets of each sample size
$n \in \{30,50,250,4000\}$ were drawn from a centered Gaussian distribution with covariance matrix
$\bs{\Sigma}_0$, and the number of principal components was estimated by the sBIC, the BIC, the two
normal-Gamma variants NG1 and NG2, and the GCV method. Two covariance structures are considered:
the isotropic case $\bs{\Sigma}_0 = \diag(5,\ldots,5,1,\ldots,1)$ and the linear case
$\bs{\Sigma}_0 = \diag(r+1,\ldots,2,1,\ldots,1)$, where $r$ denotes the true number of PCs.

\subsection{Dimension \texorpdfstring{$p = 5$}{p = 5}}

Here $\bs{\Sigma}_0 \in \mb{R}^{5 \times 5}$ and the true number of PCs is $r \in \{0,\ldots,4\}$.
Table \ref{tab:PropP5} and Figure \ref{fig:PropP5} report the proportion of the $1000$ runs in which the estimated number of PCs equals the true number $r$, for both covariance structures. Table \ref{tab:DistP5} and Figure \ref{fig:DistP5} report the average absolute distance between the estimated and the true number of PCs, for both covariance structures.

\begin{table}[H]
\centering
\caption{Proportion of simulations where the model selection procedure correctly estimates the true number of PCs. The dimension of the covariance matrix is $p = 5$ and the table is divided into two cases: $\bs{\Sigma}_0 = \diag(5,\ldots,5,1,\ldots,1)$ (isotropic) and $\bs{\Sigma}_0 = \diag(r+1,\ldots,2,1,\ldots,1)$ (linear).}
\label{tab:PropP5}
\vspace{2mm}
\begin{tabular}{|r|r|rrrrr||rrrrr|}
\hline
&& \multicolumn{5}{|c||}{True No.\ of PCs, Isotropic} & \multicolumn{5}{c|}{True No.\ of PCs, Linear}
\\
  \hline
 & $n$ & 0 & 1 & 2 & 3 & 4 & 0 & 1 & 2 & 3 & 4 \\ 
  \hline
& 30 & 0.83 & 0.82 & 0.84 & 0.86 & 0.34 & 0.81 & 0.52 & 0.41 & 0.25 & 0.07 \\ 
\textbf{sBIC} & 50 & 0.92 & 0.88 & 0.90 & 0.90 & 0.78 & 0.88 & 0.66 & 0.56 & 0.45 & 0.18 \\ 
  & 250 & 0.97 & 0.96 & 0.97 & 0.96 & 1.00 & 0.97 & 0.98 & 0.97 & 0.97 & 1.00 \\ 
  & 4000 & 1.00 & 1.00 & 1.00 & 0.99 & 1.00 & 1.00 & 1.00 & 1.00 & 0.99 & 1.00 \\ 
  \hline
  & 30 & 0.99 & 0.96 & 0.89 & 0.71 & 0.10 & 0.98 & 0.21 & 0.15 & 0.08 & 0.03 \\ 
  \textbf{BIC} & 50 & 1.00 & 0.99 & 0.98 & 0.95 & 0.34 & 1.00 & 0.31 & 0.24 & 0.19 & 0.08 \\ 
  & 250 & 1.00 & 1.00 & 1.00 & 0.99 & 1.00 & 1.00 & 0.98 & 0.99 & 0.99 & 0.98 \\ 
  & 4000 & 1.00 & 1.00 & 1.00 & 1.00 & 1.00 & 1.00 & 1.00 & 1.00 & 1.00 & 1.00 \\ 
  \hline
  & 30 & 0.01 & 0.13 & 0.32 & 0.35 & 0.97 & 0.02 & 0.35 & 0.36 & 0.41 & 0.94 \\ 
  \textbf{NG1} & 50 & 0.00 & 0.09 & 0.28 & 0.35 & 1.00 & 0.00 & 0.30 & 0.39 & 0.45 & 0.98 \\ 
  & 250 & 0.00 & 0.00 & 0.09 & 0.19 & 1.00 & 0.00 & 0.14 & 0.24 & 0.35 & 1.00 \\ 
  & 4000 & 0.00 & 0.00 & 0.00 & 0.00 & 1.00 & 0.00 & 0.00 & 0.00 & 0.10 & 1.00 \\ 
  \hline
  & 30 & 0.54 & 0.01 & 0.00 & 0.07 & 0.97 & 0.53 & 0.11 & 0.06 & 0.14 & 0.94 \\ 
  \textbf{NG2} & 50 & 0.54 & 0.00 & 0.00 & 0.03 & 1.00 & 0.52 & 0.08 & 0.03 & 0.08 & 0.98 \\ 
  & 250 & 0.53 & 0.00 & 0.00 & 0.00 & 1.00 & 0.55 & 0.00 & 0.00 & 0.00 & 1.00 \\ 
  & 4000 & 0.65 & 0.00 & 0.00 & 0.00 & 1.00 & 0.63 & 0.00 & 0.00 & 0.00 & 1.00 \\ 
  \hline
  & 30 & 0.98 & 0.96 & 0.92 & 0.74 & 0.15 & 0.98 & 0.22 & 0.18 & 0.07 & 0.02 \\ 
\textbf{GCV} & 50 & 1.00 & 0.99 & 0.99 & 0.90 & 0.14 & 1.00 & 0.20 & 0.14 & 0.05 & 0.00 \\ 
& 250 & 1.00 & 1.00 & 1.00 & 1.00 & 0.02 & 1.00 & 0.08 & 0.04 & 0.01 & 0.00 \\ 
& 4000 & 1.00 & 1.00 & 1.00 & 1.00 & 0.00 & 1.00 & 0.00 & 0.00 & 0.00 & 0.00 \\ 
   \hline
\end{tabular}
\end{table}

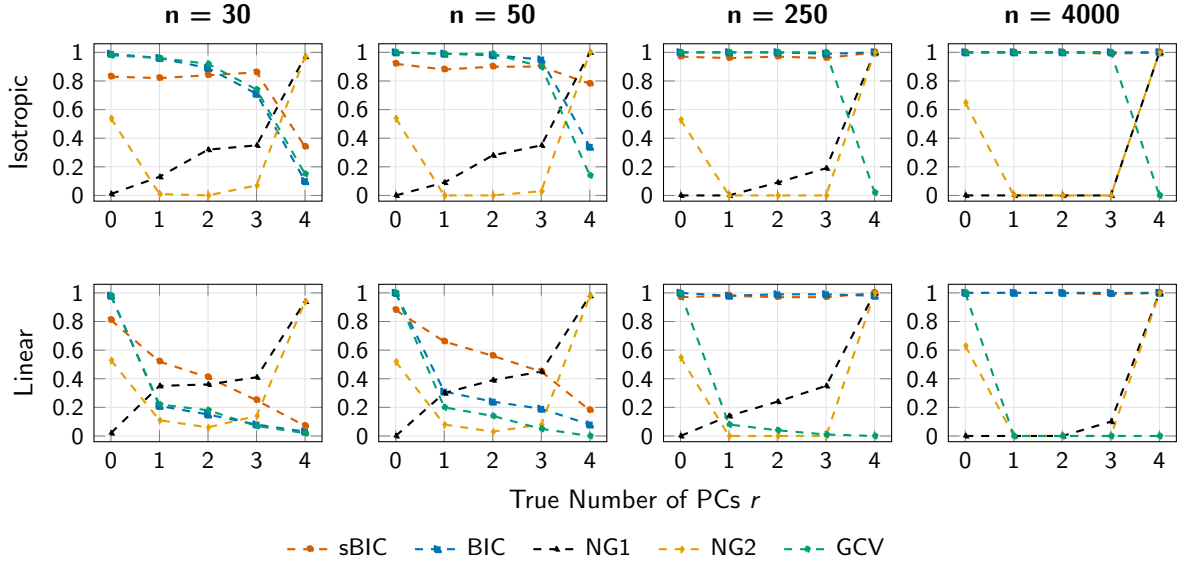
\begin{figure}[H]
    \centering
\begin{tikzpicture}
\begin{groupplot}[
  group style={group size=4 by 2, horizontal sep=0.75cm, vertical sep=1.1cm},
  width=0.295\linewidth, height=0.235\linewidth,
  xmin=-0.35, xmax=4.35, ymin=-0.04, ymax=1.06,
  xtick={0,1,2,3,4}, ytick={0,0.2,0.4,0.6,0.8,1.0},
  yticklabel style={/pgf/number format/fixed, /pgf/number format/precision=1},
  label style={font=\small\sffamily\sansmath},
  tick label style={font=\footnotesize\sffamily\sansmath},
  title style={font=\small\sffamily\bfseries, yshift=-2pt},
  grid=major, grid style={line width=.2pt, draw=gray!22},
  legend style={font=\footnotesize\sffamily, draw=none, fill=none, legend columns=5,
                /tikz/every even column/.append style={column sep=8pt}},
  xlabel={}, ylabel={},
]
\nextgroupplot[title={n = 30}, ylabel={Isotropic}, legend to name=leg:PropP5]
\addplot[color=sbiccol, mark=*, mark size=1.1pt, thick, dashed] coordinates {(0,0.83) (1,0.82) (2,0.84) (3,0.86) (4,0.34)};
\addplot[color=biccol, mark=square*, mark size=1.1pt, thick, dashed] coordinates {(0,0.99) (1,0.96) (2,0.89) (3,0.71) (4,0.1)};
\addplot[color=ngonecol, mark=triangle*, mark size=1.1pt, thick, dashed] coordinates {(0,0.01) (1,0.13) (2,0.32) (3,0.35) (4,0.97)};
\addplot[color=ngtwocol, mark=diamond*, mark size=1.1pt, thick, dashed] coordinates {(0,0.54) (1,0.01) (2,0) (3,0.07) (4,0.97)};
\addplot[color=gcvcol, mark=pentagon*, mark size=1.1pt, thick, dashed] coordinates {(0,0.98) (1,0.96) (2,0.92) (3,0.74) (4,0.15)};
\legend{sBIC,BIC,NG1,NG2,GCV}

\nextgroupplot[title={n = 50}]
\addplot[color=sbiccol, mark=*, mark size=1.1pt, thick, dashed] coordinates {(0,0.92) (1,0.88) (2,0.9) (3,0.9) (4,0.78)};
\addplot[color=biccol, mark=square*, mark size=1.1pt, thick, dashed] coordinates {(0,1) (1,0.99) (2,0.98) (3,0.95) (4,0.34)};
\addplot[color=ngonecol, mark=triangle*, mark size=1.1pt, thick, dashed] coordinates {(0,0) (1,0.09) (2,0.28) (3,0.35) (4,1)};
\addplot[color=ngtwocol, mark=diamond*, mark size=1.1pt, thick, dashed] coordinates {(0,0.54) (1,0) (2,0) (3,0.03) (4,1)};
\addplot[color=gcvcol, mark=pentagon*, mark size=1.1pt, thick, dashed] coordinates {(0,1) (1,0.99) (2,0.99) (3,0.9) (4,0.14)};

\nextgroupplot[title={n = 250}]
\addplot[color=sbiccol, mark=*, mark size=1.1pt, thick, dashed] coordinates {(0,0.97) (1,0.96) (2,0.97) (3,0.96) (4,1)};
\addplot[color=biccol, mark=square*, mark size=1.1pt, thick, dashed] coordinates {(0,1) (1,1) (2,1) (3,0.99) (4,1)};
\addplot[color=ngonecol, mark=triangle*, mark size=1.1pt, thick, dashed] coordinates {(0,0) (1,0) (2,0.09) (3,0.19) (4,1)};
\addplot[color=ngtwocol, mark=diamond*, mark size=1.1pt, thick, dashed] coordinates {(0,0.53) (1,0) (2,0) (3,0) (4,1)};
\addplot[color=gcvcol, mark=pentagon*, mark size=1.1pt, thick, dashed] coordinates {(0,1) (1,1) (2,1) (3,1) (4,0.02)};

\nextgroupplot[title={n = 4000}]
\addplot[color=sbiccol, mark=*, mark size=1.1pt, thick, dashed] coordinates {(0,1) (1,1) (2,1) (3,0.99) (4,1)};
\addplot[color=biccol, mark=square*, mark size=1.1pt, thick, dashed] coordinates {(0,1) (1,1) (2,1) (3,1) (4,1)};
\addplot[color=ngonecol, mark=triangle*, mark size=1.1pt, thick, dashed] coordinates {(0,0) (1,0) (2,0) (3,0) (4,1)};
\addplot[color=ngtwocol, mark=diamond*, mark size=1.1pt, thick, dashed] coordinates {(0,0.65) (1,0) (2,0) (3,0) (4,1)};
\addplot[color=gcvcol, mark=pentagon*, mark size=1.1pt, thick, dashed] coordinates {(0,1) (1,1) (2,1) (3,1) (4,0)};

\nextgroupplot[ylabel={Linear}]
\addplot[color=sbiccol, mark=*, mark size=1.1pt, thick, dashed] coordinates {(0,0.81) (1,0.52) (2,0.41) (3,0.25) (4,0.07)};
\addplot[color=biccol, mark=square*, mark size=1.1pt, thick, dashed] coordinates {(0,0.98) (1,0.21) (2,0.15) (3,0.08) (4,0.03)};
\addplot[color=ngonecol, mark=triangle*, mark size=1.1pt, thick, dashed] coordinates {(0,0.02) (1,0.35) (2,0.36) (3,0.41) (4,0.94)};
\addplot[color=ngtwocol, mark=diamond*, mark size=1.1pt, thick, dashed] coordinates {(0,0.53) (1,0.11) (2,0.06) (3,0.14) (4,0.94)};
\addplot[color=gcvcol, mark=pentagon*, mark size=1.1pt, thick, dashed] coordinates {(0,0.98) (1,0.22) (2,0.18) (3,0.07) (4,0.02)};

\nextgroupplot[]
\addplot[color=sbiccol, mark=*, mark size=1.1pt, thick, dashed] coordinates {(0,0.88) (1,0.66) (2,0.56) (3,0.45) (4,0.18)};
\addplot[color=biccol, mark=square*, mark size=1.1pt, thick, dashed] coordinates {(0,1) (1,0.31) (2,0.24) (3,0.19) (4,0.08)};
\addplot[color=ngonecol, mark=triangle*, mark size=1.1pt, thick, dashed] coordinates {(0,0) (1,0.3) (2,0.39) (3,0.45) (4,0.98)};
\addplot[color=ngtwocol, mark=diamond*, mark size=1.1pt, thick, dashed] coordinates {(0,0.52) (1,0.08) (2,0.03) (3,0.08) (4,0.98)};
\addplot[color=gcvcol, mark=pentagon*, mark size=1.1pt, thick, dashed] coordinates {(0,1) (1,0.2) (2,0.14) (3,0.05) (4,0)};

\nextgroupplot[]
\addplot[color=sbiccol, mark=*, mark size=1.1pt, thick, dashed] coordinates {(0,0.97) (1,0.98) (2,0.97) (3,0.97) (4,1)};
\addplot[color=biccol, mark=square*, mark size=1.1pt, thick, dashed] coordinates {(0,1) (1,0.98) (2,0.99) (3,0.99) (4,0.98)};
\addplot[color=ngonecol, mark=triangle*, mark size=1.1pt, thick, dashed] coordinates {(0,0) (1,0.14) (2,0.24) (3,0.35) (4,1)};
\addplot[color=ngtwocol, mark=diamond*, mark size=1.1pt, thick, dashed] coordinates {(0,0.55) (1,0) (2,0) (3,0) (4,1)};
\addplot[color=gcvcol, mark=pentagon*, mark size=1.1pt, thick, dashed] coordinates {(0,1) (1,0.08) (2,0.04) (3,0.01) (4,0)};

\nextgroupplot[]
\addplot[color=sbiccol, mark=*, mark size=1.1pt, thick, dashed] coordinates {(0,1) (1,1) (2,1) (3,0.99) (4,1)};
\addplot[color=biccol, mark=square*, mark size=1.1pt, thick, dashed] coordinates {(0,1) (1,1) (2,1) (3,1) (4,1)};
\addplot[color=ngonecol, mark=triangle*, mark size=1.1pt, thick, dashed] coordinates {(0,0) (1,0) (2,0) (3,0.1) (4,1)};
\addplot[color=ngtwocol, mark=diamond*, mark size=1.1pt, thick, dashed] coordinates {(0,0.63) (1,0) (2,0) (3,0) (4,1)};
\addplot[color=gcvcol, mark=pentagon*, mark size=1.1pt, thick, dashed] coordinates {(0,1) (1,0) (2,0) (3,0) (4,0)};
\end{groupplot}
\node[below=5mm, font=\small\sffamily\sansmath]
  at ($(group c2r2.south east)!0.5!(group c3r2.south west)$) {True Number of PCs $r$};
\end{tikzpicture}

\vspace{1mm}
\ref*{leg:PropP5}

    \caption{Proportion of the $1000$ simulations in which the correct model is identified, for $p = 5$. The upper row shows the isotropic case $\bs{\Sigma}_0 = \diag(5,\ldots,5,1,\ldots,1)$ and the lower row the linear case $\bs{\Sigma}_0 = \diag(r+1,\ldots,2,1,\ldots,1)$. This is the data of Table \ref{tab:PropP5}.}
    \label{fig:PropP5}
\end{figure}

\begin{table}[H]
\centering
\caption{Average distances between the estimated and the true number of PCs. The dimension of the covariance matrix is $p = 5$ and the table is divided into two cases: $\bs{\Sigma}_0 = \diag(5,\ldots,5,1,\ldots,1)$ (isotropic) and $\bs{\Sigma}_0 = \diag(r+1,\ldots,2,1,\ldots,1)$ (linear).}
\label{tab:DistP5}
\vspace{2mm}
\begin{tabular}{|r|r|rrrrr||rrrrr|}
\hline
&& \multicolumn{5}{|c||}{True No.\ of PCs, Isotropic} & \multicolumn{5}{c|}{True No.\ of PCs, Linear}
\\
  \hline
 & $n$ & 0 & 1 & 2 & 3 & 4 & 0 & 1 & 2 & 3 & 4 \\ 
  \hline
&30  & 0.20 & 0.20 & 0.19 & 0.16 & 1.94 & 0.21 & 0.50 & 0.72 & 1.17 & 2.25 \\ 
  \textbf{sBIC} & 50 & 0.09 & 0.13 & 0.11 & 0.10 & 0.67 & 0.13 & 0.35 & 0.48 & 0.69 & 1.59 \\ 
  & 250 & 0.03 & 0.04 & 0.03 & 0.04 & 0.00 & 0.02 & 0.02 & 0.03 & 0.03 & 0.00 \\ 
  & 4000 & 0.00 & 0.00 & 0.00 & 0.00 & 0.00 & 0.00 & 0.00 & 0.00 & 0.01 & 0.00 \\ 
  \hline
  & 30 & 0.01 & 0.04 & 0.14 & 0.62 & 3.49 & 0.01 & 0.79 & 1.36 & 2.12 & 3.34 \\ 
 \textbf{BIC} & 50 & 0.00 & 0.01 & 0.02 & 0.08 & 2.58 & 0.00 & 0.69 & 1.09 & 1.53 & 2.84 \\ 
  & 250 & 0.00 & 0.00 & 0.00 & 0.00 & 0.00 & 0.00 & 0.02 & 0.01 & 0.01 & 0.02 \\ 
  & 4000 & 0.00 & 0.00 & 0.00 & 0.00 & 0.00 & 0.00 & 0.00 & 0.00 & 0.00 & 0.00 \\ 
  \hline
  & 30 & 1.03 & 0.89 & 0.71 & 0.65 & 0.03 & 1.04 & 0.76 & 0.76 & 0.59 & 0.06 \\ 
  \textbf{NG1} & 50 & 1.01 & 0.91 & 0.74 & 0.65 & 0.00 & 1.01 & 0.74 & 0.65 & 0.55 & 0.02 \\ 
  & 250 & 1.00 & 0.99 & 0.91 & 0.81 & 0.00 & 1.00 & 0.87 & 0.76 & 0.65 & 0.00 \\ 
  & 4000 & 1.00 & 1.00 & 1.00 & 1.00 & 0.00 & 1.00 & 1.00 & 1.00 & 0.90 & 0.00 \\ 
  \hline
  & 30 & 1.57 & 2.54 & 1.82 & 0.93 & 0.03 & 1.64 & 2.08 & 1.68 & 0.86 & 0.06 \\ 
  \textbf{NG2} & 50 & 1.57 & 2.62 & 1.91 & 0.97 & 0.00 & 1.61 & 2.16 & 1.71 & 0.92 & 0.02 \\ 
  & 250 & 1.25 & 2.82 & 1.99 & 1.00 & 0.00 & 1.21 & 2.48 & 1.94 & 1.00 & 0.00 \\ 
  & 4000 & 0.41 & 3.00 & 2.00 & 1.00 & 0.00 & 0.45 & 2.66 & 2.00 & 1.00 & 0.00 \\ 
  \hline
  & 30 & 0.02 & 0.04 & 0.10 & 0.54 & 3.23 & 0.01 & 0.78 & 1.30 & 2.07 & 3.29 \\ 
  \textbf{GCV} & 50 & 0.00 & 0.01 & 0.01 & 0.22 & 3.41 & 0.00 & 0.80 & 1.33 & 2.03 & 3.45 \\ 
  & 250 & 0.00 & 0.00 & 0.00 & 0.00 & 3.90 & 0.00 & 0.92 & 1.26 & 2.09 & 3.88 \\ 
  & 4000 & 0.00 & 0.00 & 0.00 & 0.00 & 4.00 & 0.00 & 1.00 & 1.00 & 2.21 & 4.00 \\ 
   \hline
\end{tabular}
\end{table}

\begin{figure}[H]
    \centering
\begin{tikzpicture}
\begin{groupplot}[
  group style={group size=4 by 2, horizontal sep=0.75cm, vertical sep=1.1cm},
  width=0.295\linewidth, height=0.235\linewidth,
  xmin=-0.35, xmax=4.35, ymin=-0.16, ymax=4.24,
  xtick={0,1,2,3,4}, ytick={0,1,2,3,4},
  yticklabel style={/pgf/number format/fixed, /pgf/number format/precision=0},
  label style={font=\small\sffamily\sansmath},
  tick label style={font=\footnotesize\sffamily\sansmath},
  title style={font=\small\sffamily\bfseries, yshift=-2pt},
  grid=major, grid style={line width=.2pt, draw=gray!22},
  legend style={font=\footnotesize\sffamily, draw=none, fill=none, legend columns=5,
                /tikz/every even column/.append style={column sep=8pt}},
  xlabel={}, ylabel={},
]
\nextgroupplot[title={n = 30}, ylabel={Isotropic}, legend to name=leg:DistP5]
\addplot[color=sbiccol, mark=*, mark size=1.1pt, thick, dashed] coordinates {(0,0.2) (1,0.2) (2,0.19) (3,0.16) (4,1.94)};
\addplot[color=biccol, mark=square*, mark size=1.1pt, thick, dashed] coordinates {(0,0.01) (1,0.04) (2,0.14) (3,0.62) (4,3.49)};
\addplot[color=ngonecol, mark=triangle*, mark size=1.1pt, thick, dashed] coordinates {(0,1.03) (1,0.89) (2,0.71) (3,0.65) (4,0.03)};
\addplot[color=ngtwocol, mark=diamond*, mark size=1.1pt, thick, dashed] coordinates {(0,1.57) (1,2.54) (2,1.82) (3,0.93) (4,0.03)};
\addplot[color=gcvcol, mark=pentagon*, mark size=1.1pt, thick, dashed] coordinates {(0,0.02) (1,0.04) (2,0.1) (3,0.54) (4,3.23)};
\legend{sBIC,BIC,NG1,NG2,GCV}

\nextgroupplot[title={n = 50}]
\addplot[color=sbiccol, mark=*, mark size=1.1pt, thick, dashed] coordinates {(0,0.09) (1,0.13) (2,0.11) (3,0.1) (4,0.67)};
\addplot[color=biccol, mark=square*, mark size=1.1pt, thick, dashed] coordinates {(0,0) (1,0.01) (2,0.02) (3,0.08) (4,2.58)};
\addplot[color=ngonecol, mark=triangle*, mark size=1.1pt, thick, dashed] coordinates {(0,1.01) (1,0.91) (2,0.74) (3,0.65) (4,0)};
\addplot[color=ngtwocol, mark=diamond*, mark size=1.1pt, thick, dashed] coordinates {(0,1.57) (1,2.62) (2,1.91) (3,0.97) (4,0)};
\addplot[color=gcvcol, mark=pentagon*, mark size=1.1pt, thick, dashed] coordinates {(0,0) (1,0.01) (2,0.01) (3,0.22) (4,3.41)};

\nextgroupplot[title={n = 250}]
\addplot[color=sbiccol, mark=*, mark size=1.1pt, thick, dashed] coordinates {(0,0.03) (1,0.04) (2,0.03) (3,0.04) (4,0)};
\addplot[color=biccol, mark=square*, mark size=1.1pt, thick, dashed] coordinates {(0,0) (1,0) (2,0) (3,0) (4,0)};
\addplot[color=ngonecol, mark=triangle*, mark size=1.1pt, thick, dashed] coordinates {(0,1) (1,0.99) (2,0.91) (3,0.81) (4,0)};
\addplot[color=ngtwocol, mark=diamond*, mark size=1.1pt, thick, dashed] coordinates {(0,1.25) (1,2.82) (2,1.99) (3,1) (4,0)};
\addplot[color=gcvcol, mark=pentagon*, mark size=1.1pt, thick, dashed] coordinates {(0,0) (1,0) (2,0) (3,0) (4,3.9)};

\nextgroupplot[title={n = 4000}]
\addplot[color=sbiccol, mark=*, mark size=1.1pt, thick, dashed] coordinates {(0,0) (1,0) (2,0) (3,0) (4,0)};
\addplot[color=biccol, mark=square*, mark size=1.1pt, thick, dashed] coordinates {(0,0) (1,0) (2,0) (3,0) (4,0)};
\addplot[color=ngonecol, mark=triangle*, mark size=1.1pt, thick, dashed] coordinates {(0,1) (1,1) (2,1) (3,1) (4,0)};
\addplot[color=ngtwocol, mark=diamond*, mark size=1.1pt, thick, dashed] coordinates {(0,0.41) (1,3) (2,2) (3,1) (4,0)};
\addplot[color=gcvcol, mark=pentagon*, mark size=1.1pt, thick, dashed] coordinates {(0,0) (1,0) (2,0) (3,0) (4,4)};

\nextgroupplot[ylabel={Linear}]
\addplot[color=sbiccol, mark=*, mark size=1.1pt, thick, dashed] coordinates {(0,0.21) (1,0.5) (2,0.72) (3,1.17) (4,2.25)};
\addplot[color=biccol, mark=square*, mark size=1.1pt, thick, dashed] coordinates {(0,0.01) (1,0.79) (2,1.36) (3,2.12) (4,3.34)};
\addplot[color=ngonecol, mark=triangle*, mark size=1.1pt, thick, dashed] coordinates {(0,1.04) (1,0.76) (2,0.76) (3,0.59) (4,0.06)};
\addplot[color=ngtwocol, mark=diamond*, mark size=1.1pt, thick, dashed] coordinates {(0,1.64) (1,2.08) (2,1.68) (3,0.86) (4,0.06)};
\addplot[color=gcvcol, mark=pentagon*, mark size=1.1pt, thick, dashed] coordinates {(0,0.01) (1,0.78) (2,1.3) (3,2.07) (4,3.29)};

\nextgroupplot[]
\addplot[color=sbiccol, mark=*, mark size=1.1pt, thick, dashed] coordinates {(0,0.13) (1,0.35) (2,0.48) (3,0.69) (4,1.59)};
\addplot[color=biccol, mark=square*, mark size=1.1pt, thick, dashed] coordinates {(0,0) (1,0.69) (2,1.09) (3,1.53) (4,2.84)};
\addplot[color=ngonecol, mark=triangle*, mark size=1.1pt, thick, dashed] coordinates {(0,1.01) (1,0.74) (2,0.65) (3,0.55) (4,0.02)};
\addplot[color=ngtwocol, mark=diamond*, mark size=1.1pt, thick, dashed] coordinates {(0,1.61) (1,2.16) (2,1.71) (3,0.92) (4,0.02)};
\addplot[color=gcvcol, mark=pentagon*, mark size=1.1pt, thick, dashed] coordinates {(0,0) (1,0.8) (2,1.33) (3,2.03) (4,3.45)};

\nextgroupplot[]
\addplot[color=sbiccol, mark=*, mark size=1.1pt, thick, dashed] coordinates {(0,0.02) (1,0.02) (2,0.03) (3,0.03) (4,0)};
\addplot[color=biccol, mark=square*, mark size=1.1pt, thick, dashed] coordinates {(0,0) (1,0.02) (2,0.01) (3,0.01) (4,0.02)};
\addplot[color=ngonecol, mark=triangle*, mark size=1.1pt, thick, dashed] coordinates {(0,1) (1,0.87) (2,0.76) (3,0.65) (4,0)};
\addplot[color=ngtwocol, mark=diamond*, mark size=1.1pt, thick, dashed] coordinates {(0,1.21) (1,2.48) (2,1.94) (3,1) (4,0)};
\addplot[color=gcvcol, mark=pentagon*, mark size=1.1pt, thick, dashed] coordinates {(0,0) (1,0.92) (2,1.26) (3,2.09) (4,3.88)};

\nextgroupplot[]
\addplot[color=sbiccol, mark=*, mark size=1.1pt, thick, dashed] coordinates {(0,0) (1,0) (2,0) (3,0.01) (4,0)};
\addplot[color=biccol, mark=square*, mark size=1.1pt, thick, dashed] coordinates {(0,0) (1,0) (2,0) (3,0) (4,0)};
\addplot[color=ngonecol, mark=triangle*, mark size=1.1pt, thick, dashed] coordinates {(0,1) (1,1) (2,1) (3,0.9) (4,0)};
\addplot[color=ngtwocol, mark=diamond*, mark size=1.1pt, thick, dashed] coordinates {(0,0.45) (1,2.66) (2,2) (3,1) (4,0)};
\addplot[color=gcvcol, mark=pentagon*, mark size=1.1pt, thick, dashed] coordinates {(0,0) (1,1) (2,1) (3,2.21) (4,4)};
\end{groupplot}
\node[below=5mm, font=\small\sffamily\sansmath]
  at ($(group c2r2.south east)!0.5!(group c3r2.south west)$) {True Number of PCs $r$};
\end{tikzpicture}

\vspace{1mm}
\ref*{leg:DistP5}

    \caption{Average absolute distance between the estimated and the true number of PCs, for $p = 5$. The upper row shows the isotropic case $\bs{\Sigma}_0 = \diag(5,\ldots,5,1,\ldots,1)$ and the lower row the linear case $\bs{\Sigma}_0 = \diag(r+1,\ldots,2,1,\ldots,1)$. This is the data of Table \ref{tab:DistP5}.}
    \label{fig:DistP5}
\end{figure}
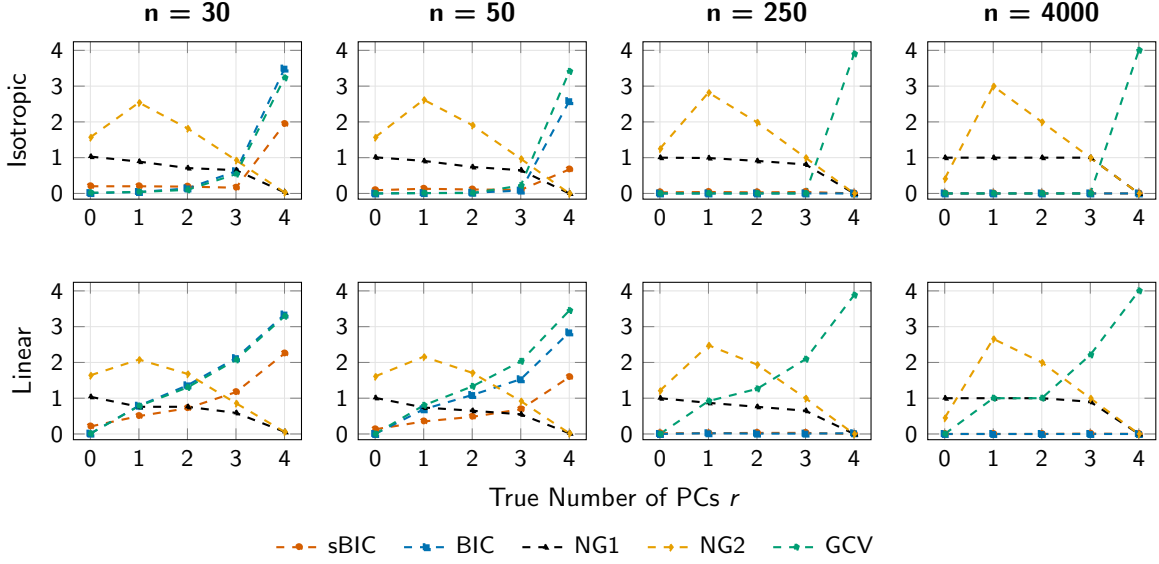

\subsection{Dimension \texorpdfstring{$p = 20$}{p = 20}}

Here $\bs{\Sigma}_0 \in \mb{R}^{20 \times 20}$ and the true number of PCs is $r \in \{0,\ldots,19\}$. The two covariance structures are reported in separate tables.
Table \ref{tab:PropIsotropicP20} and Figure \ref{fig:PropIsotropicP20} report the proportion of the $1000$ runs in which the estimated number of PCs equals the true number $r$, in the isotropic case. Table \ref{tab:PropLinearP20} and Figure \ref{fig:PropLinearP20} report the same proportions in the linear case. Figure \ref{fig:PropLinearP20} repeats Figure \ref{fig:LinearEigenvalueComparisonPlot} of the main text; it is reproduced here so that every table in this appendix is accompanied by the corresponding plot. Table \ref{tab:DistIsotropicP20} and Figure \ref{fig:DistIsotropicP20} report the average absolute distance between the estimated and the true number of PCs, in the isotropic case. Table \ref{tab:DistLinearP20} and Figure \ref{fig:DistLinearP20} report the average absolute distance between the estimated and the true number of PCs, in the linear case.

\begin{table}[H]
\centering
\caption{Proportion of the time a correct model is selected by various methods. The true covariance matrix is the $20 \times 20$ matrix $\bs{\Sigma}_0 = \diag(5,\ldots,5,1,\ldots,1)$. The sample size $n$ and true number of PCs appear in the rows and columns of the table.}
\label{tab:PropIsotropicP20}
\vspace{2mm}
{\small
\begin{tabular}{|r|r|rrrrrrrrrr|}
\hline
& & \multicolumn{10}{|c|}{True Number of PCs}
\\
  \hline
 & $n$ & 0 & 1 & 2 & 3 & 4 & 5 & 6 & 7 & 8 & 9 \\
  \hline
 & 30 & 0.81 & 0.81 & 0.77 & 0.76 & 0.68 & 0.59 & 0.44 & 0.32 & 0.18 & 0.08 \\
\textbf{sBIC} & 50 & 0.94 & 0.94 & 0.93 & 0.93 & 0.94 & 0.90 & 0.83 & 0.71 & 0.56 & 0.42 \\
 & 250 & 1.00 & 1.00 & 1.00 & 0.99 & 1.00 & 0.99 & 1.00 & 1.00 & 0.99 & 0.99 \\
 & 4000 & 1.00 & 1.00 & 1.00 & 1.00 & 1.00 & 1.00 & 1.00 & 1.00 & 1.00 & 1.00 \\
  \hline
 & 30 & 1.00 & 0.67 & 0.40 & 0.24 & 0.11 & 0.03 & 0.01 & 0.00 & 0.00 & 0.00 \\
\textbf{BIC} & 50 & 1.00 & 0.91 & 0.82 & 0.67 & 0.55 & 0.40 & 0.23 & 0.11 & 0.03 & 0.00 \\
 & 250 & 1.00 & 1.00 & 1.00 & 1.00 & 1.00 & 1.00 & 1.00 & 1.00 & 1.00 & 1.00 \\
 & 4000 & 1.00 & 1.00 & 1.00 & 1.00 & 1.00 & 1.00 & 1.00 & 1.00 & 1.00 & 1.00 \\
  \hline
 & 30 & 0.00 & 0.16 & 0.31 & 0.42 & 0.47 & 0.48 & 0.46 & 0.43 & 0.41 & 0.42 \\
\textbf{NG1} & 50 & 0.00 & 0.11 & 0.28 & 0.43 & 0.49 & 0.56 & 0.55 & 0.56 & 0.55 & 0.49 \\
 & 250 & 0.00 & 0.00 & 0.13 & 0.37 & 0.55 & 0.69 & 0.73 & 0.78 & 0.81 & 0.84 \\
 & 4000 & 0.00 & 0.00 & 0.00 & 0.06 & 0.67 & 0.96 & 1.00 & 1.00 & 1.00 & 1.00 \\
  \hline
 & 30 & 0.49 & 0.01 & 0.00 & 0.00 & 0.00 & 0.00 & 0.00 & 0.00 & 0.00 & 0.00 \\
\textbf{NG2} & 50 & 0.49 & 0.00 & 0.00 & 0.00 & 0.00 & 0.00 & 0.00 & 0.00 & 0.00 & 0.00 \\
 & 250 & 0.51 & 0.00 & 0.00 & 0.00 & 0.00 & 0.00 & 0.00 & 0.00 & 0.00 & 0.00 \\
 & 4000 & 0.61 & 0.00 & 0.00 & 0.00 & 0.00 & 0.00 & 0.00 & 0.00 & 0.00 & 0.00 \\
  \hline
 & 30 & 0.99 & 0.71 & 0.74 & 0.70 & 0.67 & 0.62 & 0.49 & 0.42 & 0.30 & 0.21 \\
\textbf{GCV} & 50 & 1.00 & 0.92 & 0.91 & 0.90 & 0.90 & 0.91 & 0.88 & 0.83 & 0.78 & 0.70 \\
 & 250 & 1.00 & 1.00 & 1.00 & 1.00 & 1.00 & 1.00 & 1.00 & 1.00 & 1.00 & 1.00 \\
 & 4000 & 1.00 & 1.00 & 1.00 & 1.00 & 1.00 & 1.00 & 1.00 & 1.00 & 1.00 & 1.00 \\
  \hline
\end{tabular}

\vspace{2mm}

\begin{tabular}{|r|r|rrrrrrrrrr|}
\hline
& & \multicolumn{10}{|c|}{True Number of PCs}
\\
  \hline
 & $n$ & 10 & 11 & 12 & 13 & 14 & 15 & 16 & 17 & 18 & 19 \\
  \hline
 & 30 & 0.04 & 0.02 & 0.00 & 0.00 & 0.00 & 0.00 & 0.00 & 0.00 & 0.00 & 0.00 \\
\textbf{sBIC} & 50 & 0.26 & 0.09 & 0.02 & 0.00 & 0.00 & 0.00 & 0.00 & 0.00 & 0.00 & 0.00 \\
 & 250 & 0.98 & 0.98 & 0.99 & 0.98 & 1.00 & 1.00 & 0.89 & 0.00 & 0.00 & 0.00 \\
 & 4000 & 1.00 & 1.00 & 1.00 & 1.00 & 1.00 & 1.00 & 1.00 & 1.00 & 0.98 & 1.00 \\
  \hline
 & 30 & 0.00 & 0.00 & 0.00 & 0.00 & 0.00 & 0.00 & 0.00 & 0.00 & 0.00 & 0.00 \\
\textbf{BIC} & 50 & 0.00 & 0.00 & 0.00 & 0.00 & 0.00 & 0.00 & 0.00 & 0.00 & 0.00 & 0.00 \\
 & 250 & 1.00 & 1.00 & 1.00 & 1.00 & 1.00 & 0.68 & 0.00 & 0.00 & 0.00 & 0.00 \\
 & 4000 & 1.00 & 1.00 & 1.00 & 1.00 & 1.00 & 1.00 & 1.00 & 1.00 & 1.00 & 1.00 \\
  \hline
 & 30 & 0.42 & 0.38 & 0.38 & 0.37 & 0.35 & 0.32 & 0.30 & 0.32 & 0.31 & 0.61 \\
\textbf{NG1} & 50 & 0.49 & 0.45 & 0.46 & 0.48 & 0.43 & 0.39 & 0.40 & 0.40 & 0.37 & 0.62 \\
 & 250 & 0.83 & 0.83 & 0.82 & 0.81 & 0.79 & 0.77 & 0.75 & 0.74 & 0.71 & 0.80 \\
 & 4000 & 1.00 & 1.00 & 1.00 & 1.00 & 1.00 & 1.00 & 1.00 & 1.00 & 1.00 & 1.00 \\
  \hline
 & 30 & 0.00 & 0.00 & 0.00 & 0.00 & 0.00 & 0.02 & 0.04 & 0.12 & 0.21 & 0.61 \\
\textbf{NG2} & 50 & 0.00 & 0.00 & 0.00 & 0.00 & 0.00 & 0.00 & 0.01 & 0.10 & 0.27 & 0.62 \\
 & 250 & 0.00 & 0.00 & 0.00 & 0.00 & 0.00 & 0.00 & 0.00 & 0.00 & 0.22 & 0.80 \\
 & 4000 & 0.00 & 0.00 & 0.00 & 0.00 & 0.00 & 0.00 & 0.00 & 0.00 & 0.00 & 1.00 \\
  \hline
 & 30 & 0.14 & 0.11 & 0.05 & 0.02 & 0.01 & 0.00 & 0.00 & 0.00 & 0.00 & 0.00 \\
\textbf{GCV} & 50 & 0.62 & 0.47 & 0.39 & 0.20 & 0.06 & 0.01 & 0.00 & 0.00 & 0.00 & 0.00 \\
 & 250 & 1.00 & 1.00 & 1.00 & 1.00 & 0.98 & 0.13 & 0.00 & 0.00 & 0.00 & 0.00 \\
 & 4000 & 1.00 & 1.00 & 1.00 & 1.00 & 1.00 & 0.39 & 0.00 & 0.00 & 0.00 & 0.00 \\
  \hline
\end{tabular}
}
\end{table}

\begin{figure}[H]
    \centering
\begin{tikzpicture}
\begin{groupplot}[
  group style={group size=2 by 2, horizontal sep=1.3cm, vertical sep=1.5cm,
               xlabels at=edge bottom, ylabels at=edge left},
  width=0.505\linewidth, height=0.40\linewidth,
  xmin=-0.6, xmax=19.6, ymin=-0.04, ymax=1.04,
  xtick={0,5,10,15,19}, ytick={0,0.2,0.4,0.6,0.8,1.0},
  yticklabel style={/pgf/number format/fixed, /pgf/number format/precision=1},
  label style={font=\small\sffamily\sansmath},
  tick label style={font=\footnotesize\sffamily\sansmath},
  title style={font=\small\sffamily\bfseries, yshift=-2pt},
  grid=major, grid style={line width=.2pt, draw=gray!22},
  legend style={font=\footnotesize\sffamily, draw=none, fill=none, legend columns=5,
                /tikz/every even column/.append style={column sep=8pt}},
  xlabel={True Number of PCs $r$}, ylabel={Proportion Correct},
]
\nextgroupplot[title={n = 30}, legend to name=leg:PropIsoP20]
\addplot[color=sbiccol, mark=*, mark size=1.1pt, thick, dashed] coordinates {(0,0.81) (1,0.81) (2,0.77) (3,0.76) (4,0.68) (5,0.59) (6,0.44) (7,0.32) (8,0.18) (9,0.08) (10,0.04) (11,0.02) (12,0) (13,0) (14,0) (15,0) (16,0) (17,0) (18,0) (19,0)};
\addplot[color=biccol, mark=square*, mark size=1.1pt, thick, dashed] coordinates {(0,1) (1,0.67) (2,0.4) (3,0.24) (4,0.11) (5,0.03) (6,0.01) (7,0) (8,0) (9,0) (10,0) (11,0) (12,0) (13,0) (14,0) (15,0) (16,0) (17,0) (18,0) (19,0)};
\addplot[color=ngonecol, mark=triangle*, mark size=1.1pt, thick, dashed] coordinates {(0,0) (1,0.16) (2,0.31) (3,0.42) (4,0.47) (5,0.48) (6,0.46) (7,0.43) (8,0.41) (9,0.42) (10,0.42) (11,0.38) (12,0.38) (13,0.37) (14,0.35) (15,0.32) (16,0.3) (17,0.32) (18,0.31) (19,0.61)};
\addplot[color=ngtwocol, mark=diamond*, mark size=1.1pt, thick, dashed] coordinates {(0,0.49) (1,0.01) (2,0) (3,0) (4,0) (5,0) (6,0) (7,0) (8,0) (9,0) (10,0) (11,0) (12,0) (13,0) (14,0) (15,0.02) (16,0.04) (17,0.12) (18,0.21) (19,0.61)};
\addplot[color=gcvcol, mark=pentagon*, mark size=1.1pt, thick, dashed] coordinates {(0,0.99) (1,0.71) (2,0.74) (3,0.7) (4,0.67) (5,0.62) (6,0.49) (7,0.42) (8,0.3) (9,0.21) (10,0.14) (11,0.11) (12,0.05) (13,0.02) (14,0.01) (15,0) (16,0) (17,0) (18,0) (19,0)};
\legend{sBIC,BIC,NG1,NG2,GCV}

\nextgroupplot[title={n = 50}]
\addplot[color=sbiccol, mark=*, mark size=1.1pt, thick, dashed] coordinates {(0,0.94) (1,0.94) (2,0.93) (3,0.93) (4,0.94) (5,0.9) (6,0.83) (7,0.71) (8,0.56) (9,0.42) (10,0.26) (11,0.09) (12,0.02) (13,0) (14,0) (15,0) (16,0) (17,0) (18,0) (19,0)};
\addplot[color=biccol, mark=square*, mark size=1.1pt, thick, dashed] coordinates {(0,1) (1,0.91) (2,0.82) (3,0.67) (4,0.55) (5,0.4) (6,0.23) (7,0.11) (8,0.03) (9,0) (10,0) (11,0) (12,0) (13,0) (14,0) (15,0) (16,0) (17,0) (18,0) (19,0)};
\addplot[color=ngonecol, mark=triangle*, mark size=1.1pt, thick, dashed] coordinates {(0,0) (1,0.11) (2,0.28) (3,0.43) (4,0.49) (5,0.56) (6,0.55) (7,0.56) (8,0.55) (9,0.49) (10,0.49) (11,0.45) (12,0.46) (13,0.48) (14,0.43) (15,0.39) (16,0.4) (17,0.4) (18,0.37) (19,0.62)};
\addplot[color=ngtwocol, mark=diamond*, mark size=1.1pt, thick, dashed] coordinates {(0,0.49) (1,0) (2,0) (3,0) (4,0) (5,0) (6,0) (7,0) (8,0) (9,0) (10,0) (11,0) (12,0) (13,0) (14,0) (15,0) (16,0.01) (17,0.1) (18,0.27) (19,0.62)};
\addplot[color=gcvcol, mark=pentagon*, mark size=1.1pt, thick, dashed] coordinates {(0,1) (1,0.92) (2,0.91) (3,0.9) (4,0.9) (5,0.91) (6,0.88) (7,0.83) (8,0.78) (9,0.7) (10,0.62) (11,0.47) (12,0.39) (13,0.2) (14,0.06) (15,0.01) (16,0) (17,0) (18,0) (19,0)};

\nextgroupplot[title={n = 250}]
\addplot[color=sbiccol, mark=*, mark size=1.1pt, thick, dashed] coordinates {(0,1) (1,1) (2,1) (3,0.99) (4,1) (5,0.99) (6,1) (7,1) (8,0.99) (9,0.99) (10,0.98) (11,0.98) (12,0.99) (13,0.98) (14,1) (15,1) (16,0.89) (17,0) (18,0) (19,0)};
\addplot[color=biccol, mark=square*, mark size=1.1pt, thick, dashed] coordinates {(0,1) (1,1) (2,1) (3,1) (4,1) (5,1) (6,1) (7,1) (8,1) (9,1) (10,1) (11,1) (12,1) (13,1) (14,1) (15,0.68) (16,0) (17,0) (18,0) (19,0)};
\addplot[color=ngonecol, mark=triangle*, mark size=1.1pt, thick, dashed] coordinates {(0,0) (1,0) (2,0.13) (3,0.37) (4,0.55) (5,0.69) (6,0.73) (7,0.78) (8,0.81) (9,0.84) (10,0.83) (11,0.83) (12,0.82) (13,0.81) (14,0.79) (15,0.77) (16,0.75) (17,0.74) (18,0.71) (19,0.8)};
\addplot[color=ngtwocol, mark=diamond*, mark size=1.1pt, thick, dashed] coordinates {(0,0.51) (1,0) (2,0) (3,0) (4,0) (5,0) (6,0) (7,0) (8,0) (9,0) (10,0) (11,0) (12,0) (13,0) (14,0) (15,0) (16,0) (17,0) (18,0.22) (19,0.8)};
\addplot[color=gcvcol, mark=pentagon*, mark size=1.1pt, thick, dashed] coordinates {(0,1) (1,1) (2,1) (3,1) (4,1) (5,1) (6,1) (7,1) (8,1) (9,1) (10,1) (11,1) (12,1) (13,1) (14,0.98) (15,0.13) (16,0) (17,0) (18,0) (19,0)};

\nextgroupplot[title={n = 4000}]
\addplot[color=sbiccol, mark=*, mark size=1.1pt, thick, dashed] coordinates {(0,1) (1,1) (2,1) (3,1) (4,1) (5,1) (6,1) (7,1) (8,1) (9,1) (10,1) (11,1) (12,1) (13,1) (14,1) (15,1) (16,1) (17,1) (18,0.98) (19,1)};
\addplot[color=biccol, mark=square*, mark size=1.1pt, thick, dashed] coordinates {(0,1) (1,1) (2,1) (3,1) (4,1) (5,1) (6,1) (7,1) (8,1) (9,1) (10,1) (11,1) (12,1) (13,1) (14,1) (15,1) (16,1) (17,1) (18,1) (19,1)};
\addplot[color=ngonecol, mark=triangle*, mark size=1.1pt, thick, dashed] coordinates {(0,0) (1,0) (2,0) (3,0.06) (4,0.67) (5,0.96) (6,1) (7,1) (8,1) (9,1) (10,1) (11,1) (12,1) (13,1) (14,1) (15,1) (16,1) (17,1) (18,1) (19,1)};
\addplot[color=ngtwocol, mark=diamond*, mark size=1.1pt, thick, dashed] coordinates {(0,0.61) (1,0) (2,0) (3,0) (4,0) (5,0) (6,0) (7,0) (8,0) (9,0) (10,0) (11,0) (12,0) (13,0) (14,0) (15,0) (16,0) (17,0) (18,0) (19,1)};
\addplot[color=gcvcol, mark=pentagon*, mark size=1.1pt, thick, dashed] coordinates {(0,1) (1,1) (2,1) (3,1) (4,1) (5,1) (6,1) (7,1) (8,1) (9,1) (10,1) (11,1) (12,1) (13,1) (14,1) (15,0.39) (16,0) (17,0) (18,0) (19,0)};
\end{groupplot}
\end{tikzpicture}

\vspace{1mm}
\ref*{leg:PropIsoP20}

    \caption{Proportion of the $1000$ simulations in which the correct model is identified, for $\bs{\Sigma}_0 = \diag(5,\ldots,5,1,\ldots,1) \in \mb{R}^{20 \times 20}$. This is the data of Table \ref{tab:PropIsotropicP20}.}
    \label{fig:PropIsotropicP20}
\end{figure}
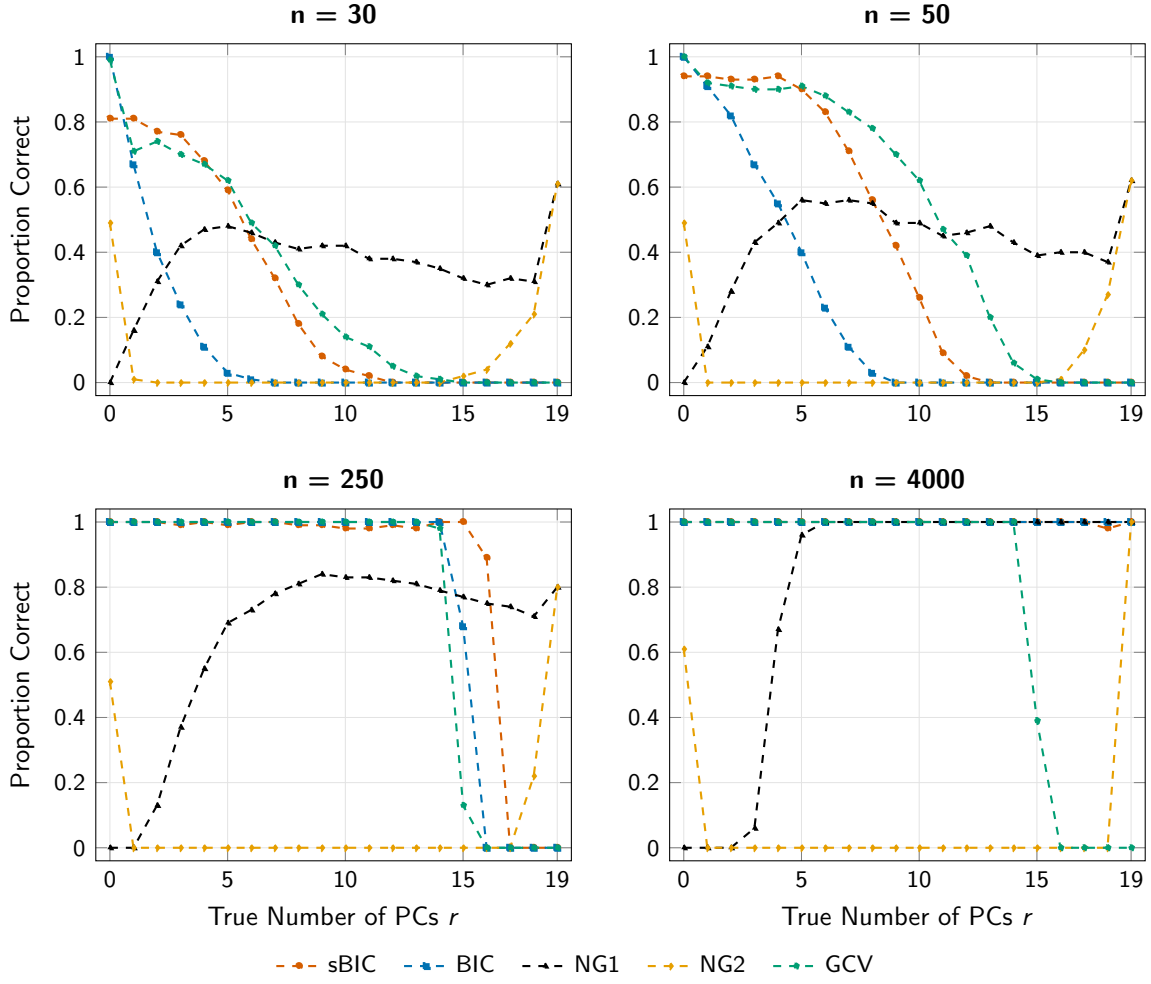

\begin{table}[H]
\centering
\caption{Proportion of the time a correct model is selected by various methods. The true covariance matrix is the $20 \times 20$ matrix $\bs{\Sigma}_0 = \diag(r+1,\ldots,2,1,\ldots,1)$. The sample size $n$ and true number of PCs appear in the rows and columns of the table. This data is plotted in Fig \ref{fig:LinearEigenvalueComparisonPlot}.}
\label{tab:PropLinearP20}
\vspace{2mm}
{\small
\begin{tabular}{|r|r|rrrrrrrrrr|}
\hline
& & \multicolumn{10}{|c|}{True Number of PCs}
\\
  \hline
 & $n$ & 0 & 1 & 2 & 3 & 4 & 5 & 6 & 7 & 8 & 9 \\
  \hline
 & 30 & 0.80 & 0.31 & 0.21 & 0.20 & 0.14 & 0.10 & 0.08 & 0.09 & 0.06 & 0.04 \\
\textbf{sBIC} & 50 & 0.94 & 0.36 & 0.25 & 0.19 & 0.19 & 0.14 & 0.14 & 0.11 & 0.09 & 0.06 \\
 & 250 & 1.00 & 0.95 & 0.96 & 0.96 & 0.99 & 0.98 & 0.98 & 0.99 & 0.99 & 0.99 \\
 & 4000 & 1.00 & 1.00 & 1.00 & 1.00 & 1.00 & 1.00 & 1.00 & 1.00 & 1.00 & 1.00 \\
  \hline
 & 30 & 1.00 & 0.00 & 0.00 & 0.00 & 0.00 & 0.00 & 0.00 & 0.00 & 0.00 & 0.00 \\
\textbf{BIC} & 50 & 1.00 & 0.00 & 0.00 & 0.00 & 0.00 & 0.00 & 0.00 & 0.00 & 0.00 & 0.00 \\
 & 250 & 1.00 & 0.18 & 0.20 & 0.26 & 0.31 & 0.39 & 0.45 & 0.54 & 0.62 & 0.69 \\
 & 4000 & 1.00 & 1.00 & 1.00 & 1.00 & 1.00 & 1.00 & 1.00 & 1.00 & 1.00 & 1.00 \\
  \hline
 & 30 & 0.00 & 0.39 & 0.41 & 0.43 & 0.45 & 0.46 & 0.46 & 0.46 & 0.43 & 0.40 \\
\textbf{NG1} & 50 & 0.00 & 0.43 & 0.46 & 0.47 & 0.54 & 0.56 & 0.56 & 0.52 & 0.53 & 0.53 \\
 & 250 & 0.00 & 0.33 & 0.49 & 0.57 & 0.69 & 0.75 & 0.82 & 0.87 & 0.89 & 0.89 \\
 & 4000 & 0.00 & 0.05 & 0.48 & 0.76 & 0.97 & 1.00 & 1.00 & 1.00 & 1.00 & 1.00 \\
  \hline
 & 30 & 0.49 & 0.02 & 0.01 & 0.00 & 0.00 & 0.00 & 0.00 & 0.00 & 0.00 & 0.00 \\
\textbf{NG2} & 50 & 0.51 & 0.02 & 0.00 & 0.00 & 0.00 & 0.00 & 0.00 & 0.00 & 0.00 & 0.00 \\
 & 250 & 0.53 & 0.01 & 0.00 & 0.00 & 0.00 & 0.00 & 0.00 & 0.00 & 0.00 & 0.00 \\
 & 4000 & 0.56 & 0.00 & 0.00 & 0.00 & 0.00 & 0.00 & 0.00 & 0.00 & 0.00 & 0.00 \\
  \hline
 & 30 & 0.99 & 0.04 & 0.16 & 0.22 & 0.20 & 0.14 & 0.14 & 0.15 & 0.12 & 0.11 \\
\textbf{GCV} & 50 & 0.99 & 0.10 & 0.29 & 0.28 & 0.26 & 0.27 & 0.25 & 0.21 & 0.20 & 0.17 \\
 & 250 & 1.00 & 0.21 & 0.32 & 0.32 & 0.31 & 0.31 & 0.29 & 0.28 & 0.27 & 0.25 \\
 & 4000 & 1.00 & 0.12 & 0.11 & 0.09 & 0.08 & 0.08 & 0.07 & 0.04 & 0.03 & 0.01 \\
  \hline
\end{tabular}

\vspace{2mm}

\begin{tabular}{|r|r|rrrrrrrrrr|}
\hline
& & \multicolumn{10}{|c|}{True Number of PCs}
\\
  \hline
 & $n$ & 10 & 11 & 12 & 13 & 14 & 15 & 16 & 17 & 18 & 19 \\
  \hline
 & 30 & 0.02 & 0.02 & 0.01 & 0.00 & 0.00 & 0.00 & 0.00 & 0.00 & 0.00 & 0.00 \\
\textbf{sBIC} & 50 & 0.06 & 0.03 & 0.02 & 0.01 & 0.00 & 0.00 & 0.00 & 0.00 & 0.00 & 0.00 \\
 & 250 & 0.99 & 0.99 & 0.99 & 0.99 & 0.98 & 0.98 & 0.97 & 0.96 & 0.96 & 0.78 \\
 & 4000 & 1.00 & 1.00 & 1.00 & 1.00 & 1.00 & 1.00 & 1.00 & 0.99 & 0.99 & 1.00 \\
  \hline
 & 30 & 0.00 & 0.00 & 0.00 & 0.00 & 0.00 & 0.00 & 0.00 & 0.00 & 0.00 & 0.00 \\
\textbf{BIC} & 50 & 0.00 & 0.01 & 0.00 & 0.01 & 0.00 & 0.00 & 0.00 & 0.00 & 0.00 & 0.00 \\
 & 250 & 0.73 & 0.82 & 0.85 & 0.91 & 0.92 & 0.96 & 0.97 & 0.98 & 0.99 & 0.98 \\
 & 4000 & 1.00 & 1.00 & 1.00 & 1.00 & 1.00 & 1.00 & 1.00 & 1.00 & 1.00 & 1.00 \\
  \hline
 & 30 & 0.40 & 0.37 & 0.35 & 0.35 & 0.33 & 0.32 & 0.32 & 0.31 & 0.31 & 0.61 \\
\textbf{NG1} & 50 & 0.51 & 0.45 & 0.47 & 0.44 & 0.42 & 0.40 & 0.40 & 0.38 & 0.37 & 0.63 \\
 & 250 & 0.87 & 0.85 & 0.82 & 0.79 & 0.78 & 0.77 & 0.74 & 0.71 & 0.72 & 0.75 \\
 & 4000 & 1.00 & 1.00 & 1.00 & 1.00 & 1.00 & 1.00 & 1.00 & 1.00 & 1.00 & 1.00 \\
  \hline
 & 30 & 0.00 & 0.00 & 0.00 & 0.00 & 0.00 & 0.01 & 0.05 & 0.14 & 0.25 & 0.61 \\
\textbf{NG2} & 50 & 0.00 & 0.00 & 0.00 & 0.00 & 0.00 & 0.00 & 0.02 & 0.10 & 0.30 & 0.63 \\
 & 250 & 0.00 & 0.00 & 0.00 & 0.00 & 0.00 & 0.00 & 0.00 & 0.01 & 0.28 & 0.75 \\
 & 4000 & 0.00 & 0.00 & 0.00 & 0.00 & 0.00 & 0.00 & 0.00 & 0.00 & 0.02 & 1.00 \\
  \hline
 & 30 & 0.10 & 0.08 & 0.07 & 0.07 & 0.06 & 0.03 & 0.03 & 0.02 & 0.01 & 0.01 \\
\textbf{GCV} & 50 & 0.17 & 0.17 & 0.14 & 0.13 & 0.12 & 0.10 & 0.06 & 0.04 & 0.00 & 0.00 \\
 & 250 & 0.23 & 0.20 & 0.19 & 0.16 & 0.13 & 0.10 & 0.07 & 0.03 & 0.00 & 0.00 \\
 & 4000 & 0.02 & 0.01 & 0.00 & 0.00 & 0.00 & 0.00 & 0.00 & 0.00 & 0.00 & 0.00 \\
  \hline
\end{tabular}
}
\end{table}

\begin{figure}[H]
    \centering
\begin{tikzpicture}
\begin{groupplot}[
  group style={group size=2 by 2, horizontal sep=1.3cm, vertical sep=1.5cm,
               xlabels at=edge bottom, ylabels at=edge left},
  width=0.505\linewidth, height=0.40\linewidth,
  xmin=-0.6, xmax=19.6, ymin=-0.04, ymax=1.04,
  xtick={0,5,10,15,19}, ytick={0,0.2,0.4,0.6,0.8,1.0},
  yticklabel style={/pgf/number format/fixed, /pgf/number format/precision=1},
  label style={font=\small\sffamily\sansmath},
  tick label style={font=\footnotesize\sffamily\sansmath},
  title style={font=\small\sffamily\bfseries, yshift=-2pt},
  grid=major, grid style={line width=.2pt, draw=gray!22},
  legend style={font=\footnotesize\sffamily, draw=none, fill=none, legend columns=5,
                /tikz/every even column/.append style={column sep=8pt}},
  xlabel={True Number of PCs $r$}, ylabel={Proportion Correct},
]
\nextgroupplot[title={n = 30}, legend to name=leg:PropLinP20]
\addplot[color=sbiccol, mark=*, mark size=1.1pt, thick, dashed] coordinates {(0,0.8) (1,0.31) (2,0.21) (3,0.2) (4,0.14) (5,0.1) (6,0.08) (7,0.09) (8,0.06) (9,0.04) (10,0.02) (11,0.02) (12,0.01) (13,0) (14,0) (15,0) (16,0) (17,0) (18,0) (19,0)};
\addplot[color=biccol, mark=square*, mark size=1.1pt, thick, dashed] coordinates {(0,1) (1,0) (2,0) (3,0) (4,0) (5,0) (6,0) (7,0) (8,0) (9,0) (10,0) (11,0) (12,0) (13,0) (14,0) (15,0) (16,0) (17,0) (18,0) (19,0)};
\addplot[color=ngonecol, mark=triangle*, mark size=1.1pt, thick, dashed] coordinates {(0,0) (1,0.39) (2,0.41) (3,0.43) (4,0.45) (5,0.46) (6,0.46) (7,0.46) (8,0.43) (9,0.4) (10,0.4) (11,0.37) (12,0.35) (13,0.35) (14,0.33) (15,0.32) (16,0.32) (17,0.31) (18,0.31) (19,0.61)};
\addplot[color=ngtwocol, mark=diamond*, mark size=1.1pt, thick, dashed] coordinates {(0,0.49) (1,0.02) (2,0.01) (3,0) (4,0) (5,0) (6,0) (7,0) (8,0) (9,0) (10,0) (11,0) (12,0) (13,0) (14,0) (15,0.01) (16,0.05) (17,0.14) (18,0.25) (19,0.61)};
\addplot[color=gcvcol, mark=pentagon*, mark size=1.1pt, thick, dashed] coordinates {(0,0.99) (1,0.04) (2,0.16) (3,0.22) (4,0.2) (5,0.14) (6,0.14) (7,0.15) (8,0.12) (9,0.11) (10,0.1) (11,0.08) (12,0.07) (13,0.07) (14,0.06) (15,0.03) (16,0.03) (17,0.02) (18,0.01) (19,0.01)};
\legend{sBIC,BIC,NG1,NG2,GCV}

\nextgroupplot[title={n = 50}]
\addplot[color=sbiccol, mark=*, mark size=1.1pt, thick, dashed] coordinates {(0,0.94) (1,0.36) (2,0.25) (3,0.19) (4,0.19) (5,0.14) (6,0.14) (7,0.11) (8,0.09) (9,0.06) (10,0.06) (11,0.03) (12,0.02) (13,0.01) (14,0) (15,0) (16,0) (17,0) (18,0) (19,0)};
\addplot[color=biccol, mark=square*, mark size=1.1pt, thick, dashed] coordinates {(0,1) (1,0) (2,0) (3,0) (4,0) (5,0) (6,0) (7,0) (8,0) (9,0) (10,0) (11,0.01) (12,0) (13,0.01) (14,0) (15,0) (16,0) (17,0) (18,0) (19,0)};
\addplot[color=ngonecol, mark=triangle*, mark size=1.1pt, thick, dashed] coordinates {(0,0) (1,0.43) (2,0.46) (3,0.47) (4,0.54) (5,0.56) (6,0.56) (7,0.52) (8,0.53) (9,0.53) (10,0.51) (11,0.45) (12,0.47) (13,0.44) (14,0.42) (15,0.4) (16,0.4) (17,0.38) (18,0.37) (19,0.63)};
\addplot[color=ngtwocol, mark=diamond*, mark size=1.1pt, thick, dashed] coordinates {(0,0.51) (1,0.02) (2,0) (3,0) (4,0) (5,0) (6,0) (7,0) (8,0) (9,0) (10,0) (11,0) (12,0) (13,0) (14,0) (15,0) (16,0.02) (17,0.1) (18,0.3) (19,0.63)};
\addplot[color=gcvcol, mark=pentagon*, mark size=1.1pt, thick, dashed] coordinates {(0,0.99) (1,0.1) (2,0.29) (3,0.28) (4,0.26) (5,0.27) (6,0.25) (7,0.21) (8,0.2) (9,0.17) (10,0.17) (11,0.17) (12,0.14) (13,0.13) (14,0.12) (15,0.1) (16,0.06) (17,0.04) (18,0) (19,0)};

\nextgroupplot[title={n = 250}]
\addplot[color=sbiccol, mark=*, mark size=1.1pt, thick, dashed] coordinates {(0,1) (1,0.95) (2,0.96) (3,0.96) (4,0.99) (5,0.98) (6,0.98) (7,0.99) (8,0.99) (9,0.99) (10,0.99) (11,0.99) (12,0.99) (13,0.99) (14,0.98) (15,0.98) (16,0.97) (17,0.96) (18,0.96) (19,0.78)};
\addplot[color=biccol, mark=square*, mark size=1.1pt, thick, dashed] coordinates {(0,1) (1,0.18) (2,0.2) (3,0.26) (4,0.31) (5,0.39) (6,0.45) (7,0.54) (8,0.62) (9,0.69) (10,0.73) (11,0.82) (12,0.85) (13,0.91) (14,0.92) (15,0.96) (16,0.97) (17,0.98) (18,0.99) (19,0.98)};
\addplot[color=ngonecol, mark=triangle*, mark size=1.1pt, thick, dashed] coordinates {(0,0) (1,0.33) (2,0.49) (3,0.57) (4,0.69) (5,0.75) (6,0.82) (7,0.87) (8,0.89) (9,0.89) (10,0.87) (11,0.85) (12,0.82) (13,0.79) (14,0.78) (15,0.77) (16,0.74) (17,0.71) (18,0.72) (19,0.75)};
\addplot[color=ngtwocol, mark=diamond*, mark size=1.1pt, thick, dashed] coordinates {(0,0.53) (1,0.01) (2,0) (3,0) (4,0) (5,0) (6,0) (7,0) (8,0) (9,0) (10,0) (11,0) (12,0) (13,0) (14,0) (15,0) (16,0) (17,0.01) (18,0.28) (19,0.75)};
\addplot[color=gcvcol, mark=pentagon*, mark size=1.1pt, thick, dashed] coordinates {(0,1) (1,0.21) (2,0.32) (3,0.32) (4,0.31) (5,0.31) (6,0.29) (7,0.28) (8,0.27) (9,0.25) (10,0.23) (11,0.2) (12,0.19) (13,0.16) (14,0.13) (15,0.1) (16,0.07) (17,0.03) (18,0) (19,0)};

\nextgroupplot[title={n = 4000}]
\addplot[color=sbiccol, mark=*, mark size=1.1pt, thick, dashed] coordinates {(0,1) (1,1) (2,1) (3,1) (4,1) (5,1) (6,1) (7,1) (8,1) (9,1) (10,1) (11,1) (12,1) (13,1) (14,1) (15,1) (16,1) (17,0.99) (18,0.99) (19,1)};
\addplot[color=biccol, mark=square*, mark size=1.1pt, thick, dashed] coordinates {(0,1) (1,1) (2,1) (3,1) (4,1) (5,1) (6,1) (7,1) (8,1) (9,1) (10,1) (11,1) (12,1) (13,1) (14,1) (15,1) (16,1) (17,1) (18,1) (19,1)};
\addplot[color=ngonecol, mark=triangle*, mark size=1.1pt, thick, dashed] coordinates {(0,0) (1,0.05) (2,0.48) (3,0.76) (4,0.97) (5,1) (6,1) (7,1) (8,1) (9,1) (10,1) (11,1) (12,1) (13,1) (14,1) (15,1) (16,1) (17,1) (18,1) (19,1)};
\addplot[color=ngtwocol, mark=diamond*, mark size=1.1pt, thick, dashed] coordinates {(0,0.56) (1,0) (2,0) (3,0) (4,0) (5,0) (6,0) (7,0) (8,0) (9,0) (10,0) (11,0) (12,0) (13,0) (14,0) (15,0) (16,0) (17,0) (18,0.02) (19,1)};
\addplot[color=gcvcol, mark=pentagon*, mark size=1.1pt, thick, dashed] coordinates {(0,1) (1,0.12) (2,0.11) (3,0.09) (4,0.08) (5,0.08) (6,0.07) (7,0.04) (8,0.03) (9,0.01) (10,0.02) (11,0.01) (12,0) (13,0) (14,0) (15,0) (16,0) (17,0) (18,0) (19,0)};
\end{groupplot}
\end{tikzpicture}

\vspace{1mm}
\ref*{leg:PropLinP20}

    \caption{Proportion of the $1000$ simulations in which the correct model is identified, for $\bs{\Sigma}_0 = \diag(r+1,\ldots,2,1,\ldots,1) \in \mb{R}^{20 \times 20}$. This is the data of Table \ref{tab:PropLinearP20}. The figure repeats Figure \ref{fig:LinearEigenvalueComparisonPlot} of the main text.}
    \label{fig:PropLinearP20}
\end{figure}
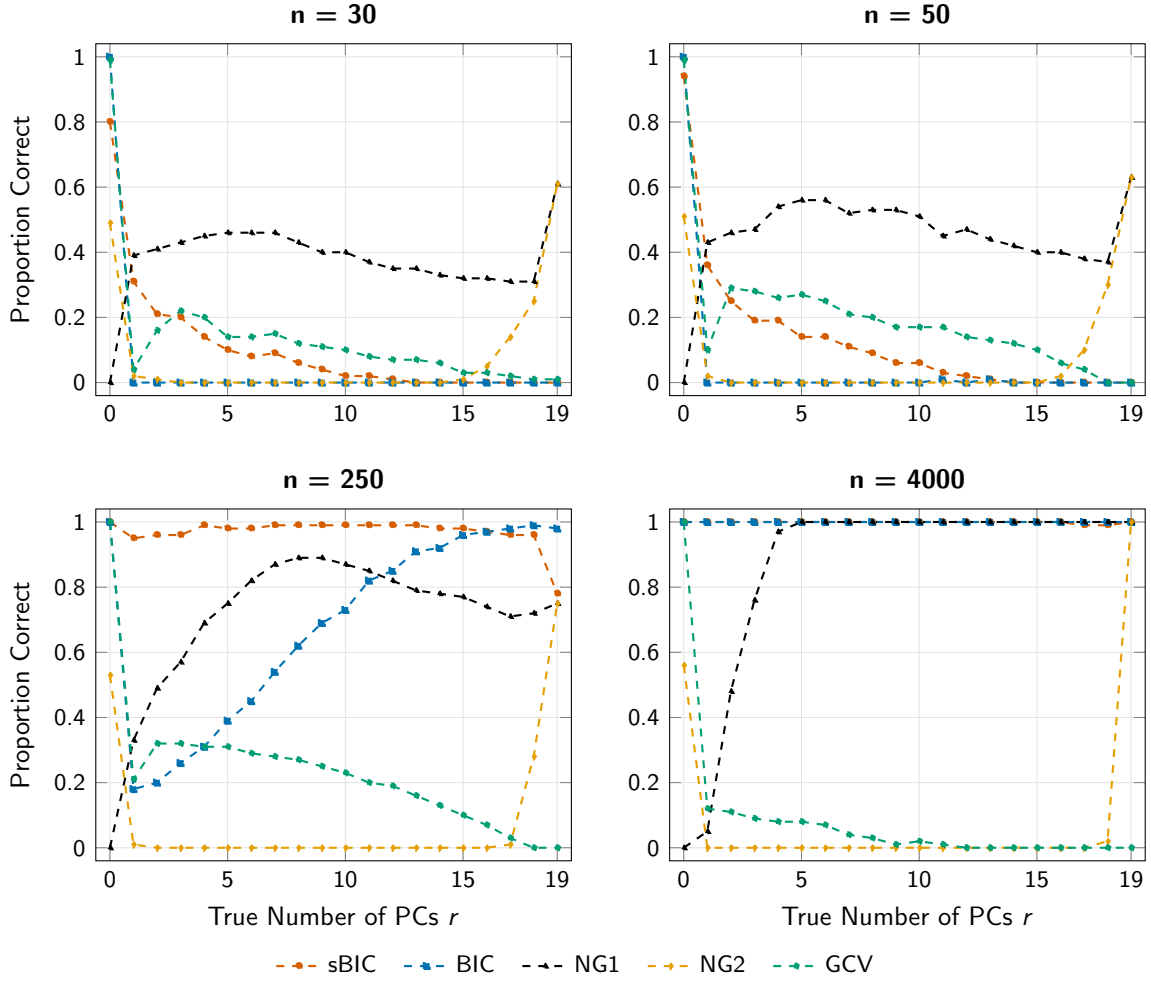

\begin{table}[H]
\centering
\caption{Average absolute distance between the estimated and the true number of PCs for various model selection procedures. The true covariance matrix is the $20 \times 20$ matrix $\bs{\Sigma}_0 = \diag(5,\ldots,5,1,\ldots,1)$. The sample size $n$ and true number of PCs appear in the rows and columns of the table.}
\label{tab:DistIsotropicP20}
\vspace{2mm}
{\small
\begin{tabular}{|r|r|rrrrrrrrrr|}
\hline
& & \multicolumn{10}{|c|}{True Number of PCs}
\\
  \hline
 & $n$ & 0 & 1 & 2 & 3 & 4 & 5 & 6 & 7 & 8 & 9 \\
  \hline
 & 30 & 0.23 & 0.23 & 0.26 & 0.28 & 0.35 & 0.48 & 0.78 & 1.11 & 1.76 & 2.65 \\
\textbf{sBIC} & 50 & 0.06 & 0.06 & 0.07 & 0.07 & 0.06 & 0.10 & 0.19 & 0.34 & 0.61 & 1.03 \\
 & 250 & 0.00 & 0.00 & 0.00 & 0.01 & 0.00 & 0.01 & 0.00 & 0.00 & 0.01 & 0.01 \\
 & 4000 & 0.00 & 0.00 & 0.00 & 0.00 & 0.00 & 0.00 & 0.00 & 0.00 & 0.00 & 0.00 \\
  \hline
 & 30 & 0.00 & 0.33 & 0.79 & 1.41 & 2.33 & 3.58 & 4.99 & 6.38 & 7.68 & 8.87 \\
\textbf{BIC} & 50 & 0.00 & 0.09 & 0.21 & 0.44 & 0.72 & 1.39 & 2.69 & 4.70 & 7.06 & 8.73 \\
 & 250 & 0.00 & 0.00 & 0.00 & 0.00 & 0.00 & 0.00 & 0.00 & 0.00 & 0.00 & 0.00 \\
 & 4000 & 0.00 & 0.00 & 0.00 & 0.00 & 0.00 & 0.00 & 0.00 & 0.00 & 0.00 & 0.00 \\
  \hline
 & 30 & 1.00 & 0.88 & 0.74 & 0.63 & 0.58 & 0.59 & 0.61 & 0.68 & 0.71 & 0.70 \\
\textbf{NG1} & 50 & 1.00 & 0.90 & 0.73 & 0.58 & 0.52 & 0.46 & 0.47 & 0.46 & 0.48 & 0.56 \\
 & 250 & 1.00 & 1.00 & 0.87 & 0.63 & 0.45 & 0.31 & 0.27 & 0.22 & 0.19 & 0.16 \\
 & 4000 & 1.00 & 1.00 & 1.00 & 0.94 & 0.33 & 0.04 & 0.00 & 0.00 & 0.00 & 0.00 \\
  \hline
 & 30 & 6.72 & 12.62 & 13.46 & 13.12 & 12.66 & 11.93 & 11.21 & 10.48 & 9.54 & 8.63 \\
\textbf{NG2} & 50 & 5.80 & 12.85 & 13.65 & 13.29 & 12.84 & 12.05 & 11.32 & 10.53 & 9.67 & 8.80 \\
 & 250 & 2.93 & 13.24 & 13.76 & 13.43 & 12.93 & 12.22 & 11.53 & 10.70 & 9.92 & 9.01 \\
 & 4000 & 0.75 & 13.26 & 13.84 & 13.51 & 12.96 & 12.12 & 11.48 & 10.85 & 9.98 & 9.00 \\
  \hline
 & 30 & 0.02 & 0.34 & 0.33 & 0.38 & 0.39 & 0.46 & 0.71 & 0.84 & 1.26 & 1.79 \\
\textbf{GCV} & 50 & 0.00 & 0.09 & 0.10 & 0.10 & 0.10 & 0.10 & 0.13 & 0.18 & 0.27 & 0.36 \\
 & 250 & 0.00 & 0.00 & 0.00 & 0.00 & 0.00 & 0.00 & 0.00 & 0.00 & 0.00 & 0.00 \\
 & 4000 & 0.00 & 0.00 & 0.00 & 0.00 & 0.00 & 0.00 & 0.00 & 0.00 & 0.00 & 0.00 \\
  \hline
\end{tabular}

\vspace{2mm}

\begin{tabular}{|r|r|rrrrrrrrrr|}
\hline
& & \multicolumn{10}{|c|}{True Number of PCs}
\\
  \hline
 & $n$ & 10 & 11 & 12 & 13 & 14 & 15 & 16 & 17 & 18 & 19 \\
  \hline
 & 30 & 3.74 & 5.33 & 7.69 & 9.93 & 11.84 & 13.53 & 14.90 & 16.30 & 17.56 & 18.66 \\
\textbf{sBIC} & 50 & 1.84 & 3.84 & 7.26 & 10.19 & 12.48 & 14.17 & 15.41 & 16.68 & 17.81 & 18.90 \\
 & 250 & 0.02 & 0.02 & 0.01 & 0.02 & 0.00 & 0.00 & 1.66 & 16.85 & 17.97 & 19.00 \\
 & 4000 & 0.00 & 0.00 & 0.00 & 0.00 & 0.00 & 0.00 & 0.00 & 0.00 & 0.01 & 0.00 \\
  \hline
 & 30 & 9.94 & 10.98 & 11.99 & 13.00 & 14.00 & 15.00 & 16.00 & 17.00 & 18.00 & 19.00 \\
\textbf{BIC} & 50 & 9.94 & 10.98 & 11.99 & 13.00 & 14.00 & 15.00 & 16.00 & 17.00 & 18.00 & 19.00 \\
 & 250 & 0.00 & 0.00 & 0.00 & 0.00 & 0.00 & 4.86 & 16.00 & 17.00 & 18.00 & 19.00 \\
 & 4000 & 0.00 & 0.00 & 0.00 & 0.00 & 0.00 & 0.00 & 0.00 & 0.00 & 0.00 & 0.00 \\
  \hline
 & 30 & 0.71 & 0.78 & 0.81 & 0.83 & 0.86 & 0.92 & 0.97 & 0.94 & 0.85 & 0.56 \\
\textbf{NG1} & 50 & 0.57 & 0.62 & 0.61 & 0.60 & 0.66 & 0.72 & 0.71 & 0.72 & 0.73 & 0.48 \\
 & 250 & 0.17 & 0.17 & 0.18 & 0.19 & 0.21 & 0.23 & 0.25 & 0.26 & 0.29 & 0.20 \\
 & 4000 & 0.00 & 0.00 & 0.00 & 0.00 & 0.00 & 0.00 & 0.00 & 0.00 & 0.00 & 0.00 \\
  \hline
 & 30 & 7.77 & 6.93 & 6.01 & 5.09 & 4.19 & 3.22 & 2.31 & 1.47 & 0.84 & 0.56 \\
\textbf{NG2} & 50 & 7.98 & 7.05 & 6.13 & 5.19 & 4.31 & 3.36 & 2.44 & 1.51 & 0.74 & 0.48 \\
 & 250 & 8.15 & 7.25 & 6.37 & 5.46 & 4.54 & 3.60 & 2.66 & 1.71 & 0.78 & 0.20 \\
 & 4000 & 8.02 & 7.07 & 6.23 & 5.43 & 4.69 & 3.86 & 2.96 & 1.99 & 1.00 & 0.00 \\
  \hline
 & 30 & 2.42 & 3.55 & 5.68 & 8.37 & 11.30 & 13.59 & 15.28 & 16.76 & 17.91 & 18.94 \\
\textbf{GCV} & 50 & 0.51 & 1.00 & 1.84 & 4.74 & 10.15 & 14.07 & 15.72 & 16.92 & 17.96 & 18.99 \\
 & 250 & 0.00 & 0.00 & 0.00 & 0.00 & 0.31 & 13.11 & 16.00 & 17.00 & 18.00 & 19.00 \\
 & 4000 & 0.00 & 0.00 & 0.00 & 0.00 & 0.00 & 9.13 & 16.00 & 17.00 & 18.00 & 19.00 \\
  \hline
\end{tabular}
}
\end{table}

\begin{figure}[H]
    \centering
\begin{tikzpicture}
\begin{groupplot}[
  group style={group size=2 by 2, horizontal sep=1.3cm, vertical sep=1.5cm,
               xlabels at=edge bottom, ylabels at=edge left},
  width=0.505\linewidth, height=0.40\linewidth,
  xmin=-0.6, xmax=19.6, ymin=-0.76, ymax=19.76,
  xtick={0,5,10,15,19}, ytick={0,5,10,15},
  yticklabel style={/pgf/number format/fixed, /pgf/number format/precision=0},
  label style={font=\small\sffamily\sansmath},
  tick label style={font=\footnotesize\sffamily\sansmath},
  title style={font=\small\sffamily\bfseries, yshift=-2pt},
  grid=major, grid style={line width=.2pt, draw=gray!22},
  legend style={font=\footnotesize\sffamily, draw=none, fill=none, legend columns=5,
                /tikz/every even column/.append style={column sep=8pt}},
  xlabel={True Number of PCs $r$}, ylabel={Average Distance},
]
\nextgroupplot[title={n = 30}, legend to name=leg:DistIsoP20]
\addplot[color=sbiccol, mark=*, mark size=1.1pt, thick, dashed] coordinates {(0,0.23) (1,0.23) (2,0.26) (3,0.28) (4,0.35) (5,0.48) (6,0.78) (7,1.11) (8,1.76) (9,2.65) (10,3.74) (11,5.33) (12,7.69) (13,9.93) (14,11.84) (15,13.53) (16,14.9) (17,16.3) (18,17.56) (19,18.66)};
\addplot[color=biccol, mark=square*, mark size=1.1pt, thick, dashed] coordinates {(0,0) (1,0.33) (2,0.79) (3,1.41) (4,2.33) (5,3.58) (6,4.99) (7,6.38) (8,7.68) (9,8.87) (10,9.94) (11,10.98) (12,11.99) (13,13) (14,14) (15,15) (16,16) (17,17) (18,18) (19,19)};
\addplot[color=ngonecol, mark=triangle*, mark size=1.1pt, thick, dashed] coordinates {(0,1) (1,0.88) (2,0.74) (3,0.63) (4,0.58) (5,0.59) (6,0.61) (7,0.68) (8,0.71) (9,0.7) (10,0.71) (11,0.78) (12,0.81) (13,0.83) (14,0.86) (15,0.92) (16,0.97) (17,0.94) (18,0.85) (19,0.56)};
\addplot[color=ngtwocol, mark=diamond*, mark size=1.1pt, thick, dashed] coordinates {(0,6.72) (1,12.62) (2,13.46) (3,13.12) (4,12.66) (5,11.93) (6,11.21) (7,10.48) (8,9.54) (9,8.63) (10,7.77) (11,6.93) (12,6.01) (13,5.09) (14,4.19) (15,3.22) (16,2.31) (17,1.47) (18,0.84) (19,0.56)};
\addplot[color=gcvcol, mark=pentagon*, mark size=1.1pt, thick, dashed] coordinates {(0,0.02) (1,0.34) (2,0.33) (3,0.38) (4,0.39) (5,0.46) (6,0.71) (7,0.84) (8,1.26) (9,1.79) (10,2.42) (11,3.55) (12,5.68) (13,8.37) (14,11.3) (15,13.59) (16,15.28) (17,16.76) (18,17.91) (19,18.94)};
\legend{sBIC,BIC,NG1,NG2,GCV}

\nextgroupplot[title={n = 50}]
\addplot[color=sbiccol, mark=*, mark size=1.1pt, thick, dashed] coordinates {(0,0.06) (1,0.06) (2,0.07) (3,0.07) (4,0.06) (5,0.1) (6,0.19) (7,0.34) (8,0.61) (9,1.03) (10,1.84) (11,3.84) (12,7.26) (13,10.19) (14,12.48) (15,14.17) (16,15.41) (17,16.68) (18,17.81) (19,18.9)};
\addplot[color=biccol, mark=square*, mark size=1.1pt, thick, dashed] coordinates {(0,0) (1,0.09) (2,0.21) (3,0.44) (4,0.72) (5,1.39) (6,2.69) (7,4.7) (8,7.06) (9,8.73) (10,9.94) (11,10.98) (12,11.99) (13,13) (14,14) (15,15) (16,16) (17,17) (18,18) (19,19)};
\addplot[color=ngonecol, mark=triangle*, mark size=1.1pt, thick, dashed] coordinates {(0,1) (1,0.9) (2,0.73) (3,0.58) (4,0.52) (5,0.46) (6,0.47) (7,0.46) (8,0.48) (9,0.56) (10,0.57) (11,0.62) (12,0.61) (13,0.6) (14,0.66) (15,0.72) (16,0.71) (17,0.72) (18,0.73) (19,0.48)};
\addplot[color=ngtwocol, mark=diamond*, mark size=1.1pt, thick, dashed] coordinates {(0,5.8) (1,12.85) (2,13.65) (3,13.29) (4,12.84) (5,12.05) (6,11.32) (7,10.53) (8,9.67) (9,8.8) (10,7.98) (11,7.05) (12,6.13) (13,5.19) (14,4.31) (15,3.36) (16,2.44) (17,1.51) (18,0.74) (19,0.48)};
\addplot[color=gcvcol, mark=pentagon*, mark size=1.1pt, thick, dashed] coordinates {(0,0) (1,0.09) (2,0.1) (3,0.1) (4,0.1) (5,0.1) (6,0.13) (7,0.18) (8,0.27) (9,0.36) (10,0.51) (11,1) (12,1.84) (13,4.74) (14,10.15) (15,14.07) (16,15.72) (17,16.92) (18,17.96) (19,18.99)};

\nextgroupplot[title={n = 250}]
\addplot[color=sbiccol, mark=*, mark size=1.1pt, thick, dashed] coordinates {(0,0) (1,0) (2,0) (3,0.01) (4,0) (5,0.01) (6,0) (7,0) (8,0.01) (9,0.01) (10,0.02) (11,0.02) (12,0.01) (13,0.02) (14,0) (15,0) (16,1.66) (17,16.85) (18,17.97) (19,19)};
\addplot[color=biccol, mark=square*, mark size=1.1pt, thick, dashed] coordinates {(0,0) (1,0) (2,0) (3,0) (4,0) (5,0) (6,0) (7,0) (8,0) (9,0) (10,0) (11,0) (12,0) (13,0) (14,0) (15,4.86) (16,16) (17,17) (18,18) (19,19)};
\addplot[color=ngonecol, mark=triangle*, mark size=1.1pt, thick, dashed] coordinates {(0,1) (1,1) (2,0.87) (3,0.63) (4,0.45) (5,0.31) (6,0.27) (7,0.22) (8,0.19) (9,0.16) (10,0.17) (11,0.17) (12,0.18) (13,0.19) (14,0.21) (15,0.23) (16,0.25) (17,0.26) (18,0.29) (19,0.2)};
\addplot[color=ngtwocol, mark=diamond*, mark size=1.1pt, thick, dashed] coordinates {(0,2.93) (1,13.24) (2,13.76) (3,13.43) (4,12.93) (5,12.22) (6,11.53) (7,10.7) (8,9.92) (9,9.01) (10,8.15) (11,7.25) (12,6.37) (13,5.46) (14,4.54) (15,3.6) (16,2.66) (17,1.71) (18,0.78) (19,0.2)};
\addplot[color=gcvcol, mark=pentagon*, mark size=1.1pt, thick, dashed] coordinates {(0,0) (1,0) (2,0) (3,0) (4,0) (5,0) (6,0) (7,0) (8,0) (9,0) (10,0) (11,0) (12,0) (13,0) (14,0.31) (15,13.11) (16,16) (17,17) (18,18) (19,19)};

\nextgroupplot[title={n = 4000}]
\addplot[color=sbiccol, mark=*, mark size=1.1pt, thick, dashed] coordinates {(0,0) (1,0) (2,0) (3,0) (4,0) (5,0) (6,0) (7,0) (8,0) (9,0) (10,0) (11,0) (12,0) (13,0) (14,0) (15,0) (16,0) (17,0) (18,0.01) (19,0)};
\addplot[color=biccol, mark=square*, mark size=1.1pt, thick, dashed] coordinates {(0,0) (1,0) (2,0) (3,0) (4,0) (5,0) (6,0) (7,0) (8,0) (9,0) (10,0) (11,0) (12,0) (13,0) (14,0) (15,0) (16,0) (17,0) (18,0) (19,0)};
\addplot[color=ngonecol, mark=triangle*, mark size=1.1pt, thick, dashed] coordinates {(0,1) (1,1) (2,1) (3,0.94) (4,0.33) (5,0.04) (6,0) (7,0) (8,0) (9,0) (10,0) (11,0) (12,0) (13,0) (14,0) (15,0) (16,0) (17,0) (18,0) (19,0)};
\addplot[color=ngtwocol, mark=diamond*, mark size=1.1pt, thick, dashed] coordinates {(0,0.75) (1,13.26) (2,13.84) (3,13.51) (4,12.96) (5,12.12) (6,11.48) (7,10.85) (8,9.98) (9,9) (10,8.02) (11,7.07) (12,6.23) (13,5.43) (14,4.69) (15,3.86) (16,2.96) (17,1.99) (18,1) (19,0)};
\addplot[color=gcvcol, mark=pentagon*, mark size=1.1pt, thick, dashed] coordinates {(0,0) (1,0) (2,0) (3,0) (4,0) (5,0) (6,0) (7,0) (8,0) (9,0) (10,0) (11,0) (12,0) (13,0) (14,0) (15,9.13) (16,16) (17,17) (18,18) (19,19)};
\end{groupplot}
\end{tikzpicture}

\vspace{1mm}
\ref*{leg:DistIsoP20}

    \caption{Average absolute distance between the estimated and the true number of PCs, for $\bs{\Sigma}_0 = \diag(5,\ldots,5,1,\ldots,1) \in \mb{R}^{20 \times 20}$. This is the data of Table \ref{tab:DistIsotropicP20}.}
    \label{fig:DistIsotropicP20}
\end{figure}

\begin{table}[H]
\centering
\caption{Average absolute distance between the estimated and the true number of PCs for various model selection procedures. The true covariance matrix is the $20 \times 20$ matrix $\bs{\Sigma}_0 = \diag(r+1,\ldots,2,1,\ldots,1)$. The sample size $n$ and true number of PCs appear in the rows and columns of the table.}
\label{tab:DistLinearP20}
\vspace{2mm}
{\small
\begin{tabular}{|r|r|rrrrrrrrrr|}
\hline
& & \multicolumn{10}{|c|}{True Number of PCs}
\\
  \hline
 & $n$ & 0 & 1 & 2 & 3 & 4 & 5 & 6 & 7 & 8 & 9 \\
  \hline
 & 30 & 0.24 & 0.71 & 1.00 & 1.12 & 1.31 & 1.51 & 1.65 & 1.78 & 2.05 & 2.30 \\
\textbf{sBIC} & 50 & 0.06 & 0.64 & 0.85 & 0.99 & 1.02 & 1.15 & 1.20 & 1.28 & 1.40 & 1.52 \\
 & 250 & 0.00 & 0.05 & 0.04 & 0.04 & 0.01 & 0.02 & 0.02 & 0.01 & 0.01 & 0.01 \\
 & 4000 & 0.00 & 0.00 & 0.00 & 0.00 & 0.00 & 0.00 & 0.00 & 0.00 & 0.00 & 0.00 \\
  \hline
 & 30 & 0.00 & 1.00 & 1.93 & 2.72 & 3.31 & 3.90 & 4.49 & 5.06 & 5.87 & 6.87 \\
\textbf{BIC} & 50 & 0.00 & 1.00 & 1.88 & 2.42 & 2.78 & 3.00 & 3.16 & 3.34 & 3.57 & 3.79 \\
 & 250 & 0.00 & 0.82 & 0.80 & 0.75 & 0.69 & 0.61 & 0.55 & 0.46 & 0.38 & 0.30 \\
 & 4000 & 0.00 & 0.00 & 0.00 & 0.00 & 0.00 & 0.00 & 0.00 & 0.00 & 0.00 & 0.00 \\
  \hline
 & 30 & 1.00 & 0.72 & 0.69 & 0.65 & 0.65 & 0.62 & 0.64 & 0.62 & 0.66 & 0.72 \\
\textbf{NG1} & 50 & 1.00 & 0.61 & 0.57 & 0.56 & 0.49 & 0.46 & 0.47 & 0.52 & 0.50 & 0.51 \\
 & 250 & 1.00 & 0.67 & 0.51 & 0.43 & 0.31 & 0.25 & 0.18 & 0.13 & 0.11 & 0.11 \\
 & 4000 & 1.00 & 0.95 & 0.52 & 0.24 & 0.03 & 0.00 & 0.00 & 0.00 & 0.00 & 0.00 \\
  \hline
 & 30 & 6.41 & 8.71 & 11.29 & 12.09 & 12.09 & 11.61 & 10.89 & 10.16 & 9.38 & 8.38 \\
\textbf{NG2} & 50 & 5.85 & 8.64 & 11.38 & 12.45 & 12.28 & 11.67 & 11.05 & 10.40 & 9.55 & 8.64 \\
 & 250 & 3.16 & 8.72 & 11.90 & 12.70 & 12.51 & 11.97 & 11.22 & 10.46 & 9.63 & 8.79 \\
 & 4000 & 0.86 & 8.51 & 11.90 & 12.70 & 12.49 & 11.96 & 11.10 & 10.43 & 9.79 & 8.96 \\
  \hline
 & 30 & 0.02 & 0.98 & 1.41 & 1.31 & 1.20 & 1.35 & 1.46 & 1.51 & 1.71 & 1.78 \\
\textbf{GCV} & 50 & 0.00 & 0.90 & 1.01 & 0.87 & 0.87 & 0.91 & 0.94 & 1.00 & 1.09 & 1.14 \\
 & 250 & 0.00 & 0.79 & 0.68 & 0.68 & 0.69 & 0.69 & 0.71 & 0.72 & 0.73 & 0.75 \\
 & 4000 & 0.00 & 0.88 & 0.89 & 0.91 & 0.92 & 0.92 & 0.93 & 0.95 & 0.97 & 0.99 \\
  \hline
\end{tabular}

\vspace{2mm}

\begin{tabular}{|r|r|rrrrrrrrrr|}
\hline
& & \multicolumn{10}{|c|}{True Number of PCs}
\\
  \hline
 & $n$ & 10 & 11 & 12 & 13 & 14 & 15 & 16 & 17 & 18 & 19 \\
  \hline
 & 30 & 2.66 & 3.10 & 3.65 & 4.36 & 5.67 & 7.51 & 9.67 & 12.06 & 14.26 & 16.02 \\
\textbf{sBIC} & 50 & 1.74 & 2.04 & 2.40 & 2.87 & 3.89 & 5.71 & 8.48 & 11.78 & 14.16 & 16.16 \\
 & 250 & 0.01 & 0.01 & 0.01 & 0.01 & 0.02 & 0.02 & 0.03 & 0.04 & 0.04 & 0.40 \\
 & 4000 & 0.00 & 0.00 & 0.00 & 0.00 & 0.00 & 0.00 & 0.00 & 0.01 & 0.01 & 0.00 \\
  \hline
 & 30 & 7.94 & 9.09 & 10.76 & 12.26 & 13.62 & 14.82 & 15.88 & 16.96 & 17.97 & 18.98 \\
\textbf{BIC} & 50 & 4.19 & 5.34 & 7.15 & 9.75 & 12.77 & 14.62 & 15.85 & 16.91 & 17.95 & 18.98 \\
 & 250 & 0.27 & 0.18 & 0.15 & 0.09 & 0.08 & 0.04 & 0.03 & 0.02 & 0.01 & 0.01 \\
 & 4000 & 0.00 & 0.00 & 0.00 & 0.00 & 0.00 & 0.00 & 0.00 & 0.00 & 0.00 & 0.00 \\
  \hline
 & 30 & 0.71 & 0.78 & 0.82 & 0.85 & 0.90 & 0.92 & 0.92 & 0.97 & 0.86 & 0.58 \\
\textbf{NG1} & 50 & 0.53 & 0.61 & 0.60 & 0.62 & 0.70 & 0.71 & 0.72 & 0.74 & 0.73 & 0.48 \\
 & 250 & 0.13 & 0.15 & 0.18 & 0.21 & 0.22 & 0.23 & 0.26 & 0.29 & 0.28 & 0.25 \\
 & 4000 & 0.00 & 0.00 & 0.00 & 0.00 & 0.00 & 0.00 & 0.00 & 0.00 & 0.00 & 0.00 \\
  \hline
 & 30 & 7.70 & 6.81 & 5.90 & 5.01 & 4.06 & 3.15 & 2.27 & 1.43 & 0.78 & 0.58 \\
\textbf{NG2} & 50 & 7.78 & 6.90 & 6.02 & 5.16 & 4.20 & 3.24 & 2.38 & 1.47 & 0.72 & 0.48 \\
 & 250 & 7.92 & 7.07 & 6.17 & 5.30 & 4.39 & 3.50 & 2.58 & 1.63 & 0.72 & 0.25 \\
 & 4000 & 8.00 & 7.00 & 6.03 & 5.08 & 4.27 & 3.51 & 2.77 & 1.93 & 0.98 & 0.00 \\
  \hline
 & 30 & 1.98 & 2.24 & 2.50 & 2.73 & 3.30 & 4.22 & 5.80 & 8.40 & 11.67 & 15.10 \\
\textbf{GCV} & 50 & 1.25 & 1.29 & 1.43 & 1.51 & 1.67 & 2.09 & 2.93 & 5.17 & 10.00 & 14.53 \\
 & 250 & 0.77 & 0.80 & 0.82 & 0.86 & 0.90 & 0.96 & 1.11 & 1.46 & 4.94 & 15.51 \\
 & 4000 & 0.98 & 0.99 & 0.99 & 1.00 & 1.00 & 1.00 & 1.00 & 1.01 & 2.31 & 18.74 \\
  \hline
\end{tabular}
}
\end{table}

\begin{figure}[H]
    \centering
\begin{tikzpicture}
\begin{groupplot}[
  group style={group size=2 by 2, horizontal sep=1.3cm, vertical sep=1.5cm,
               xlabels at=edge bottom, ylabels at=edge left},
  width=0.505\linewidth, height=0.40\linewidth,
  xmin=-0.6, xmax=19.6, ymin=-0.76, ymax=19.76,
  xtick={0,5,10,15,19}, ytick={0,5,10,15},
  yticklabel style={/pgf/number format/fixed, /pgf/number format/precision=0},
  label style={font=\small\sffamily\sansmath},
  tick label style={font=\footnotesize\sffamily\sansmath},
  title style={font=\small\sffamily\bfseries, yshift=-2pt},
  grid=major, grid style={line width=.2pt, draw=gray!22},
  legend style={font=\footnotesize\sffamily, draw=none, fill=none, legend columns=5,
                /tikz/every even column/.append style={column sep=8pt}},
  xlabel={True Number of PCs $r$}, ylabel={Average Distance},
]
\nextgroupplot[title={n = 30}, legend to name=leg:DistLinP20]
\addplot[color=sbiccol, mark=*, mark size=1.1pt, thick, dashed] coordinates {(0,0.24) (1,0.71) (2,1) (3,1.12) (4,1.31) (5,1.51) (6,1.65) (7,1.78) (8,2.05) (9,2.3) (10,2.66) (11,3.1) (12,3.65) (13,4.36) (14,5.67) (15,7.51) (16,9.67) (17,12.06) (18,14.26) (19,16.02)};
\addplot[color=biccol, mark=square*, mark size=1.1pt, thick, dashed] coordinates {(0,0) (1,1) (2,1.93) (3,2.72) (4,3.31) (5,3.9) (6,4.49) (7,5.06) (8,5.87) (9,6.87) (10,7.94) (11,9.09) (12,10.76) (13,12.26) (14,13.62) (15,14.82) (16,15.88) (17,16.96) (18,17.97) (19,18.98)};
\addplot[color=ngonecol, mark=triangle*, mark size=1.1pt, thick, dashed] coordinates {(0,1) (1,0.72) (2,0.69) (3,0.65) (4,0.65) (5,0.62) (6,0.64) (7,0.62) (8,0.66) (9,0.72) (10,0.71) (11,0.78) (12,0.82) (13,0.85) (14,0.9) (15,0.92) (16,0.92) (17,0.97) (18,0.86) (19,0.58)};
\addplot[color=ngtwocol, mark=diamond*, mark size=1.1pt, thick, dashed] coordinates {(0,6.41) (1,8.71) (2,11.29) (3,12.09) (4,12.09) (5,11.61) (6,10.89) (7,10.16) (8,9.38) (9,8.38) (10,7.7) (11,6.81) (12,5.9) (13,5.01) (14,4.06) (15,3.15) (16,2.27) (17,1.43) (18,0.78) (19,0.58)};
\addplot[color=gcvcol, mark=pentagon*, mark size=1.1pt, thick, dashed] coordinates {(0,0.02) (1,0.98) (2,1.41) (3,1.31) (4,1.2) (5,1.35) (6,1.46) (7,1.51) (8,1.71) (9,1.78) (10,1.98) (11,2.24) (12,2.5) (13,2.73) (14,3.3) (15,4.22) (16,5.8) (17,8.4) (18,11.67) (19,15.1)};
\legend{sBIC,BIC,NG1,NG2,GCV}

\nextgroupplot[title={n = 50}]
\addplot[color=sbiccol, mark=*, mark size=1.1pt, thick, dashed] coordinates {(0,0.06) (1,0.64) (2,0.85) (3,0.99) (4,1.02) (5,1.15) (6,1.2) (7,1.28) (8,1.4) (9,1.52) (10,1.74) (11,2.04) (12,2.4) (13,2.87) (14,3.89) (15,5.71) (16,8.48) (17,11.78) (18,14.16) (19,16.16)};
\addplot[color=biccol, mark=square*, mark size=1.1pt, thick, dashed] coordinates {(0,0) (1,1) (2,1.88) (3,2.42) (4,2.78) (5,3) (6,3.16) (7,3.34) (8,3.57) (9,3.79) (10,4.19) (11,5.34) (12,7.15) (13,9.75) (14,12.77) (15,14.62) (16,15.85) (17,16.91) (18,17.95) (19,18.98)};
\addplot[color=ngonecol, mark=triangle*, mark size=1.1pt, thick, dashed] coordinates {(0,1) (1,0.61) (2,0.57) (3,0.56) (4,0.49) (5,0.46) (6,0.47) (7,0.52) (8,0.5) (9,0.51) (10,0.53) (11,0.61) (12,0.6) (13,0.62) (14,0.7) (15,0.71) (16,0.72) (17,0.74) (18,0.73) (19,0.48)};
\addplot[color=ngtwocol, mark=diamond*, mark size=1.1pt, thick, dashed] coordinates {(0,5.85) (1,8.64) (2,11.38) (3,12.45) (4,12.28) (5,11.67) (6,11.05) (7,10.4) (8,9.55) (9,8.64) (10,7.78) (11,6.9) (12,6.02) (13,5.16) (14,4.2) (15,3.24) (16,2.38) (17,1.47) (18,0.72) (19,0.48)};
\addplot[color=gcvcol, mark=pentagon*, mark size=1.1pt, thick, dashed] coordinates {(0,0) (1,0.9) (2,1.01) (3,0.87) (4,0.87) (5,0.91) (6,0.94) (7,1) (8,1.09) (9,1.14) (10,1.25) (11,1.29) (12,1.43) (13,1.51) (14,1.67) (15,2.09) (16,2.93) (17,5.17) (18,10) (19,14.53)};

\nextgroupplot[title={n = 250}]
\addplot[color=sbiccol, mark=*, mark size=1.1pt, thick, dashed] coordinates {(0,0) (1,0.05) (2,0.04) (3,0.04) (4,0.01) (5,0.02) (6,0.02) (7,0.01) (8,0.01) (9,0.01) (10,0.01) (11,0.01) (12,0.01) (13,0.01) (14,0.02) (15,0.02) (16,0.03) (17,0.04) (18,0.04) (19,0.4)};
\addplot[color=biccol, mark=square*, mark size=1.1pt, thick, dashed] coordinates {(0,0) (1,0.82) (2,0.8) (3,0.75) (4,0.69) (5,0.61) (6,0.55) (7,0.46) (8,0.38) (9,0.3) (10,0.27) (11,0.18) (12,0.15) (13,0.09) (14,0.08) (15,0.04) (16,0.03) (17,0.02) (18,0.01) (19,0.01)};
\addplot[color=ngonecol, mark=triangle*, mark size=1.1pt, thick, dashed] coordinates {(0,1) (1,0.67) (2,0.51) (3,0.43) (4,0.31) (5,0.25) (6,0.18) (7,0.13) (8,0.11) (9,0.11) (10,0.13) (11,0.15) (12,0.18) (13,0.21) (14,0.22) (15,0.23) (16,0.26) (17,0.29) (18,0.28) (19,0.25)};
\addplot[color=ngtwocol, mark=diamond*, mark size=1.1pt, thick, dashed] coordinates {(0,3.16) (1,8.72) (2,11.9) (3,12.7) (4,12.51) (5,11.97) (6,11.22) (7,10.46) (8,9.63) (9,8.79) (10,7.92) (11,7.07) (12,6.17) (13,5.3) (14,4.39) (15,3.5) (16,2.58) (17,1.63) (18,0.72) (19,0.25)};
\addplot[color=gcvcol, mark=pentagon*, mark size=1.1pt, thick, dashed] coordinates {(0,0) (1,0.79) (2,0.68) (3,0.68) (4,0.69) (5,0.69) (6,0.71) (7,0.72) (8,0.73) (9,0.75) (10,0.77) (11,0.8) (12,0.82) (13,0.86) (14,0.9) (15,0.96) (16,1.11) (17,1.46) (18,4.94) (19,15.51)};

\nextgroupplot[title={n = 4000}]
\addplot[color=sbiccol, mark=*, mark size=1.1pt, thick, dashed] coordinates {(0,0) (1,0) (2,0) (3,0) (4,0) (5,0) (6,0) (7,0) (8,0) (9,0) (10,0) (11,0) (12,0) (13,0) (14,0) (15,0) (16,0) (17,0.01) (18,0.01) (19,0)};
\addplot[color=biccol, mark=square*, mark size=1.1pt, thick, dashed] coordinates {(0,0) (1,0) (2,0) (3,0) (4,0) (5,0) (6,0) (7,0) (8,0) (9,0) (10,0) (11,0) (12,0) (13,0) (14,0) (15,0) (16,0) (17,0) (18,0) (19,0)};
\addplot[color=ngonecol, mark=triangle*, mark size=1.1pt, thick, dashed] coordinates {(0,1) (1,0.95) (2,0.52) (3,0.24) (4,0.03) (5,0) (6,0) (7,0) (8,0) (9,0) (10,0) (11,0) (12,0) (13,0) (14,0) (15,0) (16,0) (17,0) (18,0) (19,0)};
\addplot[color=ngtwocol, mark=diamond*, mark size=1.1pt, thick, dashed] coordinates {(0,0.86) (1,8.51) (2,11.9) (3,12.7) (4,12.49) (5,11.96) (6,11.1) (7,10.43) (8,9.79) (9,8.96) (10,8) (11,7) (12,6.03) (13,5.08) (14,4.27) (15,3.51) (16,2.77) (17,1.93) (18,0.98) (19,0)};
\addplot[color=gcvcol, mark=pentagon*, mark size=1.1pt, thick, dashed] coordinates {(0,0) (1,0.88) (2,0.89) (3,0.91) (4,0.92) (5,0.92) (6,0.93) (7,0.95) (8,0.97) (9,0.99) (10,0.98) (11,0.99) (12,0.99) (13,1) (14,1) (15,1) (16,1) (17,1.01) (18,2.31) (19,18.74)};
\end{groupplot}
\end{tikzpicture}

\vspace{1mm}
\ref*{leg:DistLinP20}

    \caption{Average absolute distance between the estimated and the true number of PCs, for $\bs{\Sigma}_0 = \diag(r+1,\ldots,2,1,\ldots,1) \in \mb{R}^{20 \times 20}$. This is the data of Table \ref{tab:DistLinearP20}.}
    \label{fig:DistLinearP20}
\end{figure}

\end{document}